\documentclass[11pt,french]{smfart}
\usepackage[latin1]{inputenc}
\usepackage[T1]{fontenc}
\usepackage[english,french]{}
\usepackage{hyperref}
\usepackage{fancyhdr}
\usepackage{layouts}
\usepackage{layout}
\usepackage{url}
\def\lien{\mathrel{\mkern-4mu}}
\def\too{\relbar\lien\rightarrow}
\def\tooo{\relbar\lien\relbar\lien\too}

\newtheorem{theorem}{Th\'eor\`eme}[section]
\newtheorem{lemma}{Lemme}[section]
\newtheorem{proposition}{Proposition}[section]
\newtheorem{corollary}{Corollaire}[section]
\newtheorem{definition}{D\'efinition}[section]
\newtheorem{remark}{Remarque}[section]
\newtheorem{remarks}{Remarques}[section]

\newtheorem{conjecture}{Conjecture}[section]
\newtheorem{hypothesis}{Hypoth\`ese}[section]
\numberwithin{equation}{section}
\def\N{\mathbb{N}}

\def\Q{\mathbb{Q}}
\def\Z{\mathbb{Z}}

\def\plus{\displaystyle\mathop{\raise 1.0pt \hbox{$\bigoplus $}}\limits}
\def\prd{\displaystyle\mathop{\raise 2.0pt \hbox{$\prod$}}\limits}
\def\sm{\displaystyle\mathop{\raise 2.0pt \hbox{$\sum$}}\limits}
\def\lien{\mathrel{\mkern-4mu}}
\def\too{\relbar\lien\rightarrow}
\def\tooo{\relbar\lien\relbar\lien\too}
\def\toooo{\relbar\lien\relbar\lien\tooo}

\let\ds=\displaystyle
\let\wt=\widetilde
\let\ds=\displaystyle
\let\wt=\widetilde

\def\Cl{{\mathcal C}\hskip-2pt{\ell}}
\def\cl{c\hskip-1pt{\ell}}

\def\order{\raise1.5pt \hbox{${\scriptscriptstyle \#}$}}
\def\nmid{\  /\!\! \!\!\mid}

\begin{document}

\title{Approche $p$-adique \\ de la conjecture de Greenberg \\  (cas totalement r\'eel $p$-d\'ecompos\'e)}
\author{Georges Gras}

\address{Villa la Gardette, Chemin Ch\^ateau Gagni\`ere
\\ F--38520 Le Bourg d'Oisans, France -- \url{https://www.researchgate.net/profile/Georges_Gras}}
\email{g.mn.gras@wanadoo.fr}
\keywords{Greenberg's conjecture, Iwasawa's theory, $p$-class groups, class field theory,
Fermat quotients, $p$-adic regulators, Leopoldt's conjecture}
\subjclass{11R23, 11R29, 11R37; 11Y40}

\begin{abstract}
Soit $k$ un corps de nombres totalement r\'eel et soit $k_\infty$ sa $\Z_p$-extension
cyclotomique pour un premier $p>2$. Nous donnons (Th\'eor\`eme  \ref{thm2})
une condition suffisante de nullit\'e des invariants d'Iwasawa $\lambda, \mu$, 
lorsque $p$ est totalement d\'ecompos\'e dans $k$, et nous obtenons d'importantes tables
de corps quadratiques et $p$ pour lesquels on peut conclure que 
$\lambda = \mu=0$.
Nous montrons que le nombre de $p$-classes ambiges de $k_n$ ($n$-i\`eme \'etage
dans $k_\infty$) est \'egal \`a l'ordre du groupe de torsion ${\mathcal T}_k$ 
du groupe de Galois de la pro-$p$-extension Ab\'elienne $p$-ramifi\'ee maximale de $k$
(Th\'eor\`eme \ref{regulateur}), pour tout $n \geq e$, o\`u $p^e$ est l'exposant de 
$U_k^*/ \overline E_k$ (en termes d'unit\'es locales et globales).
Puis nous \'etablissons des analogues de la formule de Chevalley en utilisant 
une famille $(\Lambda_i^n)_{0\leq i \leq m_n}$ 
de sous-groupes de  $k^\times$ contenant $E_k$, dans lesquels tout $x$ 
est norme d'un id\'eal de $k_n$. Cette famille est attach\'ee \`a la filtration 
classique du $p$-groupe des classes de $k_n$ d\'efinissant
l'algorithme de calcul de son ordre en $m_n$ pas.
A partir de cela, nous montrons (Theorem \ref{O}) que 
$m_n \geq (\lambda \cdot n + \mu \cdot p^n + \nu)/{v_p(\order {\mathcal T}_k)}$ 
et que la condition $m_n = O(1)$ (i.e., $\lambda = \mu=0$) d\'epend essentiellement
des valuations ${\mathfrak p}$-adiques des $\frac{x^{p-1}-1}{p}$, 
$x \in \Lambda_i^n$, pour ${\mathfrak p} \mid p$, de sorte que la conjecture de 
Greenberg est fortement d\'ependante de ``quotients de Fermat'' dans $k^\times$. 
Des heuristiques et statistiques sur ces quotients de Fermat
(Sections \ref{quant}, \ref{nsuite}, \ref{BC}) montrent qu'ils suivent 
des lois de probabilit\'es naturelles, li\'ees \`a ${\mathcal T}_k$ quel que 
soit $n$, sugg\'erant que $\lambda = \mu=0$ (Heuristiques \ref{heur1}, 
\ref{heur2}, \ref{probas}).

\noindent
Ceci impliquerait que, pour une preuve de la conjecture de Greenberg, certains
r\'esultats $p$-adiques profonds (probablement inaccessibles actuellement),
ayant une certaine analogie avec la conjecture de Leopoldt,
sont n\'ecessaires avant toute r\'ef\'erence \`a la seule th\'eorie d'Iwasawa alg\'ebrique.
\end{abstract}

\begin{altabstract}
Let $k$ be a totally real number field ant let $k_\infty$ be its cyclotomic 
$\Z_p$-extension for a prime $p>2$. We give (Theorem \ref{thm2}) a 
sufficient condition of nullity of the Iwasawa invariants $\lambda, \mu$, 
when $p$ totally splits in $k$, and we obtain important tables of quadratic fields and 
$p$ for which we can conclude that $\lambda = \mu=0$.
We show that the number of ambiguous $p$-classes of $k_n$ ($n$th stage in $k_\infty$)
is equal to the order of the torsion group ${\mathcal T}_k$, of the Galois group of the 
maximal Abelian $p$-ramified pro-$p$-extension of $k$ (Theorem \ref{regulateur}),
for all $n \geq e$, where $p^e$ is the exponent of $U_k^*/ \overline E_k$
(in terms of local and global units). 
Then we establish analogs of Chevalley's formula using a family 
$(\Lambda_i^n)_{0\leq i \leq m_n}$ 
of subgroups of $k^\times$ containing $E_k$, in which any $x$ 
is norm of an ideal of $k_n$. This family is attached to the classical filtration 
of the $p$-class group of $k_n$ defining the algorithm of computation of its order
in $m_n$ steps.
From this, we prove (Theorem \ref{O}) that 
$m_n \geq (\lambda \cdot n + \mu \cdot p^n + \nu)/{v_p(\order {\mathcal T}_k)}$ 
and that the condition $m_n = O(1)$ (i.e., $\lambda = \mu=0$) essentially depends
on the ${\mathfrak p}$-adic valuations of the $\frac{x^{p-1}-1}{p}$, 
$x \in \Lambda_i^n$, for ${\mathfrak p} \mid p$, so that Greenberg's conjecture 
is strongly related to ``Fermat quotients'' in $k^\times$. 
Heuristics and statistical analysis of these Fermat quotients 
(Sections \ref{quant}, \ref{nsuite}, \ref{BC}) show that they follow 
natural probabilities, linked to ${\mathcal T}_k$ whatever $n$, suggesting that 
$\lambda = \mu=0$ (Heuristics \ref{heur1}, \ref{heur2}, \ref{probas}). 

\noindent
This would imply that, for a proof of Greenberg's conjecture, some deep $p$-adic 
results (probably out of reach now), having some analogy with Leopoldt's conjecture, 
are necessary before referring to the sole algebraic Iwasawa theory. 
\end{altabstract}

\maketitle

\parindent=0cm

\section{Introduction}\label{sect1}
Nous appelons {\it Conjecture de Greenberg pour les corps 
de nombres totalement r\'eels} $k$,
le fait que les invariants d'Iwasawa $\lambda_p(k)$ et $\mu_p(k)$,
associ\'es \`a la limite projective des $p$-groupes de classes d'id\'eaux dans la
$p$-tour cyclo\-tomique $k_\infty$, sont nuls (quel que soit le nombre premier $p$). 
Comme la nullit\'e de $\lambda_p(k)$ et $\mu_p(k)$ implique celle relative 
aux sous-corps de~$k$, nous supposerons $k/\Q$ Galoisienne r\'eelle.

\smallskip
Le corps $k$ et le nombre premier $p$ \'etant fix\'es, on d\'esigne par 
$\lambda$, $\mu$, $\nu$ les invariants d'Iwasawa pour $k$ et $p$.

\smallskip
Dans l'approche classique, la d\'ecomposition de $p$ dans $k/\Q$ joue un r\^ole important.
En effet, soient $d := [k : \Q]$ et $t \mid d$ le nombre 
d'id\'eaux premiers au-dessus de $p$ dans $k$~; par exemple,
si dans un premier temps on ne s'int\'eresse qu'\`a la trivialit\'e du $p$-groupe 
des classes $\Cl_{k_n}$ du $n$-i\`eme \'etage $k_n $ de $k_\infty$, celle-ci est, 
en supposant implicitement {\it $p$ totalement ramifi\'e dans $k_\infty/k$}, 
\'equivalente \`a la trivialit\'e de chacun des deux facteurs de la formule des classes 
ambiges de Chevalley qui s'\'ecrit dans ce cas~:
\begin{equation}\label{che}
\order  \Cl_{k_n}^{G_n} 
= \order  \Cl_{k} \cdot \ds\frac{p^{n \cdot (t -1)}}{(E_k : E_k \cap  {\rm N}_{k_n/k}(k_n^\times))},
\end{equation}

o\`u $G_n={\rm Gal}(k_n/k)$, $\Cl_{k}$ est le $p$-groupe des classes de $k$, 
et $E_k$ son groupe des unit\'es. Chaque facteur joue alors un r\^ole sp\'ecifique 
pour la conjecture~; en ce qui concerne le facteur normique, on a les faits suivants~:

\smallskip
(i) Le cas $t =1$ implique $E_k \subset {\rm N}_{k_n/k}(k_n^\times)$ (formule du produit des 
symboles de restes normiques) et les invariants $\lambda$ et $\mu$ d\'ependent 
essentiellement du comportement du $p$-groupe des classes de $k$ par extension des 
id\'eaux dans la tour \cite[Theorem 1, \S\,4 (1976)]{Gre}.

\smallskip
(ii) Le cas $t  = d$ montrera que, sous la conjecture de Leopoldt, le facteur 
$\ds \frac{p^{n \cdot (d -1)}}{(E_k : E_k \cap  {\rm N}_{k_n/k}(k_n^\times))}$ 
a, pour tout $n \gg 0$, m\^eme valuation $p$-adique que le r\'egulateur 
$p$-adique normalis\'e $R_k$ de $k$ (Th\'eor\`eme \ref{regulateur}).

\smallskip
(iii) Si $1 < t  < d $, l'\'etude se ram\`ene grosso modo aux cas pr\'ec\'edents, ce que l'on peut
illustrer au moyen du cas Ab\'elien r\'eel de degr\'e $d$ \'etranger \`a $p$,
par d\'ecoupage semi-simple (selon les caract\`eres $p$-adiques de 
${\rm Gal}(k/\Q)$, comme dans \cite{Gra5}, \cite{H}, \cite{IS1}, \cite{IS2}).
Si $k'$ est le corps de d\'ecom\-position de $p$ dans $k$, alors 
$\order \Cl_{k} \cdot\ds \frac{p^{n \cdot (t  -1)}}{(E_k : E_k \cap  {\rm N}_{k_n/k}(k_n^\times))}$ 
s'\'ecrit comme produit de 

$\order \Cl_{k'} \cdot \ds \frac{p^{n \cdot (t  -1)}}{(E_{k'} : E_{k'} 
\cap {\rm N}_{k'_n/k'}(k'^{\times}_n))}\ $
par
$\ \order \Cl_{k}^* \cdot \ds \frac{1}{(E^*_{k} : E^*_{k} \cap  {\rm N}_{k_n/k}(k_n^\times))} =
 \order \Cl_{k}^*$,

\smallskip
o\`u l'on a pos\'e $E_k = E_{k'} \plus E^*_{k}$ (\`a un indice pr\`es \'etranger \`a $p$), o\`u
$E^*_{k}$ est le sous-groupe des unit\'es de norme 1 sur $k'$. 
Autrement dit,  pour tout $n$ on a $E^*_{k} \subset {\rm N}_{k_n/k}(k_n^\times)$
et on est ramen\'e au cas totalement d\'ecompos\'e pour~$k'$. 
Cependant, s'il existe un $p$-groupe de classes relatives $\Cl_{k}^*$ dans $k/k'$, 
la question est analogue, en relatif, au cas non d\'ecompos\'e $t =1$.

\section{Conjecture de Greenberg ({$p$} d\'ecompos\'e)} \label{sect2}

Soit $k$ un corps de nombres Galoisien r\'eel de degr\'e $d$ et
soit $p>2$ un nombre premier totalement d\'ecompos\'e dans $k$.
Soit $\Q_\infty$ la $\Z_p$-extension cyclotomique de $\Q$ et $k_\infty := k\, \Q_\infty$
celle de $k$~; puisque $p$ est non ramifi\'e dans $k/\Q$, on a $k \cap \Q_\infty = \Q$ et
totale ramification de $p$ dans $k_\infty/k$.

\medskip
On d\'esigne par $K=k_n \subset k_\infty$ l'extension de degr\'e $p^n$ de $k$ et par
$G=G_n$ son groupe de Galois.
Soient $\Cl_k$ et $\Cl_{K}$ (resp.  $\Cl_k^{S_k}$ et $\Cl_{K}^{S_{K}}$) les $p$-groupes 
des classes de $k$ et $K$ (resp. les $S_k$ et $S_{K}$-groupes des classes de 
$k$ et ${K}$) o\`u $S_k$  et $S_{K}$ sont les ensembles des $p$-places de $k$ 
et ${K}$ (on a $\order S_k = \order S_{K} = d$). 
On a $\Cl_k^{S_k} := \Cl_k/ \cl_k(S_k)$ o\`u $\cl_k(S_k)$ 
est le sous-groupe de $\Cl_k$ engendr\'e par les $p$-classes des \'el\'ements de $S_k$ 
(idem pour $\Cl_{K}^{S_{K}}$).

\begin{theorem} \label{green}\cite[Theorem 2, \S\,4 (1976)]{Gre}. Sous les hypoth\`eses et 
notations pr\'ec\'edentes, et sous la conjecture de 
Leopoldt pour $p$ dans $k$, les invariants $\lambda_p(k)$ et $\mu_p(k)$ d'Iwasawa 
sont nuls si et seulement si pour tout $n \gg 0$, on a $\Cl_{K}^{G} = \cl_{K}(S_{K})$ 
(i.e., le sous-groupe de $\Cl_{K}$ form\'e des $p$-classes ambiges est \'egal au 
sous-groupe engendr\'e par les $p$-classes des id\'eaux premiers de $K$ 
au-dessus de $p$).
\end{theorem}

Dans cet \'enonc\'e interviennent le $p$-groupe des classes ambiges
$\Cl_{K}^{G}$ et le groupe $\cl_{K}(S_{K})$ qui est un sous-groupe du
$p$-groupe des classes du groupe des id\'eaux invariants  $I_{K}^{G}$~; 
on a l'isomorphisme classique~:
\begin{equation}\label{amb}
\Cl_{K}^{G} \big / \cl_{K}(I_{K}^{G}) \simeq 
E_k \cap {\rm N}_{K/k}(K^\times) / {\rm N}_{K/k}(E_{K}), 
\end{equation}

o\`u $E_k$ et $E_{K}$ sont les groupes des unit\'es de $k$ et ${K}$. 
Or seul $E_k \cap {\rm N}_{K/k}(K^\times)$, donn\'e par le corps de classes local, 
est accessible en pratique, ${\rm N}_{K/k}(E_{K})$ \'etant un invariant arithm\'etique 
non trivial associ\'e \`a l'extension $K/k$, en relation avec les probl\`emes de capitulation
de classes d'id\'eaux, ce qui explique la difficult\'e du calcul de 
$\lambda$ et $\mu$ en termes d'unit\'es globales (\'eventuellement compar\'ees aux unit\'es 
cyclotomiques du cadre Ab\'elien comme dans \cite[(1995)]{KS} pour le cas non d\'ecompos\'e).
On a par ailleurs $\cl_{K}(I_{K}^{G}) = \cl_{K}(S_{K}) \cdot j_{K/k} (\Cl_k)$, o\`u $j_{K/k}$ 
est l'extension des classes de $k$ \`a ${K}$, et la relation $\Cl_{K}^{G} = \cl_{K}(S_{K})$ 
implique de plus $j_{K/k} (\Cl_k) \subseteq \cl_{K}(S_{K})$.

\begin{remark} \label{remgene}
Si $\Cl_{k_n}^{G_n}=1\ \forall n \gg 0$, 
la conjecture de Greenberg est vraie sous la forme 
$\lambda = \mu = \nu = 0$ et r\'ecipro\-quement~; or 
$\Cl_{k_n}^{G_n}=1\ \forall n \gg 0$ \'equivaut \`a 
$\order \Cl_k=1\,  \& \,  (E_k : E_k \cap {\rm N}_{k_n/k}(k_n^\times)) 
= p^{n \cdot (d -1)} \ \forall n \gg 0$ (cf. \eqref{che}).

La Proposition \ref{basetopo} montrera que $(E_k : E_k \cap {\rm N}_{k_n/k}(k_n^\times)) 
= p^{n \cdot(d -1)}$ pour un $n \geq 1$ \'equivaut \`a $(E_k : E_k \cap {\rm N}_{k_1/k}(k_1^\times)) 
= p^{d -1}$, et ainsi $\Cl_{k_n}^{G_n}=1\ \forall n \gg 0$ \'equivaut \`a
$\Cl_{k_1}^{G_1}=1$ qui est effectif en pratique.
Or la condition normique sur les unit\'es a lieu si et seulement si 
le r\'egulateur $p$-adique normalis\'e $R_k$ de $k$ (au sens de \cite[Section 5]{Gra7}) 
est une unit\'e $p$-adique (Th\'eor\`eme \ref{regulateur} (ii)).
Il en r\'esulte que la condition $\lambda = \mu = \nu = 0$ \'equivaut
(sous la conjecture de Leopoldt pour $p$ dans $k$) \`a la nullit\'e du groupe de torsion 
${\mathcal T}_k$ de ${\rm Gal}(H_k^{\rm pr}/k)$, o\`u $H_k^{\rm pr}$ 
est la pro-$p$-extension Ab\'elienne $p$-ramifi\'ee maximale de $k$ 
(en effet, dans ce cas, $\order {\mathcal T}_k = \order \Cl_k \cdot R_k$), 
ce qui \'equivaut \`a la $p$-rationalit\'e de $k$ (notion d\'efinie dans 
\cite[\S\,IV.3]{Gra1}, \cite{GJ}, \cite{JN}, \cite{MN}).
\end{remark}

\section{Condition suffisante de nullit\'e de $\lambda$ et $\mu$}
\label{sect3}

Nous avons d\'emontr\'e dans \cite[(1973)]{Gra2}, \cite[(1994)]{Gra3}, et repris r\'ecemment 
dans \cite[Theorem 3.6 (2016)]{Gra4}, le r\'esultat suivant valable en toute g\'en\'eralit\'e,
ici \'enonc\'e dans le cas ``$T=\emptyset$'' qui correspond aux groupes de classes au sens classique 
(l'\'enonc\'e n'utilisant que la ramification des places finies et l'\'eventuelle complexification des 
places \`a l'infini on doit utiliser le sens restreint)~; on d\'esigne par $I_K$ et $P_K$
le groupe des id\'eaux de $K$ et son sous-groupe des id\'eaux principaux 
(au sens restreint)~:

\begin{theorem} \label{thm1}
Soit $K/k$ une extension cyclique de corps de nombres, de groupe de Galois $G$. 
Soient $\Cl_{K}$ et  $\Cl_{k}$ les groupes des classes au sens restreint de
$K$ et $k$ respectivement. Soit $e_{\mathfrak q}$ l'indice de ramification dans $K/k$ d'un 
id\'eal premier ${\mathfrak q}$. 
Alors pour tout sous-$G$-module ${\mathcal H}$ de $\Cl_{K}$ 
et tout sous-groupe ${\mathcal I}$ de $I_{K}$ tel que
${\mathcal I} \cdot P_{K} / P_{K} = {\mathcal H}$, on a~:
\begin{equation*}
\order \big( \Cl_{K} /{\mathcal H}\big)^G= 
 \frac{\order  \Cl_{k} \cdot \prod_{{\mathfrak q}} e_{\mathfrak q}}
{[K : k] \cdot \order {\rm N}_{K/k}({\mathcal H})  
\cdot (\Lambda : \Lambda \cap {\rm N}_{K/k}(K^{\times}))},
\end{equation*}
o\`u ${\rm N}_{K/k}$ est la norme arithm\'etique et
$\Lambda := \{x \in k^{\times},\   (x) \in {\rm N}_{K/k}({\mathcal I}) \}$.
\end{theorem}

\begin{corollary} 
Si ${\mathcal H} = \cl_K (S_K)$, o\`u $S_K$ est un ensemble fini quelconque d'id\'eaux premiers
de $K$, on obtient~:
\begin{equation*}
\order  \Cl_K^{S_K  G} =  \frac{\order  \Cl_k \cdot 
\prod_{{\mathfrak q}} e_{\mathfrak q}}{[K : k] \cdot \order  \cl_k ( {\rm N} S_K )
\cdot (E_k^{{\rm N} S_K}  : E_k^{{\rm N} S_K} \cap {\rm N}_{K/k}(K^\times))} ,
\end{equation*}

o\`u ${\rm N} = {\rm N}_{K/k}$ et 
$E_k^{{\rm N} S_K} = \{ x \in E_k^{S_k},  \  v_{\mathfrak q}(x) 
\equiv 0  \pmod {f_{\mathfrak q}} \ \forall {\mathfrak q} \in S_k$\}, o\`u $S_k$ est 
l'ensemble des id\'eaux premiers de $k$ au-dessous de ceux de $S_K$,
$v_{\mathfrak q}$ la valuation ${\mathfrak q}$-adique, et 
$f_{\mathfrak q}$ le degr\'e r\'esiduel de ${\mathfrak q}$ dans $K/k$.
\end{corollary}

Jaulent a obtenu dans \cite[p. 177 (1986)]{J1} l'autre \'ecriture~:
\begin{equation*}
\order  \Cl_K^{S_K G} =  \frac{\order  \Cl_k^{S_k}\cdot 
\prod_{{\mathfrak q} \notin S_k} e_{\mathfrak q} \cdot \prod_{{\mathfrak q} \in S_k} 
e_{\mathfrak q}\,f_{\mathfrak q}}{[K : k] \cdot (E_k^{S_k}  : E_k^{S_k} \cap {\rm N}_{K/k}(K^\times))}.
\end{equation*}

Utiliser la relation $E_k^{S_k} \cap {\rm N}_{K/k}(K^\times) =
E_k^{{\rm N} S_K} \cap {\rm N}_{K/k}(K^\times)$ et la suite exacte
$1 \to E_k^{S_k} / E_k^{{\rm N} S_K} \too \langle S_k \rangle_\Z / \langle {\rm N}_{K/k} S_K \rangle_\Z
\too  \cl_k( S_k  )/ \cl_k(  {\rm N}_{K/k} S_K ) \to 1$
pour comparer les deux expressions. 

\begin{remarks}\label{rem1}
(i) La relation du Th\'eor\`eme \ref{thm1} (et ses analogues) se met 
sous la forme du produit de deux entiers~:

\smallskip
\centerline{$\order \big( \Cl_{K} /{\mathcal H}\big)^G= \ds
\frac{[H_k : K \cap H_k]}{\order {\rm N}_{K/k}({\mathcal H}) } \,\cdot \,
\frac{[K : K \cap H_k]^{-1} \cdot \prod_{{\mathfrak q}} e_{\mathfrak q}} 
{(\Lambda : \Lambda \cap {\rm N}_{K/k}(K^{\times}))}$,}

o\`u $H_k$ est le corps de classes de Hilbert de $k$~; si $K\cap H_k=k$, alors le premier 
facteur est \'egal \`a $\ds \frac{\order  \Cl_{k} }{\order {\rm N}_{K/k}({\mathcal H})}$ et le 
second \`a $\ds \frac{ [K : k]^{-1} \cdot \prod_{{\mathfrak q}} e_{\mathfrak q}} 
{(\Lambda : \Lambda \cap {\rm N}_{K/k}(K^{\times}))}$. 

(ii) Pour $K = k_n \subset k_\infty$, $G=G_n = {\rm Gal}(k_n/k)$, on obtient les formules~:
\centerline{$\order \big( \Cl_{K} /{\mathcal H}\big)^G=
\ds \frac{\order  \Cl_{k} }{\order {\rm N}_{K/k}({\mathcal H})} 
\cdot \frac{p^{n\cdot (d -1)}} {(\Lambda : \Lambda \cap {\rm N}_{K/k}(K^{\times}))}$,}

\centerline {$\ds \order \Cl_{K}^{S_{K}} {}^{G} = \order \Cl_k^{S_k} \cdot 
\frac{p^{n\cdot (d -1)}} {(E_k^{S_k} : E_k^{S_k}\cap {\rm N}_{K/k}(K^\times))}$.}
\end{remarks}

On peut \'enoncer, en d\'esignant maintenant par $\Cl$ les {\it $p$-groupes de classes},
la condition seulement suffisante suivante, en rapport avec les r\'esultats \'evoqu\'es
dans la Remarque \ref{remgene}~:

\begin{theorem} \label{thm2} Soit $k$ Galoisien r\'eel et soit $p>2$ totalement 
d\'ecompos\'e dans $k$~; on suppose que $p$ v\'erifie la conjecture de Leopoldt dans $k$. 
Soit $k_1$, de degr\'e $p$ sur $k$, le premier \'etage de la $\Z_p$-extension cyclotomique de $k$.

Alors une condition suffisante pour que $\lambda = \mu = 0$ est que les deux 
conditions suivantes soient r\'ealis\'ees, o\`u $S_k$ est l'ensemble des id\'eaux 
premiers de $k$ au-dessus de $p$ et $E_k^{S_k}$ le groupe des $S_k$-unit\'es de $k$~:

\smallskip
(i) $\Cl_k^{S_k}=1$ (i.e., $S_k$ engendre le $p$-groupe des classes de $k$),

\smallskip
(ii) $(E_k^{S_k} : E_k^{S_k} \cap {\rm N}_{k_1/k}(k_1^\times)) = p^{d -1}$, o\`u $d = [k : \Q]$.
\end{theorem}

\begin{proof} Soit $n\geq1$ et consid\'erons $K:=k_n$, $G:=G_n$.
On a la suite exacte de ${G}$-modules
$1 \too \cl_{K}({S_{K}}) \tooo \Cl_{K} \tooo \Cl_{K}^{S_{K}} \too 1$,
qui conduit \`a~:
$$1 \too \cl_{K}({S_{K}})^{G} \tooo \Cl_{K}^{G} \tooo \Cl_{K}^{S_{K} {G}} 
\tooo {\rm H}^1({G}, \cl_{K}({S_{K}}) ).$$

Comme $p$ est totalement ramifi\'e dans $k_\infty/k$,  
$\cl_{K}({S_{K}})^{G} = \cl_{K}({S_{K}})$ et finalement on obtient~:
$$1 \to \Cl_{K}^{G} / \cl_{K}(S_{K}) \to 
\Cl_{K}^{S_{K} {G}}\!\! = (\Cl_{K} / \cl_{K}(S_{K}) )^{G} 
\!\! \to {\rm H}^1({G}, \cl_{K}({S_{K}})) , $$

qui fait que la condition $\Cl_{K}^{S_{K} {G}}=1$ pour tout $n \gg 0$ implique 
la conjecture de Greenberg (Th\'eor\`eme \ref{green})~; cette condition est \'equivalente
(cf. Remarque \ref{rem1} (ii)) \`a la r\'eunion de la condition (i) et de la condition 
$(E_k^{S_k} : E_k^{S_k} \cap {\rm N}_{K/k}(K^\times)) = p^{n\cdot (d -1)}$. 
D'apr\`es la Proposition \ref{basetopo}, il suffira qu'elle soit satisfaite 
pour $K=k_1$ pour qu'elle le soit pour tout $n \geq 1$.
\end{proof}

\begin{remarks}\label{cohomo}
(i) On a ${\rm H}^1({G}, \cl_{K}({S_{K}})) = {}_{\nu}\cl_{K}({S_{K}})$ 
o\`u $\nu := \nu_n$ est la norme alg\'ebrique pour $G$~; par cons\'equent, si ${\rm N}_{K/k}$ 
est la norme arithm\'etique (ici surjective) et $j_{K/k}$ l'extension des classes, on a 
$\nu = j_{K/k} \circ {\rm N}_{K/k}$. 
On a ${\rm H}^1({G}, \cl_{K}({S_{K}}))=0$ si et seulement si $\nu$ est injective 
sur $\cl_{K}({S_{K}})$, donc (comme $\nu (\Cl_K({S_K})) =
j_{K/k}(\cl_k({S_k}))$ car ${\rm N}_{K/k}(S_{K}) = S_k$)
si et seulement si  on a l'isomorphisme
$\nu : \cl_{K}({S_{K}}) \ds \mathop{\too}^{\simeq}  j_{K/k}(\cl_k({S_k})) 
\subseteq \cl_{K}({S_{K}})$
qui indique que $\cl_{K}({S_{K}}) = j_{K/k}(\cl_k({S_k}))$~; or
$j_{K/k}(\cl_k({S_k})) = \cl_{K}({S_{K}})^{p^n}$ puisque si ${\mathfrak p} \in S_k$, on a
$j_{K/k}({\mathfrak p}) =  {\mathfrak P}^{p^n}$, ${\mathfrak P}\mid {\mathfrak p}$ dans $K$, 
d'o\`u $\cl_{K}({S_{K}}) = \cl_{K}({S_{K}})^{p^n}$ et $\cl_{K}({S_{K}}) =1$.

\smallskip
On a donc ${\rm H}^1({G}, \cl_{K}({S_{K}}))=0$ si et seulement si $\cl_{K}({S_{K}}) =1$.

\smallskip
(ii) Si $\lambda=\mu=0$, $\cl_{K}({S_{K}})$ est born\'e dans la tour et
on a $j_{K/k}(\cl_k({S_k}))= \cl_{K}({S_{K}})^{p^n} = 1$ pour $n \gg 0$ 
(capitulation de $\cl_k({S_k})$ dans $k_\infty$). De m\^eme il y a, pour tout $n$, 
capitu\-lation de $\cl_{K}({S_{K}})$ dans $k_ \infty $.
\end{remarks}

\section{Symboles de restes normiques} \label{hasse}\label{sect4}
Par commodit\'e, rappelons (cf. \cite[\S\,II.4.4.3]{Gra1}) une m\'ethode 
de calcul des symboles de restes normiques de Hasse
$\big(\frac{x\,,\,{k_n}/k}{{\mathfrak p}} \big) \in G_n := {\rm Gal}(k_n/k)$, dans le cas particulier
d'un corps de nombres  Galoisien $k$ avec $k_n \subset k_\infty$ de degr\'e $p^n$ sur $k$, 
et relativement \`a un id\'eal premier ${\mathfrak p} \mid p$ de $k$ pour $p > 2$
totalement d\'ecompos\'e dans $k/\Q$. 
Dans ce cas, le conducteur de $k_n/k$ divise $(p^{n+1})$ car on a, localement,
$1+ p^{n+1} \,\alpha_0 = (1+ p\,\alpha'_0)^{p^n} =  {\rm N}_{k_n/k}  (1+ p\,\alpha'_0)$,
o\`u $\alpha_0$, $\alpha'_0$ sont des $p$-entiers du produit des $p$-compl\'et\'es de $k$.
Le conducteur de $\Q_n$ est $p^{n+1}$ ($n \geq 1$).

\smallskip
Les calculs en question sont li\'es \`a la th\'eorie des genres dont nous rappelons
d'abord l'essentiel.

\subsection{Suite exacte des genres pour les sous-corps 
de $k_\infty/k$}\label{seg}
On d\'esigne par $H_k$ et $H_{K}$ les $p$-corps de classes de Hilbert de $k$ 
et $K := k_n$. Les groupes d'inertie $I_{\mathfrak p}(K/k)$ des 
${\mathfrak p} \mid p$ dans $K/k$ sont \'egaux \`a $G:=G_n$. 

\smallskip
On consid\`ere
l'application $\omega := \omega_n$ qui associe \`a $x \in E_k$ la famille des symboles de Hasse
$\big( \frac{x \, ,\, K/k}{{\mathfrak p}}\big) \in G$, ${\mathfrak p}\mid p$.

\smallskip
On obtient alors la suite exacte des genres interpr\`etant la formule du produit des symboles de Hasse 
d'une unit\'e (voir, e.g., \cite[Proposition IV.4.5.1]{Gra1} pour $T=S=\emptyset$)~:
\begin{equation}\label{genera}
\begin{aligned}
1  \to  E_k/E_k  \cap {\rm N}_{K/k}(K^\times)\  \mathop{\tooo}^{\omega} \ 
&\Omega(K/k) \subseteq \plus_{{\mathfrak p} \mid p} I_{\mathfrak p}(K/k)\\
&\  \mathop{\tooo}^{\pi}\  {\rm Gal}(H_{K/k}/ K H_k) \to 1 ,
\end{aligned}
\end{equation}

\noindent
o\`u $\Omega(K/k)  := \Big \{ (s_{\mathfrak p})_{\mathfrak p} \in \plus_{{\mathfrak p} \mid p} 
I_{\mathfrak p}(K/k), \ \, \prd_{{\mathfrak p} \mid p} s_{\mathfrak p} = 1\Big\}
\simeq (\Z/ p^n \Z)^{d -1}$ et o\`u
$H_{K/k}$ est {\it le $p$-corps des genres} de $K $ d\'efini comme la
sous-extension maximale de $H_{K}$, Ab\'elienne sur $k$, selon le sch\'ema suivant,
$H_{K/k}$ \'etant fix\'e par l'image, par l'application d'Artin, de 
$\Cl_{K}^{1-\sigma}$, o\`u $\sigma := \sigma_n$ est un g\'en\'erateur de $G$ 
(en effet, le groupe des commutateurs $[\Gamma, \Gamma]$ de
$\Gamma = {\rm Gal}(H_{K}/k)$ est ${\rm Gal}(H_{K}/K)^{1-\sigma}$, 
$\Gamma/{\rm Gal}(H_{K}/K)$ \'etant cyclique)~:
\unitlength=0.75cm
$$\vbox{\hbox{\hspace{-5.5cm}  
\begin{picture}(11.5,3.1)
\put(8.2,2.50){\line(1,0){2.2}}
\put(11.9,2.50){\line(1,0){2.5}}
\put(4.2,2.50){\line(1,0){2.5}}
\put(4.1,0.50){\line(1,0){2.8}}
\put(8.6,2.75){\footnotesize$\prod_{{\mathfrak p} \mid p} s'_{\mathfrak p}$}
\put(12.4,2.7){\footnotesize $\simeq \! \Cl_{K}^{1-\sigma}$}
\put(5.1,0.6){\footnotesize $\simeq \! \Cl_k$}
\put(3.7,0.9){\line(0,1){1.20}}
\put(7.5,0.9){\line(0,1){1.20}}
\put(10.5,2.4){$H_{K/k}$}
\put(14.6,2.4){$H_{K}$}
\put(6.9,2.4){$K H_k$}
\put(2.6,2.4){$K\! =\! k_n$}
\put(7.2,0.4){$H_k$}
\put(3.6,0.4){$k$}
\put(2.2,1.4){\footnotesize $G\!=\!G_n$}
\bezier{350}(7.85,0.5)(10.5,0.8)(11.0,2.2)
\put(10.3,1.0){\footnotesize
$\langle I_{\mathfrak p}(H_{K\!/\!k}/k) \rangle_{{\mathfrak p} \mid p}$}
\end{picture}   }} $$
\unitlength=1.0cm

L'image de $\omega $ est contenue dans $\Omega(K/k)$
et l'application $\pi := \pi_n$ est ainsi d\'efinie~: 
\`a $(s_{\mathfrak p})_{\mathfrak p} \in 
\bigoplus_{{\mathfrak p} \mid p} I_{\mathfrak p}(K/k)$, 
$\pi$ associe le produit $\prod_{{\mathfrak p} \mid p} s'_{\mathfrak p}$ 
des rel\`evements $s'_{\mathfrak p}$ des 
$s_{\mathfrak p}$ dans les groupes d'inertie $I_{\mathfrak p}(H_{K/k}/k)$
qui engendrent ${\rm Gal}(H_{K/k}/ H_k)$~; il r\'esulte de la formule 
du produit que si $(s_{\mathfrak p})_{\mathfrak p}$ est dans $\Omega(K/k)$, 
$\prod_{{\mathfrak p} \mid p} s'_{\mathfrak p}$ fixe \`a la fois $H_k$ et $K $, donc $K H_k$. 

\smallskip
On a ainsi $\pi \big(\Omega(K/k) \big) = {\rm Gal}(H_{K/k}/K H_k)$ et ${\rm Ker}(\pi) = \omega(E_k)$. 

\smallskip
On a, comme attendu, 
$[H_{K/k} : K] = \frac{\order \Cl_{K}}{\order \Cl_{K}^{1-\sigma}} = \order  \Cl_{K}^{G}$.
On montrera que ce degr\'e $[H_{K/k} : K]$ (et donc $\order  \Cl_{K}^{G}$) est constant \`a partir 
du rang $n=e$ (i.e., $[K : \Q] \geq p^e$)
donn\'e par l'exposant $p^e$ du groupe de torsion de ${\rm Gal}(H_k^{\rm pr}/H_k)$,
o\`u $H_k^{\rm pr}$ est la pro-$p$-extension Ab\'elienne $p$-ramifi\'ee maximale de $k$,
et que ce degr\'e est \'egal \`a $\order  {\mathcal T}_k$, o\`u ${\mathcal T}_k$ est le groupe de torsion
de ${\rm Gal}(H_k^{\rm pr}/k)$ (cf. Th\'eor\`emes \ref{regulateur} et \ref{constant}).

\subsection{Calcul effectif des symboles 
$\big(\frac{x\,,\, k_n/k}{{\mathfrak p}} \big),\ n \geq 1$} \label{norm}

Soit $x\in k^\times$ et soit ${\mathfrak p} \mid p$ un id\'eal premier de $k$ au-dessus de $p$
($x$ n'est pas suppos\'e \'etranger \`a ${\mathfrak p}$).
 Soit $x'_{\mathfrak p} \in k^\times$ (appel\'e un
${\mathfrak p}$-associ\'e de $x$ relativement \`a $k_n/k$; on ne fait 
provisoirement aucune hypoth\`ese 
sur la d\'ecomposition de $p$ et si ${\mathfrak p} \mid p$ est unique, $x$ est son 
propre associ\'e) tel que (th\'eor\`eme des restes chinois)~:

\smallskip
\quad (i) $x'_{\mathfrak p} x^{-1} \equiv 1 \pmod {{\mathfrak p}^{n+1}}$,

\quad (ii) $x'_{\mathfrak p} \equiv 1 \pmod {{\mathfrak p'}^{n+1}}$, pour tout  
${\mathfrak p}' \mid p,\ {\mathfrak p}'\ne {\mathfrak p}$.

Par la formule du produit, on a
$\Big(\frac{x'_{\mathfrak p}\,,\,k_n/k}{{\mathfrak p}} \Big) = \hbox{$\prd_{{\mathfrak q} \,,\, 
{\mathfrak q} \,\ne\, {\mathfrak p}}$}
\Big (\frac{x'_{\mathfrak p}\,,\, k_n/k}{{\mathfrak q}} \Big)^{-1}$, et comme 
$\Big( \frac {x\,,\, k_n/k}{{\mathfrak p}} \Big) = \Big(\frac {x'_{\mathfrak p}\,,\, k_n/k}{{\mathfrak p}} \Big)$
d'apr\`es (i) et la d\'efinition du ${\mathfrak p}$-conducteur de $k_n/k$,
$\Big(\frac{x\,,\, k_n/k}{{\mathfrak p}} \Big) = \prod_{{\mathfrak q}  \,,\, {\mathfrak q} \,\ne \, {\mathfrak p}}
\Big (\frac{x'_{\mathfrak p}\,,\, k_n/k}{{\mathfrak q}} \Big)^{-1}$.
Calculons les symboles de ce~produit~:

$\quad\bullet\ $ 
si ${\mathfrak q} = {\mathfrak p}' \mid p$,  ${\mathfrak p'} \ne {\mathfrak p}$, 
$x'_{\mathfrak p} \equiv 1 \pmod {{\mathfrak p'}^{n+1}}$ et on a
$\Big(\frac {x'_{\mathfrak p}\,,\, k_n/k}{{\mathfrak p'}} \Big) = 1$,

$\quad\bullet\ $ 
si ${\mathfrak q} \nmid p$, ${\mathfrak q}$ est non ramifi\'e
et dans ce cas, $\Big (\frac{x'_{\mathfrak p}\,,\, k_n/k}{{\mathfrak q}} \Big) =
\Big (\frac{k_n /k}{{\mathfrak q}} \Big)^{v_{\mathfrak q}(x'_{\mathfrak p})}$ 
(o\`u $\Big (\frac{k_n /k}{{\mathfrak q}} \Big)$ est le symbole de Frobenius de ${\mathfrak q}$
et $v_{\mathfrak q}$ la valuation ${\mathfrak q}$-adique).

Finalement, $\Big (\frac{x\,,\, k_n/k}{{\mathfrak p}} \Big) = 
\prod_{{\mathfrak q} \nmid p}\Big (\frac{k_n /k}{{\mathfrak q}}\Big)^{-{v_{\mathfrak q}}(x'_{\mathfrak p})}$.
Posons ${\mathfrak a}_{\mathfrak p}(x) = \prod_{{\mathfrak q}\nmid p} 
{\mathfrak q}^{{v_{\mathfrak q}}(x'_{\mathfrak p})}$ 
(${\mathfrak a}_{\mathfrak p}(x)$ est \'etranger \`a $p$), alors on a
$(x'_{\mathfrak p}) =: {\mathfrak p}^{ v_{\mathfrak p} (x'_{\mathfrak p})  } {\mathfrak a}_{\mathfrak p}(x) = 
{\mathfrak p}^{v_{\mathfrak p}(x)} {\mathfrak a}_{\mathfrak p}(x)$,
et on a obtenu
$\Big (\frac{x\,,\, k_n/k}{{\mathfrak p}} \Big) = \Big (\frac{k_n /k}{{\mathfrak a}_{\mathfrak p}(x)} \Big)^{-1}$
(inverse du symbole d'Artin de ${\mathfrak a}_{\mathfrak p}(x)$). On v\'erifie que 
$\Big (\frac{k_n /k}{{\mathfrak a}_{\mathfrak p}(x)}\Big)$ ne d\'epend pas du choix de $x'_{\mathfrak p}$.
Si $v_{\mathfrak p}(x)=0$, alors $ {\mathfrak a}_{\mathfrak p}(x)=(x'_{\mathfrak p})$. 
En d\'epit des notations, $(x'_{\mathfrak p})$ et ${\mathfrak a}_{\mathfrak p}(x)$ d\'ependent 
de $n \geq 1$.

\smallskip
D'apr\`es le th\'eor\`eme de rel\`evement normique dans $k/\Q$, l'image canonique de 
$\Big (\frac{k_n /k}{{\mathfrak a}_{\mathfrak p}(x)}\Big) \in {G_n}$
dans ${\rm Gal}(\Q_n/\Q) \simeq \{a \in (\Z/ p^{n+1}\Z)^\times, \ a \equiv 1 \pmod p \}$,
isomorphe \`a $\Z/ p^n \Z$, est  le symbole d'Artin
$\Big (\frac{\Q_n/\Q}{{\rm N}_{k/\Q}( {\mathfrak a}_{\mathfrak p}(x))}\Big)$ qui 
{\it caract\'erise} $\Big (\frac{k_n /k}{{\mathfrak a}_{\mathfrak p}(x)}\Big)$ par
rel\`evement, et o\`u ${\rm N}_{k/\Q}( {\mathfrak a}_{\mathfrak p}(x)) >0$ 
(norme absolue).

\begin{definition} \label{delta}  
(i) Si $x \in k^\times$ est \'etranger \`a $p$, on d\'efinit les coefficients $\delta_{\mathfrak p}(x)$, 
${\mathfrak p} \mid p$, par la relation~:
$(x^{p-1}-1) = p\cdot \prd_{{\mathfrak p} \mid p} {\mathfrak p}^{\delta_{\mathfrak p}(x)} \cdot 
{\mathfrak b}_p$, ${\mathfrak b}_p$  \'etranger \`a $p$.

(ii) Si $x \in k^\times$ n'est pas \'etranger \`a $p$, on d\'efinit les coefficients $\delta_{\mathfrak p}(x)$ par
les relations~:
$\big( (x\cdot p^{-v_{\mathfrak p}(x)})^{p-1} - 1\big) = {\mathfrak p} \cdot 
{\mathfrak p}^{\delta_{\mathfrak p}(x)} \cdot {\mathfrak b}_{\mathfrak p}, 
\ \ \hbox{${\mathfrak b}_{\mathfrak p}$ \'etranger \`a ${\mathfrak p}$, ${\mathfrak p} \mid p$.}$
\end{definition}

Ces d\'efinitions peuvent s'exprimer en termes de valuations logarithmiques conduisant aux
groupes de classes logarithmiques introduits par Jaulent (cf. \cite{J4}
ainsi que \cite{BJ} pour les aspects num\'eriques, et \cite[\S\,III.5]{Gra1} pour des 
g\'en\'eralit\'es logarithmiques et cyclotomiques li\'ees \`a la conjecture de Gross). 
Jaulent (\cite[Th\'eor\`eme 7]{J6}) montre que 
la conjecture de Greenberg \'equivaut \`a la capitulation du $p$-groupe des classes 
logarithmiques de $k$ dans $k_\infty$.

\medskip
On suppose d\'esormais $p$ totalement d\'ecompos\'e dans $k$.
 
\begin{lemma} \label{lemx}
Soit $x \in k^\times$ \'etranger \`a $p$ et soit ${\mathfrak p} \mid p$ fix\'e.
Soit ${\mathfrak a}_{\mathfrak p}(x) = (x'_{\mathfrak p})$, o\`u
$x'_{\mathfrak p}$ est un ${\mathfrak p}$-associ\'e de $x$ relativement \`a $k_n/k$. 
Alors ${\rm N}_{k/\Q} ({\mathfrak a}_{\mathfrak p}(x)) 
\equiv x \pmod {{\mathfrak p}^{n+1}}$ et, pour tout $n \geq \delta_{\mathfrak p}(x)$,
$\hbox{$\frac{1}{p}$} \cdot {\rm log} ({\rm N}_{k/\Q}({\mathfrak a}_{\mathfrak p}(x))) \equiv
\alpha_{\mathfrak p}(x) \cdot p^{\delta_{\mathfrak p}(x)} \pmod{p^{n}}$, 
$\alpha_{\mathfrak p}(x)  \in \Z_p^\times$.
\end{lemma}

\begin{proof} 
On peut \'ecrire (i) et (ii) d\'efinissant $x'_{\mathfrak p}$, sous la forme~:

\smallskip
\quad (i$'$) $x'_{\mathfrak p} \, x^{-1} \equiv 1 \pmod {{\mathfrak p}^{n+1}}$,

\smallskip
\quad (ii$'$) $x'_{\mathfrak p} \equiv 1 \pmod {\tau^{-1} {\mathfrak p}^{n+1}}$, pour tout  
$\tau \in {\rm Gal}(k/\Q)$, $\tau \ne 1$.

\smallskip
On a ${\rm N}_{k/\Q}( x'_{\mathfrak p}) = \prd_{\tau \in {\rm Gal}(k/\Q)}(\tau x'_{\mathfrak p}) = 
x'_{\mathfrak p} \cdot \prd_{\tau \ne 1}(\tau x'_{\mathfrak p})$ qui conduit \`a~:
$${\rm N}_{k/\Q} (x'_{\mathfrak p}) \equiv x \pmod {{\mathfrak p}^{n+1}}. $$

Donc, pour $n > \delta_{\mathfrak p}(x)$, il vient en \'elevant la congruence ci-dessus 
\`a la puissance $p-1$~:
$$v_{\mathfrak p}({\rm N}_{k/\Q}(x'_{\mathfrak p})^{p-1} - 1) = v_{\mathfrak p}(x^{p-1}-1)=
\delta_{\mathfrak p}(x) + 1 ; $$

mais comme ${\rm N}_{k/\Q}(x'_{\mathfrak p})$ est
rationnel, on a $v_{p}({\rm N}_{k/\Q}(x'_{\mathfrak p})^{p-1} - 1) = 
\delta_{\mathfrak p}(x) + 1$, d'o\`u le lemme en prenant le ``logarithme
normalis\'e'' $\frac{1}{p} \cdot {\rm log}$ (ce qui donne $0$ modulo $p^n$
si $n \leq \delta_{\mathfrak p}(x)$).
\end{proof}

Le cas $v_{\mathfrak p}(x) \ne 0$ se traite comme suit et couvre tous les cas~:

\begin{lemma} \label{lemxp}
Soient ${\mathfrak p} \mid p$ fix\'e et $x'_{\mathfrak p}$ un 
${\mathfrak p}$-associ\'e de $x$ relativement \`a $k_n/k$~; soit
${\mathfrak a}_{\mathfrak p}(x) = (x'_{\mathfrak p}) \cdot {\mathfrak p}^{- v_{\mathfrak p}(x)}$.
Alors on a ${\rm N}_{k/\Q}({\mathfrak a}_{\mathfrak p}(x)) \equiv 
x \cdot p^{- v_{\mathfrak p}(x)}  \pmod {{\mathfrak p}^{n+1}}$, et on a,
pour tout $n \geq \delta_{\mathfrak p}(x)$,
$\hbox{$\frac{1}{p}$} \cdot {\rm log} ({\rm N}_{k/\Q}({\mathfrak a}_{\mathfrak p}(x)) ) \equiv
\alpha_{\mathfrak p}(x) \cdot p^{\delta_{\mathfrak p}(x)} \pmod{p^{n}}$, 
$\alpha_{\mathfrak p}(x)  \in \Z_p^\times$.
\end{lemma}

\begin{proof}
On a ${\rm N}_{k/\Q}({\mathfrak a}_{\mathfrak p}(x)) = 
{\rm N}_{k/\Q} (x'_{\mathfrak p}) \cdot p^{- v_{\mathfrak p}(x)}$~; on est ramen\'e
au calcul pr\'ec\'edent via (i$'$), (ii$'$), o\`u l'on aura 
${\rm N}_{k/\Q} (x'_{\mathfrak p}) \, x^{-1} \equiv 1 \pmod {{\mathfrak p}^{n+1}}$.
D'o\`u ${\rm N}_{k/\Q}({\mathfrak a}_{\mathfrak p}(x)) \equiv 
x \cdot p^{- v_{\mathfrak p}(x)} \!\! \pmod {{\mathfrak p}^{n+1}}$ qui conduit,
pour $n > \delta_{\mathfrak p}(x)$,~\`a~:
$$v_p ({\rm N}_{k/\Q} \big({\mathfrak a}_{\mathfrak p}(x))^{p-1} - 1\big) =
v_{\mathfrak p} \big ((x \cdot p^{- v_{\mathfrak p}(x)})^{p-1} - 1\big ) = \delta_{\mathfrak p}(x) + 1, $$

et \`a une conclusion analogue pour l'expression du logarithme.
\end{proof}

Vu comme \'el\'ement de $\{a \in (\Z/ p^{n+1}\Z)^\times, \ a \equiv 1 \!\! \pmod p \}$,
l'automorphisme $\Big (\frac{\Q_n/\Q}{{\rm N}_{k/\Q}({\mathfrak a}_{\mathfrak p}(x))}\Big)$
est repr\'esent\'e sous forme additive par le logarithme normalis\'e $\frac{1}{p} \cdot {\rm log}$ de
${\rm N}_{k/\Q}({\mathfrak a}_{\mathfrak p}(x)) > 0$.
On identifie cet automorphisme \`a~:
$$\hbox{$\frac{1}{p} \cdot {\rm log} ({\rm N}_{k/\Q}({\mathfrak a}_{\mathfrak p}(x))) =:
\alpha_{\mathfrak p}(x) \cdot p^{\delta_{\mathfrak p}(x)} \!\!\!\pmod{p^{n}}, \ \ 
\alpha_{\mathfrak p}(x) \in \Z_p^\times$.} $$

On peut \'enoncer en r\'esum\'e~:

\begin{theorem}\label{ordre} Soit $k$ Galoisien r\'eel et soit $p>2$ un nombre 
premier totalement d\'ecompos\'e dans $k$. Soit $k_\infty$ la $\Z_p$-extension cyclotomique de $k$
et, pour tout $n \geq 0$, soit $k_n$ le sous-corps de $k_\infty$ de degr\'e $p^n$ sur $k$. 

\smallskip
Soit $x \in k^\times$ et soient $\delta_{\mathfrak p}(x) \geq 0$, pour tout 
${\mathfrak p} \mid p$, les entiers d\'efinis par $\delta_{\mathfrak p}(x) +1:= 
v_{\mathfrak p}((x\,p^{-v_{\mathfrak p}(x)})^{p-1}-1)$ (cf. D\'efinition~\ref{delta})~: 

\smallskip
(i)Alors $x$ est norme locale en ${\mathfrak p} \mid p$ dans $k_n/k$ si et seulement si
$\delta_{\mathfrak p}(x) \geq n$. 

\smallskip
(ii) Soit $x'_{\mathfrak p}$ un ${\mathfrak p}$-associ\'e de $x$
relativement au calcul du symbole $\big (\frac{x\,,\,{k_n}/k}{{\mathfrak p}} \big)$.
 Si $\delta_{\mathfrak p}(x) \leq n$, l'ordre et l'image de 
$\big (\frac{x\,,\,{k_n/k}}{{\mathfrak p}} \big)$ dans ${\rm Gal}(\Q_n/\Q)$ 
sont $p^{n - \delta_{\mathfrak p}(x)}$ et $\frac{1}{p} \cdot 
{\rm log} ({\rm N}_{k/\Q}((x'_{\mathfrak p}) \cdot p^{-v_{\mathfrak p}(x)})) =: 
\alpha_{\mathfrak p}(x) \cdot p^{\delta_{\mathfrak p}(x)} \pmod {p^n}$, 
$\alpha_{\mathfrak p}(x) \in \Z_p^\times$.

\smallskip
(iii) On a $\delta_{\mathfrak p}(\tau x) = \delta_{\tau^{-1}{\mathfrak p}}(x)$ pour tout 
${\mathfrak p} \mid p$ et $\tau \in {\rm Gal}(k/\Q)$.
\end{theorem}

Soit ${\mathcal N}_k^n$ le sous-groupe des $x \in k^\times$ partout 
normes locales dans $k_n/k$ en dehors de $p$.
D'apr\`es l'\'etude et le calcul effectif des symboles de Hasse vus pr\'ec\'edemment, 
on peut identifier $\omega_n$ (cf. \S\,\ref{seg}), qui \`a $x \in {\mathcal N}_k^n$
associe la famille des symboles de Hasse 
$\big(\big (\frac{x\,,\,{k_n/k}}{{\mathfrak p}}\big) \big)_{{\mathfrak p}\mid p}$ 
dans $\Omega(k_n/k)$, \`a l'application~:
$$\begin{array}{cccc}  \omega'_n  \, : & {\mathcal N}_k^n
& \tooo & \ds  \prd_{{\mathfrak p}\mid p} \Z/ p^n \,\Z . \\ & x & \longmapsto 
& \big(\alpha_{\mathfrak p}(x) \cdot p^{\delta_{\mathfrak p}(x)} \!\!\pmod {p^n}\big)_{{\mathfrak p}\mid p}
\end{array}$$

\noindent
dont l'image est dans l'ensemble des 
$(\alpha_{\mathfrak p} \cdot  p^{\delta_{\mathfrak p}})_{{\mathfrak p}\mid p}
\in \prd_{{\mathfrak p}\mid p} \Z/ p^n \,\Z$ v\'erifiant la relation $\sm_{{\mathfrak p}\mid p} 
\alpha_{\mathfrak p} \cdot p^{\delta_{\mathfrak p}} 
\equiv 0 \pmod {p^n}$ (formule du produit sur les $p$-places).

\begin{corollary}\label{fp} Si $x \in k^\times$ est partout norme locale en dehors de $p$ 
dans $k_n/k$ (i.e., $(x)$ norme d'un id\'eal de $k_n$), il est alors partout norme locale 
(donc norme globale dans $k_n/k$) si et seulement si $\delta_{\mathfrak p}(x) \geq n$ 
pour tout ${\mathfrak p} \mid p$ (sauf un en raison de la formule du produit).
\end{corollary}

D'un point de vue heuristique on a, a priori, $\delta_{\mathfrak p}(x) \geq r$ 
avec la probabilit\'e $\frac{1}{p^r}$, de sorte qu'en g\'en\'eral
$\big (\frac{x\,,\,{k_n}/k}{{\mathfrak p}} \big) = \big (\frac{k_n /k}{{\mathfrak a}_{\mathfrak p}(x)} \big)^{-1}$ 
est un g\'en\'erateur de $G_n$, ce qui est tr\`es favorable pour la conjecture de Greenberg 
comme le montre, par exemple, le Th\'eor\`eme \ref{thm2} o\`u le point (ii) est satisfait d\`es que
l'on a suffisamment de $S_k$-unit\'es dont les
symboles engendrent $G_1$.
Cependant, le cas de $x \in k^\times$ partout norme locale en dehors de $p$ dans $k_n/k$ 
montre que les $\delta_{\mathfrak p}(x)$ ne sont pas ind\'ependants 
en raison de la formule du produit (consid\'er\'ee comme unique)~; pour $x$ 
\'etranger \`a $p$, ceci implique ${\rm N}_{k/\Q}(x)^{p-1} \equiv 1 \pmod {p^{n+1}}$,
car le symbole d'Artin de ${\rm N}_{k/\Q}(x)$ dans $\Q_n/\Q$ est \'egal \`a $1$.

\begin{proposition} \label{basetopo}
(i) Soit $\Lambda$ un sous-groupe de $k^\times$ tel que tout $x \in \Lambda$ 
soit partout norme locale en dehors de $p$ dans $k_n/k$, pour $n\geq 1$ donn\'e (i.e., $(x)$ norme 
d'un id\'eal de $k_n$). On suppose que $(\Lambda : \Lambda \cap {\rm N}_{k_{1}/k}(k_{1}^\times)) 
= p^{d-1}$~; alors on a $(\Lambda : \Lambda \cap {\rm N}_{k_{n}/k}(k_{n} ^\times)) = p^{n \cdot (d-1)}$.

(ii) Si $\Lambda$ est un sous-groupe de $E_k^{S_k}$, on a 
$(\Lambda : \Lambda \cap {\rm N}_{k_n/k}(k_{n} ^\times))=p^{n \cdot (d-1)}$ 
pour tout $n \geq 1$ d\`es que ceci est vrai pour $n=1$.
\end{proposition}

\begin{proof} L'hypoth\`ese signifie $\omega_{1}(\Lambda)= \Omega(k_{1}/k)$ 
(cf. \S\,\ref{seg}). Il existe donc des \'el\'ements $x_j \in \Lambda$, $1 \leq j \leq d-1$, tels que:

\centerline{$\Omega(k_{1}/k) = \plus_{j=1}^{d-1} \langle \omega_{1}(x_j) \rangle$~;}
 
les $\omega_{1}(x_j)$ sont les images canoniques des $\omega_{\infty}(x_j) := 
\big (\big (\frac{x_j\,,\,{k_{\infty}}/k}{{\mathfrak p}} \big) \big)_{{\mathfrak p} \mid p}$ 
(par restriction des symboles de Hasse) qui constituent une $\Z_p$-base topologique 
d'un sous-$\Z_p$-module (de dimension $d-1$) de ${\rm Gal}(k_\infty/k)^d \simeq \Z_p^d$
car toute relation $\prod_{j=1}^{d-1}\omega_{\infty}(x_j)^{a_j} = 1$, $a_j \in \Z_p$ non tous 
divisibles par $p$, conduit par restriction \`a une relation non triviale au niveau $n=1$.

\smallskip
Au niveau $n$, $\omega_{n}(\Lambda) = \bigoplus_{j=1}^{d-1} \langle \omega_{n}(x_j) \rangle$
d'ordre $p^{n \cdot (d-1)}$. Mais pour $m>n$, les $x \in \Lambda$ ne sont plus n\'ecessairement
normes locales en dehors de $p$ dans $k_m/k$ et donc $\omega_{m}(x)$ n'est 
plus n\'ecessairement dans $\Omega(k_{m}/k)$.

\smallskip
Pour $\Lambda  \subseteq E_k^{S_k}$, la condition de normes locales 
en dehors de $p$ dans $k_n/k$ est satisfaite pour tout $n$
et dans ce cas, les $\omega_{\infty}(x_j)$ engendrent $\Omega(k_\infty/k)$.
\end{proof}

\subsection{Groupe de torsion de la $p$-ramification Ab\'elienne}\label{torsion}
Soit ${\mathcal T}_k$ le groupe de torsion du groupe de Galois de la pro-$p$-extension Ab\'elienne 
$p$-ramifi\'ee maximale $H_k^{\rm pr}$ de $k$. 

\smallskip
On a, en d\'esignant par $U_k := \prod_{{\mathfrak p} \mid p} U_{\mathfrak p}$
(o\`u $U_{\mathfrak p}=1+{\mathfrak p}$) le groupe des unit\'es locales principales 
de $k$, et par $\overline E_k$ l'adh\'erence de $E_k$ dans $U_k$ ou plus
pr\'ecis\'ement l'image dans $U_k$ de $E_k \otimes \Z_p$, la suite exacte classique~:

\medskip
\centerline{$1 \too  {\rm tor}_{\Z_p} (U_k/ \overline E_k) \tooo {\mathcal T}_k \tooo \Cl_k \too 1$,}

\medskip
puisqu'ici $H_k \cap k_\infty = k$.
En appliquant la formule analytique donnant $\order  {\mathcal T}_k$ dans le cas r\'eel
sous la conjecture de Leopoldt (cf. \cite[Appendix (1975)]{Co},
\cite[(1978)]{Se2},  \cite[Remark III.2.6.5 (i)]{Gra1}), en
tenant compte du fait que $p$ est totalement d\'ecompos\'e dans 
$k$, que $\frac{1} {p} \prod_{{\mathfrak p} \mid p} {\rm N}_{k/\Q}( {\mathfrak p})= p^{d-1}$, 
et que le r\'egulateur $p$-adique normalis\'e $R_k$ est le 
r\'egulateur $p$-adique classique divis\'e par $p^{d-1}$, on a
$\order  {\rm tor}_{\Z_p}(U_k/ \overline E_k) \sim R_k$ (voir aussi \cite{Gra7}).

\smallskip
Par abus, nous \'ecrirons que ce r\'egulateur $R_k$ est \'egal \`a $p^{v_p(R_k)}$
bien qu'il soit (sous la conjecture de Leopoldt) un \'el\'ement non nul de $\Z_p$ 
d\'efini \`a une unit\'e $p$-adique pr\`es.
Posons $U_k^* := \{u \in U_k, \  {\rm N}_{k/\Q} (u) = 1\}$. De fait, on a 
$${\rm tor}_{\Z_p}(U_k/ \overline E_k) = U_k^*/ \overline E_k, $$ 

car si 
$u^{p^r} \in \overline E_k$, alors ${\rm N}_{k/\Q} (u) = 1$ puisque
${\rm N}_{k/\Q} (\overline E_k) = \{1\}$ et que $U_k$ est sans $p$-torsion~;
enfin, $U_k^*$ est un $\Z_p$-module libre de rang $d-1$ dans lequel 
$\overline E_k$ est d'indice fini (conjecture de Leopoldt). D'o\`u finalement~:
\begin{equation}\label{T}
\order {\mathcal T}_k = \order  \Cl_k \cdot \order  (U_k^*/ \overline E_k) \ \ \  \&
\ \ \  \order (U_k^*/ \overline E_k) = R_k. \hspace{0.5cm}
\end{equation}

\begin{theorem} \label{regulateur} Soit $k$ Galoisien r\'eel
de degr\'e $d$ dans lequel $p >2$ est totalement d\'ecompos\'e. On suppose que la 
conjecture de Leopoldt est vraie pour $p$ dans $k$.
Soit ${\mathcal T}_k$ le groupe de torsion du groupe de Galois 
de la pro-$p$-extension Ab\'elienne $p$-ramifi\'ee maximale $H_k^{\rm pr}$ de $k$ 
et soit $p^{e}$ l'exposant de $U_k^*/ \overline E_k$.

\smallskip
(i) Pour tout $n$, il existe une injection $\psi_n$ de 
$E_k \cap {\rm N}_{k_n/k}( k_n^\times) \big / E_k^{p^n}$
dans $U_k^*/ \overline E_k$~; par cons\'e\-quent
$\ds \frac{p^{n \cdot (d -1)}}{(E_k : E_k \cap {\rm N}_{k_n/k}( k_n^\times))}$ divise $R_k$
et il en r\'esulte que $\order  \Cl_{k_n}^{G_n} = \order \Cl_k \cdot \ds 
\frac{p^{n \cdot (d -1)}}{(E_k : E_k \cap {\rm N}_{k_n/k}( k_n^\times))}$ 
divise $\order  {\mathcal T}_k$ (cf. Relation \eqref{T}).

\smallskip
(ii) Pour tout $n \geq e$, l'application $\psi_n$ est surjective, et alors les 
divisibilit\'es pr\'ec\'edentes sont des \'egalit\'es. En particulier on a $\order  \Cl_{k_n}^{G_n}=
\order  {\mathcal T}_k$ et la suite exacte
$1 \to E_k^{p^n} \tooo E_k \cap {\rm N}_{k_n/k}(k_n^\times) \ds
\mathop{\tooo}^{\psi_n}$ $U_k^*/ \overline E_k \to 1$.

\smallskip
(iii) Si l'on peut trouver $n_0 \geq e$ tel que le crit\`ere de Greenberg soit
v\'erifi\'e en $n_0$ (i.e., $\Cl_{k_{n_0}}^{G_{n_0}} = \cl_{k_{n_0}}(S_{k_{n_0}})$), alors 
il est v\'erifi\'e pour tout $n \geq n_0$ et il en r\'esulte que $\lambda = \mu =0$.
 \end{theorem}

\begin{proof} 
Nous supposerons dans la suite que tout nombre \'etranger \`a $p$ est de fait congru \`a $1$
modulo $p$ (quitte \`a l'\'elever \`a la puissance $p-1$).

\smallskip
Soit $\varepsilon \in E_k \cap {\rm N}_{k_n/k}( k_n^\times)$~; par cons\'equent on a 
$\delta_{\mathfrak p}(\varepsilon) \geq n$ pour tout ${\mathfrak p} \mid p$ 
(Corollaire \ref{fp}). Donc $\varepsilon = u^{p^n}$, o\`u $u \in U_k^*$ est unique.

\smallskip
($\alpha$) (d\'efinition de $\psi_n$). Soit $\psi_n$ l'application qui \`a 
$\varepsilon \in E_k \cap {\rm N}_{k_n/k}( k_n^\times)$ 
associe la classe de $u$ dans $U_k^*/ \overline E_k$. 
L'application $\psi_n$ est bien d\'efinie.

\smallskip
($\beta$) (calcul de ${\rm Ker}(\psi_n)$). Si $u_0 \in \overline E_k$, alors
$\varepsilon \in (\overline E_k)^{p^n}$ est arbitrairement proche d'un \'el\'ement de $E_k^{p^n}$~;
en effet, $\varepsilon= \varepsilon'^{p^n}_N \cdot u_N$, $\varepsilon'_N \in E_k$, $u_N \to 1$ 
dans $E_k$ pour $N \to \infty$ (e.g., $u_N \equiv 1\pmod {p^{N+1}}$), d'o\`u $u_N \in E_k^{p^n}$
(conjecture de Leopoldt \cite[Th\'eor\`eme III.3.6.2 (iv)]{Gra1} en lien avec
la constante de Kummer--Leopoldt \cite{AN}, \cite{Gra7})~; d'o\`u $\varepsilon  \in E_k^{p^n}$ 
et ${\rm Ker}(\psi_n) = E_k^{p^n}$. A ce stade, puisque $(E_k : E_k^{p^n}) = p^{n \cdot (d -1)}$, 
on obtient pour tout $n$ les divisibilit\'es du (i).

\smallskip
($\gamma$) (surjectivit\'e pour $n \geq e$). 
Soit $u \in U_k^*$ et \'ecrivons que $u^{p^{e}} \in \overline E_k$~; 
pour tout $N \gg 0$, il existe $\varepsilon'_N \in E_k$ tel que 
$u^{p^{e}} = \varepsilon'_N \cdot u_N$, $u_N \to 1$. 
D'o\`u $u_N = u_1^{p^{N}} =(u_1^{p^{e}})^{p^{N-e}}$,
$u_1 \in U_k^*$ et $u_1^{p^{e}} =: \overline\varepsilon \in \overline E_k$ par d\'efinition de $e$.
Donc $u_N  = \overline \varepsilon'{}^{p^{e}}$ pour $N \gg 0$,
auquel cas $u^{p^{e}} = \varepsilon'_N \cdot \overline\varepsilon'{}^{p^{e}}$~; 
par cons\'equent il existe $u' = u \cdot \overline\varepsilon'{}^{-1}$ 
dans la classe de $u'$ donc de $u$ modulo $\overline E_k$ tel que $u'{}^{p^{e}} = \varepsilon'_N$.
Posons $\varepsilon := \varepsilon'^{p^{n-e}}_N = u'{}^{p^n}$ (car $n \geq e$)~; 
\'etant partout norme locale, $\varepsilon \in  {\rm N}_{k_n/k}( k_n^\times)$ et $\psi_n(\varepsilon)$
est la classe de $u$ modulo $\overline E_k$.

\smallskip
D'o\`u la surjectivit\'e de $\psi_n$, puis
$\ds \frac{p^{n \cdot (d -1)}}{(E_k : E_k \cap {\rm N}_{k_n/k}( k_n^\times))} 
= \order  (U_k^*/ \overline E_k) = R_k$, et finalement 
$\ds \order  \Cl_{k_n}^{G_n} = \order  {\mathcal T}_k$.

\smallskip
($\delta$) (calcul direct de l'ordre pour $n \geq e$). Puisque $p^e$ annule 
$U_k^*/ \overline E_k$, on a pour $n \geq e$, 
$U_k^*/ \overline E_k = U_k^*/ U_k^*{}^{p^n} \overline E_k$ et la suite exacte~:
$$1 \too \overline E_k / \overline E_k \cap U_k^*{}^{p^n}  \tooo U_k^*/ U_k^*{}^{p^n}
\tooo U_k^*/ U_k^*{}^{p^n} \overline E_k \too 1. $$

On a $\overline E_k / \overline E_k \cap U_k^*{}^{p^n} =
E_k / E_k \cap U_k^*{}^{p^n}$~; or $E_k \cap U_k^*{}^{p^n} = 
E_k / E_k \cap {\rm N}_{k_n/k}( k_n^\times)$~;
d'o\`u le r\'esultat puisque $\order  (U_k^*/ U_k^*{}^{p^n}) = p^{n \cdot (d -1)}$.

\smallskip
($\varepsilon$) (effectivit\'e du crit\`ere de Greenberg). Posons $M_1^n := \Cl_{k_{n}}^{G_{n}}$ 
pour tout $n \geq 0$~; alors $\order M_1^n =\order {\mathcal T}_k$ pour tout $n \geq e$. 
Soit $n \geq n_0$ et posons 
$M'_1 := \cl_{k_{n}}(S_{k_{n}}) \subseteq M_1^n$~; par la norme ${\rm N}' :=  {\rm N}_{k_n/k_{n_0}}$
il vient ${\rm N}' (M'_1) = M_1^{n_0}$, car ${\rm N}' (S_{k_{n}}) = S_{k_{n_0}}$. En raison des ordres,
il en r\'esulte que ${\rm N}'$ induit un isomorphisme $M'_1 \simeq M_1^{n_0}$, d'o\`u 
$M'_1 = M_1^n = \cl_{k_{n}}(S_{k_{n}})$ pour tout $n \geq n_0$, d'o\`u $\lambda=\mu=0$ 
(Th\'eor\`eme \ref{green}).
\end{proof}

\begin{theorem} \label{constant} Soit $k$ Galoisien r\'eel de degr\'e $d$ dans 
lequel $p > 2$ est totalement d\'ecompos\'e. On suppose que la 
conjecture de Leopoldt est vraie pour $p$ dans $k$.
Soit ${\mathcal T}_k$ le groupe de torsion du groupe de Galois de la pro-$p$-extension 
Ab\'elienne $p$-ramifi\'ee maximale $H_k^{\rm pr}$ de $k$ et soit $p^{e}$
l'exposant de $U_k^*/ \overline E_k$. 
Soit $H_{k_n/k}$ le $p$-corps des genres de $k_n$ (cf. \S\,\ref{seg}).

\smallskip
(i) On a $[H_{k_n/k} : k_n] = \order  {\mathcal T}_k$, pour tout $n\geq {e}$. 

\smallskip
(ii) On a $k_\infty H_{k_{e}/k} = H_k^{\rm pr}$ et  l'extension $H_k^{\rm pr}/ k_\infty$ 
est non ramifi\'ee.\footnote{Ce r\'esultat est d\'emontr\'e dans 
\cite[Theorem 1.1, Lemma 2.3]{T2} par introduction du corps des genres. 
Voir aussi les approches de \cite{J6} et \cite{Ng3}.}

\smallskip
(iii) Soit $\Lambda \supseteq E_k $ un sous-groupe de $k^\times$
tel que tout $x \in \Lambda$ soit norme locale dans $k_n/k$ en dehors 
de $p$, pour un entier $n \geq 0$ donn\'e (i.e., $(x)$ norme d'un id\'eal de $k_n$). 
Alors $\ds \frac{p^{n \cdot (d -1)}}{(\Lambda : \Lambda \cap {\rm N}_{k_{n}/k}( k_{n}^\times))}$
divise $R_k$. 

Pour $E_k \subseteq \Lambda \subseteq E_k^{S_k}$, cette divisibilit\'e a lieu pour tout $n$.
\end{theorem}

\begin{proof}
(i) Consid\'erons les $p$-corps des genres $H_{k_e/k}$ et $H_{k_{n}/k}$~; 
on a $H_{k_e/k} k_{n}\subseteq H_{k_{n}/k}$ en raison 
de la totale ramification de $p$ dans $k_\infty/k$ et de l'Ab\'elianit\'e 
de $H_{k_e/k} k_{n}/k$. D'o\`u la valeur
stationnaire du degr\'e $[H_{k_{n}/k} : k_{n}]$ puisque $[H_{k_n/k} : k_n] = 
\order  \Cl_{k_n}^{G_n}$ et que $\order  \Cl_{k_n}^{G_n} =  \order  {\mathcal T}_k$ pour tout $n \geq e$
d'apr\`es le Th\'eor\`eme \ref{regulateur} (ii). D'o\`u (i) et par cons\'equent (ii).

\smallskip
(iii) On a les injections canoniques
$$E_k/E_k \cap {\rm N}_{k_{n}/k}(k_{n}^\times) \hookrightarrow \Lambda /\Lambda  \cap 
{\rm N}_{k_{n}/k}(k_{n}^\times)  \hookrightarrow \Omega(k_{n}/k) \simeq (\Z/p^n \Z)^{d-1},$$
car les symboles de Hasse en $p$ (pour $k_n/k$) des $x \in \Lambda$ v\'erifient la formule du produit 
par hypoth\`ese.

\smallskip
Donc $\ds \frac{p^{n \cdot (d -1)}} {(\Lambda : \Lambda \cap {\rm N}_{k_{n}/k}( k_{n}^\times))}\,$ 
divise $\, \ds \frac{p^{n \cdot (d -1)}} {(E_k : E_k \cap {\rm N}_{k_{n}/k}( k_{n}^\times))},\,$ 
lequel divise $R_k$ (Th\'eor\`eme \ref{regulateur} (i)).
\end{proof}

\begin{remark}\label{rematheta}
 La th\'eorie des genres dit que  l'image par $\pi_n$ de $\Omega(k_n/k)$ est 
${\rm Gal}(H_{k_n/k}/k_n H_k)$ et que le noyau de $\pi_n$ est
$\omega_n(E_k) \subseteq \Omega(k_n/k), $
o\`u $\omega_n(\varepsilon) = \big( \big( \frac{\varepsilon \, ,\, k_n/k}
{{\mathfrak p}} \big) \big)_{{\mathfrak p} \mid p}$ pour tout $\varepsilon \in E_k$
et  $\order  \omega_n(E_k) = (E_k : E_k \cap {\rm N}_{k_n/k}( k_n^\times))$~;
en introduisant le sous-groupe ${\mathcal N}_k^n$ des $x \in k^\times$ partout 
norme locale dans $k_n/k$ en dehors de $p$, on a
$\Omega(k_n/k) = \omega_n({\mathcal N}_k^n) \simeq (\Z/p^n \Z)^{d-1}$, 
et $\pi_n(\Omega(k_n/k)) = \pi_n \circ \omega_n ({\mathcal N}_k^n) = 
{\rm Gal}(H_{k_n/k}/k_n H_k)$. Ainsi, pour tout $n \geq 0$ fix\'e, l'application~:
$$\theta_n := \pi_n \circ \omega_n\  : \  {\mathcal N}_k^n 
\mathop{\tooo}^{\omega_n} \Omega(k_n/k) \mathop{\tooo}^{\pi_n} {\rm Gal}(H_{k_n/k}/k_n H_k)$$

est surjective de ${\mathcal N}_k^n$ sur ${\rm Gal}(H_{k_n/k}/k_n H_k)$ et
de noyau $E_k \cdot {\rm N}_{k_n/k}(k_n^\times)$.

\smallskip
Par cons\'equent, l'image par $\theta_n$ d'un sous-groupe $\Lambda$ de ${\mathcal N}_k^n$ 
contenant $E_k$ est un sous-groupe de ${\rm Gal}(H_{k_n/k}/k_n H_k)$ isomorphe 
\`a $\omega_n(\Lambda)/\omega_n(E_k)$, o\`u $\order  \omega_n(\Lambda) = 
(\Lambda : \Lambda \cap {\rm N}_{k_n/k}( k_n^\times))$. 
Autrement dit, on a la suite exacte~:
\begin{equation*}
\begin{aligned}
1 \to E_k/E_k \cap {\rm N}_{k_n/k}( k_n^\times) \tooo 
\Lambda / \Lambda \cap  {\rm N}_{k_n/k} & ( k_n^\times) \\ 
\mathop{\tooo}^{\theta_n} \theta_n  (\Lambda) & \subseteq {\rm Gal}(H_{k_n/k}/k_n H_k) \to 1.
\end{aligned} 
\end{equation*}

L'image par $\pi_n$ d'un sous-groupe de $\bigoplus_{{\mathfrak p} \mid p} 
I_{\mathfrak p}(k_n/k)$, non contenu dans $\Omega(k_n/k)$, n'est pas 
n\'ecessairement dans ${\rm Gal}(H_{k_n/k}/k_n H_k)$ par absence de la formule 
du produit sur les $p$-places.
\end{remark}

\begin{theorem} \label{theta}  Soit ${\mathcal N}_k^n$, pour $n\geq 0$ fix\'e,
le sous-groupe des $x \in k^\times$ partout norme locale en dehors 
de $p$ dans $k_n/k$ (i.e., $(x)$ norme d'un id\'eal de $k_n$) et soit 
$\theta_n := \pi_n \circ \omega_n$ (cf. \S\,\ref{seg} et Remarque \ref{rematheta}).

\smallskip
(i) Pour tout sous-groupe $\Lambda$ de ${\mathcal N}_k^n$ contenant 
$E_k$, $\theta_n (\Lambda)$ est isomorphe \`a un sous-groupe de 
${\rm Gal}(H_{k_n/k}/k_n H_k)$, lequel est isomorphe \`a un quotient de
${\rm Gal}(H_k^{\rm pr} /k_\infty H_k)  \simeq U_k^*/ \overline E_k  
\subseteq {\mathcal T}_k$ et $\order  \theta_n (\Lambda)$ divise $R_k$. 

\smallskip
Pour $r \leq n$,  $\order \theta_r (\Lambda)$ divise $\order \theta_n(\Lambda)$.

\smallskip
(ii) Si $n \geq e$, o\`u $p^e$ est l'exposant de $U_k^*/\overline E_k$, alors
$\theta_n(\Lambda) \simeq \theta_e(\Lambda)$.

\smallskip
(iii) Si $E_k \subseteq \Lambda \subseteq E_k^{S_k}$, alors les $\order \theta_n (\Lambda)$
forment une $n$-suite croissante (stationnaire \`a partir de $n=e$) de diviseurs de $R_k$.
\end{theorem}

\begin{proof}
Soit $r \leq n$. Comme $H_{k_{r}/k} \subseteq H_{k_{n}/k}$, pour ${\mathfrak p} \mid p$,
les symboles $\ds \Big(\frac{x \, ,\, H_{k_{r}/k}/ H_k}{{\mathfrak p}} \!\Big)$
sont les restrictions des $\ds\Big(\frac{x \, ,\, H_{k_{n}/k} / H_k} {{\mathfrak p}} \Big)$,
et l'image de $\theta_{n}(x) =  \prd_{{\mathfrak p} \mid p} \ds
\Big(\frac{x \, ,\, H_{k_{n}/k}/ H_k}{{\mathfrak p}} \!\Big)$ dans la surjection canonique~:
$${\rm Gal}(H_{k_{n}/k} /k_{n} H_k) \too \hspace{-0.5cm} \too  {\rm Gal}(H_{k_{r}/k}/k_{r} H_k)$$

est $\theta_{r}(x)$, et par cons\'equent on a une surjection $\theta_{n}(\Lambda) 
\too \hspace{-0.5cm} \too \theta_{r}(\Lambda)$,
d'o\`u la relation entre les ordres.
Si $n \geq e$, la surjection devient l'isomorphisme~:
$${\rm Gal}(H_{k_{n}/k} /k_{n} H_k) \simeq {\rm Gal}(H_{k_e/k}/k_e H_k)$$ 

du Th\'eor\`eme \ref{constant} (i), et $\theta_n(\Lambda)$ est isomorphe \`a  
$\theta_{e}(\Lambda)$ d'o\`u le r\'esultat. 

\smallskip
Si $E_k \subseteq \Lambda \subseteq E_k^{S_k}$, 
ce qui pr\'ec\`ede est valable quel que soit $n$
\end{proof}

\begin{remarks}\label{tor=1}
 (i) Certains r\'esultats de th\'eorie d'Iwasawa donnent des 
isomorphismes au niveau infini mettant en jeu le groupe ${\mathcal T}_k$ 
(cf. \cite[Lemma 4.7]{LMN},  \cite{T1},  \cite{T2}, \cite{OT}, \cite[Lemme 3.1]{BaN}), mais
en r\'ealit\'e, il y a r\'egularisation \`a un niveau fini explicite ne d\'ependant que
du r\'egulateur $p$-adique de $k$. 

\smallskip
(ii) Si ${\mathcal T}_k = 1$ (i.e., $k$ est $p$-rationnel, \cite[\S\,IV.3]{Gra1}, 
\cite{GJ}, \cite{JN}, \cite{MN}), on a \'evidemment 
$\lambda = \mu = \nu = 0$~; ceci s'applique par exemple au
corps cubique \'etudi\'e dans \cite{T3} pour $p=5$.

\smallskip
On trouvera dans \cite{Gra6} une \'etude d\'etaill\'ee des r\'egulateurs $p$-adiques
qui repr\'esentent le facteur crucial puisque $\Cl_k$ est non trivial uniquement pour un 
nombre fini de $p$ tandis que c'est seulement conjectur\'e pour $R_k$.

\smallskip
(iii) On a ${\mathcal T}_k = 1$ si et seulement si $\Cl_k = R_k = 1$, et alors
(sous la conjecture de Leopoldt dans la tour), la formule de points fixes donnant 
$\order  {\mathcal T}_{k_n}^{G_n}$ (\cite[Th\'eor\`eme IV.3.3, \S\,IV\,(b)]{Gra1},
\cite{MN}) conduit \`a ${\mathcal T}_{k_n}=1$ car $k_n/k$ est ``$p$-primitivement 
ramifi\'ee'' ($p$-rationalit\'e dans la tour). D'o\`u $\Cl_{k_n} = 1$ et 
$R_{k_n} = 1$ pour tout $n \geq 0$.

\smallskip
(iv) Soit $p^{\wt e}$ l'exposant de ${\mathcal T}_k$.
Soit $I_k$ le groupe des id\'eaux de $k$ \'etrangers \`a $p$ et, pour $n\geq \wt e$, soit~:
$$\wt E_k^n := \{x \in k^\times, \,  (x) \in I_k^{p^n}\}. $$

Alors on peut \'etablir 
(en utilisant les techniques du \S\,\ref{resglob}) une suite exacte de la forme~:
$$1 \too \{x \in k^\times, \, \hbox{$x$ \'etranger \`a $p$} \}^{p^n}
\tooo \wt E_k^n \cap {\rm N}_{k_n/k}( k_n^\times) \ds
\mathop{\tooo}^{\wt \psi_n} {\mathcal T}_k \too 1. $$
\end{remarks}

\begin{remarks}
(i) Une $S_k$-unit\'e est norme locale en dehors de $p$ dans $k_n/k$ pour tout $n\geq 0$, 
tandis que pour $x \notin E_k^{S_k}$ ceci n'est pas possible~;
en effet, supposons $(x)$ norme d'un id\'eal ${\mathfrak A}$ de $k_n$, pour tout $n$, 
et posons~:
$$(x) = {\rm N}_{k_n/k}({\mathfrak A})=
 {\mathfrak a} \cdot  \prd_{{\mathfrak p} \mid p} {\mathfrak p}^{c_{\mathfrak p}}, $$

${\mathfrak a}$ \'etranger \`a $p$, $c_{\mathfrak p} \geq 0$~; 
si ${\mathfrak l}^{c_{\mathfrak l}}$ est la composante 
${\mathfrak l}$-primaire de ${\mathfrak a}$ pour l'id\'eal premier ${\mathfrak l}$, pour un id\'eal 
premier ${\mathfrak L} \mid {\mathfrak l}$ de $k_n$, il existe $\tau_n \in \Z[G_n]$ tel que~:
$${\mathfrak l}^{c_{\mathfrak l}} = {\rm N}_{k_{n}/k}({\mathfrak L}^{\tau_n}) =
{\mathfrak l}^{f_{\mathfrak l}^n \cdot \tau_n(1)}, $$ 

o\`u $f_{\mathfrak l}^n$ est le degr\'e r\'esiduel de ${\mathfrak l}$ dans $k_{n}/k$ 
et $\tau_n(1) \in \Z$ est l'image de $\tau_n$ par l'application d'augmentation. 
Les $f_{\mathfrak l}^n$ \'etant strictement 
croissants pour $n \gg 0$, ceci est impossible sauf si 
$\tau_n(1)=0$, auquel cas $c_{\mathfrak l} = 0$ et $x$ est une $S_k$-unit\'e.

\smallskip
Par cons\'equent, si $x \notin E_k^{S_k}$, il n'y a pas de formule du produit 
portant uniquement sur les $p$-places dans $k_n/k$ pour $n$ au-del\`a
d'une certaine valeur d\'ependant des degr\'es r\'esiduels des ${\mathfrak l} \mid (x)$
dans $k_\infty/k$.

\smallskip
(ii) Le cadre de la conjecture de Greenberg et l'introduction de la filtration des groupes de 
classes $\Cl_{k_n}$ par les $\cl_{k_n}({\mathcal I}_i^n)$ (${\mathcal I}_i^n$ \'etrangers \`a $p$)
qui sont les sous-groupes des classes annul\'ees par $(1-\sigma_n)^i$, $0 \leq i \leq m_n-1$,
mettra en jeu les groupes
$\Lambda_i^n = \{ x \in k^\times, \ \ (x) \in  {\rm N}_{k_{n}/k} ({\mathcal I}_i^n) \}$
et le calcul des 
$(\Lambda_i^n : \Lambda_i^n \cap {\rm N}_{k_{n}/k} (k_{n}^\times))$. Ceci fait 
l'objet de la Section \ref{quant}.

\smallskip
Donc $\Lambda_i^n$ ne contiendra pas de $S_k$-unit\'es autres que les unit\'es et
l'aspect algorithmique se ram\`enera \`a l'\'etude des $x \in \Lambda_i^n$ et de leurs 
symboles normiques donn\'es par leurs quotients de Fermat en les ${\mathfrak p} \mid p$, 
ce qui constitue une approche diff\'erente de la conjecture de Greenberg. 
\end{remarks}

\section{Aspects num\'eriques -- Corps quadratiques r\'eels}\label{quad}\label{sect5}

Soient ${\mathfrak p}_1$ et ${\mathfrak p}_2$ les deux id\'eaux premiers de $k=\Q(\sqrt m)$
au-dessus de $p$ et $S_k$ leur ensemble. Les $S_k$-unit\'es g\'en\'eratrices modulo les unit\'es
sont donn\'ees par $\pi_1$ et sa $\Q$-conjugu\'ee $\pi_2$, et sont telles que 
$(\pi_j) = {\mathfrak p}_j^{h_0}$ o\`u $h_0$ est l'ordre de la classe des ${\mathfrak p}_j$, $j=1, 2$.

\begin{lemma}\label{norme1} Soit $\alpha \in k^\times$ \'etranger \`a $p$ tel que
${\rm N}_{k/\Q}(\alpha)^{p-1} \equiv 1\!\! \pmod {p^{n+1}}$, pour $n \geq 1$~; 
alors on a (cf.  D\'efinition \ref{delta})
$\delta_{\mathfrak p_1}(\alpha) = \delta_{\mathfrak p_2}(\alpha)$ ou bien 
$\delta_{\mathfrak p_1}(\alpha) \geq n\ \  \&\ \  \delta_{\mathfrak p_2}(\alpha) \geq n$.
Donc si $\alpha^{p-1}\not \equiv 1 \pmod {p^{n+1}}$, on a
$\alpha^{p-1} = 1 +p \cdot p^{\delta_p(\alpha)} \! \cdot \beta$, avec
$\delta_p(\alpha)<n$ et $\beta$ \'etranger \`a $p$.
\end{lemma}

\begin{proof} 
Soient $\omega_j \in k^\times$, $j=1, 2$, deux nombres conjugu\'es tels que
$v_{ {\mathfrak p}_j}(\omega_j) = 1$ et $\omega_1 \cdot \omega_2 = p \,\beta'$,
avec $\beta'$ \'etranger \`a $p$, et posons
$\alpha^{p-1} =: \alpha_1^{p-1} = 1 +p \cdot \omega_1^{\delta_{\mathfrak p_1}(\alpha)}
\omega_2^{\delta_{\mathfrak p_2}(\alpha)} \!\!\cdot \beta_1$, $\beta_1$ \'etranger \`a $p$,
et soit $\alpha_2$ le conjugu\'e de $\alpha_1$.
On a ${\rm N}_{k/\Q}(\alpha)^{p-1} = (\alpha_1\cdot \alpha_2)^{p-1}
 \equiv 1 \pmod{p^{n+1}}$~; donc on obtient~:
\begin{equation*}
\begin{aligned}
\omega_1^{\delta_{\mathfrak p_1}(\alpha)} \!\cdot 
\omega_2^{\delta_{\mathfrak p_2}(\alpha)} \!\cdot \beta_1
+  \omega_2^{\delta_{\mathfrak p_1}(\alpha)}\! \cdot &
\omega_1^{\delta_{\mathfrak p_2}(\alpha)} \!\cdot \beta_2 \\
& + p^{1+\delta_{\mathfrak p_1}(\alpha) + \delta_{\mathfrak p_2}(\alpha)} 
\cdot \beta_1\,\beta_2 \,\beta'' \equiv 0 \pmod {p^{n}} ;
\end{aligned} 
\end{equation*}

si $\delta_{\mathfrak p_1}(\alpha) < n$ ou $\delta_{\mathfrak p_2}(\alpha) < n$, il vient
n\'ecessairement $\delta_{\mathfrak p_1}(\alpha) = \delta_{\mathfrak p_2}(\alpha)$.
\end{proof}

\begin{remark}\label{remunit}
Dans cette situation, pour les corps quadratiques, les $\delta_{\mathfrak p}(\alpha)$ 
seront not\'es $\delta_p(\alpha)$. Ceci vaut pour un $\alpha  \in k^\times$ \'etranger \`a 
$p$ et partout norme locale en dehors de $p$ dans $k_n/k$ (i.e., norme d'un id\'eal de $k_n$)~; 
en effet, le conducteur de $\Q_n/\Q$ \'etant $p^{n+1}$, on a ${\rm N}_{k/\Q}(\alpha)^{p-1} 
\equiv 1\pmod {p^{n+1}}$.

\smallskip
Pour une unit\'e $\varepsilon$ de $k$, on a (ind\'ependamment de $n$) 
$\varepsilon^{p-1} = 1 +p \cdot p^{\delta_p(\varepsilon)} \! \cdot \beta$, $\delta_p(\varepsilon) \geq 0$, 
avec $\beta$ \'etranger \`a $p$.
\end{remark}

\subsection{Calcul pratique des symboles normiques des $S_k$-unit\'es} \label{delta1}

Pour le calcul des $\Big(\frac{\pi_j \, ,\, k_n/k}{{\mathfrak p}_1} \Big)$ et 
$\Big(\frac{\pi_j \, ,\, k_n/k}{{\mathfrak p}_2} \Big)$,
$j= 1, 2$, on remarque que les $\pi_j$ sont normes locales en dehors de $p$ et que la formule du 
produit permet de ne calculer que les $\Big(\frac{\pi_j \, ,\, k_n/k}{{\mathfrak p}_1} \Big)$ par exemple~; 
mais l'action de ${\rm Gal}(k/\Q)=: \{1, \tau\}$ montre que
$\Big(\frac{\tau(\pi_j) \, ,\, k_n/k}{\tau({\mathfrak p}_1)} \Big) = 
\tau  \Big(\frac{\pi_j \, ,\, k_n/k}{{\mathfrak p}_1} \Big) \tau^{-1} = 
\Big(\frac{\pi_j \, ,\, k_n/k}{{\mathfrak p}_1} \Big)$.

On est donc ramen\'e au seul calcul de $\Big(\frac{\pi_2 \, ,\, k_n/k}{{\mathfrak p}_1} \Big)$ avec 
$\pi_2$ \'etranger \`a ${\mathfrak p}_1$. On a \`a r\'esoudre le syst\`eme de congruences
d\'efinissant un ${\mathfrak p}_1$-associ\'e $x'$ de $x=\pi_2$
(pour $n > \delta_{{\mathfrak p}_1}(\pi_2)$)~:
\begin{equation}
\begin{aligned}
& x'  \equiv \pi_2 \pmod {\pi_1^{n+1}}, \\
& x' \equiv 1 \pmod {\pi_2^{n+1}},
\end{aligned}
\end{equation}

On d\'etermine une ``relation de B\'ezout''
$U_1\cdot \pi_1^{n+1} + U_2\cdot \pi_2^{n+1} = 1$, o\`u $U_1, U_2 \in Z_{k, (p)}$ 
(anneau des $p$-entiers de $k$), ce qui conduit \`a la solution 
$x' = U_1\cdot \pi_1^{n+1} + U_2\cdot \pi_2^{n+1} \cdot \pi_2 \pmod {p^{n+1}}$.
On a ${\mathfrak a}_{{\mathfrak p}_1}(x) = (x')$ dont on prend la 
norme dans $k/\Q$ pour caract\'eriser le symbole d'Artin pour $k_n/k$~;
son ordre, dans $(\Z/p^{n+1} \,\Z)^\times$, est
de la forme $p^{n-\delta_{{\mathfrak p}_1}(\pi_2)}$. 

\smallskip
Le cas de $\varepsilon$ est identique \`a partir du ${\mathfrak p}_1$-associ\'e 
$x'' = U_1\cdot \pi_1^{n+1} + U_2\cdot \pi_2^{n+1} \cdot \varepsilon \pmod {p^{n+1}}$.

\subsection{Programmes PARI}\label{prog}
Le programme ci-dessous (d'apr\`es \cite{P}) fournit les informations suivantes~:

\smallskip
$m$,  $h=$ nombre de classes de $k$,  $\varepsilon  = u + v \, \sqrt m =$ unit\'e fondamentale de $k$,
$u, v \in \Z$ ou $\frac{1}{2}\,\Z$, $S_k =: \{\pi_1, \,\pi_2\}$,
$z_\pi =p^{-\delta_{{\mathfrak p}_1}(\pi_2)}$ et $z_\varepsilon=p^{-\delta_p(\varepsilon)}$,
et aussi $n_\pi$ et $n_\varepsilon$ qui figurent les symboles de Hasse de $\pi_2$ et $\varepsilon$
dans $\Omega(k_n/k)$.

\smallskip
D'apr\`es le Th\'eor\`eme \ref{thm2}, la conjecture de Greenberg est v\'erifi\'ee d\`es que 
$\cl_k(S_k)=\Cl_k$ et que l'un au moins des nombres $z_\pi$ ou $z_\varepsilon$ est \'egal \`a $1$ 
(i.e., $\delta_{{\mathfrak p}_1}(\pi_2)=0$ ou $\delta_p(\varepsilon)=0$)
car le symbole de Hasse correspondant engendre ${G_n}$ pour tout
$n$ et de fait on a $\Omega(k_n/k) \simeq G_n$
(si $z_\varepsilon = 1$, c'est que le r\'egulateur normalis\'e 
$R_k \sim \frac{1}{p}\,{\rm log}(\varepsilon)$ est \'egal \`a $1$). 
En pratique le programme utilise un $n$ de l'ordre de $8$ qui suffit largement dans tous les
r\'esultats num\'eriques obtenus pour avoir les {\it valeurs exactes }de
$z_\pi$ et $z_\varepsilon$, mais d'apr\`es la Proposition \ref{basetopo}, il suffit de prendre $n_0=1$
(donc $n=n_0+1=2$ et des calculs modulo $p^2$ seulement) pour obtenir tous les cas 
o\`u le test est positif (i.e., $z_\pi=1$ ou $z_\varepsilon=1$).

\smallskip
Les mentions `` {\footnotesize PROBL\`EME-NORMIQUE} '' 
(resp. `` {\footnotesize PROBL\`EME-CLASSES} '') signifient $z_\pi < 1\ \&\ 
z_\varepsilon < 1$ (resp. $\cl_k(S_k) \ne \Cl_k$). On a $p \geq 3$.
\footnotesize
\begin{verbatim}
{for(j=2, 10, p=prime(j); m=1; n0=8; n=n0+1; while(m<10^4, m=m+1;
if(core(m)==m & kronecker(m, p)==1, y=x; Q=x^2-m; K= bnfinit(Q,1);
M=m; t=Mod(m,4); if(t!=1, M=4*m); E=quadunit(M);
h=qfbclassno(M); e1=component(E,2); e2=component(E,3);
if(t==1, e2=e2/2; e1=e1+e2); E=e1+e2*x; print(" "); 
print("m = ", m," h = ",h," E = ",E);
Su=bnfsunit(K,idealprimedec(K,p));
pi1=component(component(Su,1),1);
pi2=component(pi1,2)*x-component(pi1,1);
print("p = ",p," S = ",pi1," ",pi2); Pi1= pi1^n; Pi2= pi2^n;
Z=bezout(Pi1,Pi2); U1=component(Z,1); U2=component(Z,2);
P=y^2-Mod(m,p^n); Y=Mod(y,P); x=Y;
A1=eval(U1); A2= eval(U2); B1= eval(Pi1); B2= eval(Pi2);
b1= eval(pi1); b2= eval(pi2); e=eval(E);
XPpi=Mod(A1*B1+A2*B2*b2,P); XPe=Mod(A1*B1+A2*B2*e,P);
hs=norm(Mod(pi1,Q)); h0=valuation(hs,p);
delta=valuation(h,p)-valuation(h0,p);
npi=norm(XPpi)^(p-1); ne=norm(XPe)^(p-1);
zpi=znorder(npi)/p^n0; ze=znorder(ne)/p^n0;
if(zpi+ze <1, print("PROBLEME-NORMIQUE"));
if(delta!=0, print("PROBLEME-CLASSES")); print(zpi," ",ze); x=y)))}
\end{verbatim}

\normalsize
Donnons l'extrait suivant pour $30007 \leq m \leq 30097$, $m \equiv 1 \pmod 3$, et $p=3$
(la derni\`ere colonne donne la structure du $3$-groupe des classes du premier
\'etage $k_1$ de la $\Z_3$-extension cyclotomique de $k$)~:
\footnotesize
$$\begin{array}{ccclllccclll}
&     m        &   h &   z_\pi &  z_\varepsilon  &  \Cl_{k_1}  & \hspace{1cm} 
m &   h &   z_\pi &  z_\varepsilon  &  \Cl_{k_1} \vspace{0.1cm} \\  
&  30001    &  1  &  1  &  1 &  1  &  \hspace{1cm} 30055   &   2  & 1/27 &  1/27 & \Z/3\Z \\
&  30007   &   2  & 1/9 &  1/3 &  \Z/3\Z \times \Z/3\Z  &  \hspace{1cm} 30058   &   4  & 1  & 1 & 1 \\
&  30010   &  8  &  1 &  1 &  1  &  \hspace{1cm} 30061   &   1  & 1  & 1 & 1 \\
&  30013   &   1  &  1/3 &  1 &  1  &  \hspace{1cm} 30067   &   2  & 1  & 1 &  1 \\
&  30019   &   4 &  1  & 1/3 & \Z/3\Z  &  \hspace{1cm} 30070  &    4 &  1  & 1/3 &  \Z/3\Z \\
&  30022   &   4 &  1/3 &  1 & 1    &  \hspace{1cm} 30073   &   4 &  1 &  1/27 &\Z/3\Z  \\
&  30031   &   2  & 1/3 &  1/3 &  \Z/9\Z \times \Z/3\Z   &  \hspace{1cm} 30079   &   2 &  1  & 1 & 1 \\
&  30034   &   2  & 1 &  1 &  1   &  \hspace{1cm} 30085   &   2 &  1/3  & 1 & 1 \\
&  30043   &   18  & 1 & 1 &  \Z/9\Z     &  \hspace{1cm} 30091   &   1 &  1  & 1 & 1 \\
&  30046  &    2  & 1 &  1 &  1  &  \hspace{1cm} 30094  &    8  & 1  & 1/3 & \Z/3\Z \\
&  30049  &    1 &  1 &  1/3 & \Z/3\Z   &  \hspace{1cm} 30097  &    1  & 1  & 1&  1\\
\end{array} $$

\normalsize
Si le nombre de classes est divisible par $p$, il faut v\'erifier si le $p$-groupe des classes de $k$
est engendr\'e par les id\'eaux premiers au-dessus de $p$, sinon la conclusion n'est pas valable. Pour cela
le programme retient l'ordre $h_0$ de la classe de ${\mathfrak p}_1$ pour lequel
${\mathfrak p}_1^{h_0} = (\pi_1)$~; si les valuations 
$p$-adiques de $h_0$ et du nombre de classes $h$ sont \'egales, ceci veut dire que $\Cl_k$ est 
cyclique et engendr\'e par la $p$-classe de ${\mathfrak p}_1$. Par exemple,
dans le cas de $m= 30043$ o\`u $h=18$, la $S_k$-unit\'e g\'en\'eratrice est $317+2\, \sqrt m$, de norme
$3^9$, et par cons\'equent la classe de ${\mathfrak p}_1$, d'ordre $h_0=9$, est g\'en\'eratrice de $\Cl_k$.

\smallskip
Selon le choix par PARI d'une $S_k$-unit\'e fondamentale modulo $E_k$, la valeur 
de $z_\pi$ n'est pas intrins\`eque, mais le test normique est invariant.

\smallskip
La m\'ethode est extr\^emement simple et le programme tr\`es rapide pour n'importe quel $p$~; 
pour les $22794$ valeurs de $m$ inf\'erieures \`a $10^5$, on a  
$19993$ valeurs pour lesquelles on peut conclure que $\lambda_3 =\mu_3 =0$.
Mais d\`es que $p$ est un peu grand le test est presque toujours positif et montre que 
$\lambda_p = \mu_p =0$.

\smallskip
Pour chaque $p$, $3\leq p \leq 541$, on obtient le tableau ci-apr\`es
selon le programme de comptage ci-dessous indiquant successivement
le nombre $C_1$ de $m \leq 10^4$ (tels que $p$ soit d\'ecompos\'e dans $\Q(\sqrt m)$), 
le nombre $C_2$ de cas donnant $\lambda_p = \mu_p = 0$, et $C_1-C_2$ (cas non r\'esolus)~:
\footnotesize 
\begin{verbatim}
{for(j=2, 100, C1=0; C2=0; m=1; p=prime(j); n0=1; n=n0+1;
while(m<10^4, m=m+1; if(core(m)==m & kronecker(m,p)==1, C1=C1+1;
y=x; Q=x^2-m; K=bnfinit(Q,1); M=m; t=Mod(m,4); if(t!=1,M=4*m);
E=quadunit(M); h=qfbclassno(M); e1=component(E,2); e2=component(E,3);
if(t==1, e2=e2/2; e1=e1+e2); E=e1+e2*x; Su=bnfsunit(K,idealprimedec(K,p));
pi1=component(component(Su,1),1); pi2=component(pi1,2)*x-component(pi1,1);
Pi1=pi1^n; Pi2=pi2^n; Z=bezout(Pi1, Pi2);
U1=component(Z,1); U2=component(Z,2); P=y^2-Mod(m,p^n);
Y=Mod(y,P); x=Y; A1=eval(U1); A2=eval(U2);
B1=eval(Pi1); B2=eval(Pi2); b1=eval(pi1); b2=eval(pi2);
e=eval(E); XPpi=Mod(A1*B1+A2*B2*b2,P); XPe=Mod(A1*B1+A2*B2*e,P);
hs=norm(Mod(pi1,Q)); h0=valuation(hs, p);
delta=valuation(h,p)-valuation(h0,p);
npi=norm(XPpi)^(p-1); ne=norm(XPe)^(p-1);
zpi=znorder(npi)/p^n0; ze=znorder(ne)/p^n0;
if(zpi+ze>1&delta==0, C2=C2+1); x=y)); print(p," ",C1," ",C2," ",C1-C2))}
\end{verbatim}
$$\begin{array}{cccccccccc}
 p &  C_1 &  C_2 &  C_1-C_2   & \hspace{1cm}  p &  C_1 &  C_2 &  C_1-C_2 \vspace{0.1cm} \\ 
3   &   2279   &    2042   &    237   & \hspace{1cm}  67  &    2993   &    2993   &    0  \\
5  &    2534   &    2459    &   75   & \hspace{1cm}  71  &    2994   &    2994   &    0  \\
7   &   2660  &     2599    &   61   & \hspace{1cm}  73  &    3001   &    3001   &    0  \\
11   &   2781   &    2759   &    22   & \hspace{1cm}  79  &    3001   &    3000   &    1  \\
13   &   2822   &    2808   &    14   &  \hspace{1cm} 83  &    3002   &    3002   &    0  \\
17  &    2873   &    2860   &    13   & \hspace{1cm}  89   &   3008   &    3007   &    1  \\
19   &   2886   &    2877   &    9   & \hspace{1cm}  97  &    3011   &    3011   &    0  \\
23 &     2908   &    2904   &    4   & \hspace{1cm}  101  &    3010   &    3009   &    1  \\
29 &     2936   &    2931   &    5   & \hspace{1cm}  103  &    3005   &    3004    &   1  \\
31  &    2944   &    2939   &    5   & \hspace{1cm}  \cdots &\cdots &\cdots &\cdots \\
37  &    2960   &    2958   &    2   & \hspace{1cm}  149  &    3020  &  3019  &  1 \\
41 &     2968    &   2967   &    1   & \hspace{1cm}  \cdots &\cdots &\cdots &\cdots \\
43  &    2971   &    2971   &    0   & \hspace{1cm}  193   &   3023  & 3022 &  1 \\
47  &    2971   &    2971   &    0   & \hspace{1cm}  197   &   3029 &  3028  & 1 \\
53  &    2983   &    2982   &    1   & \hspace{1cm}  \cdots &\cdots &\cdots &\cdots \\
59  &    2986   &    2984   &    2   & \hspace{1cm}  211   &   3027 &  3026 &  1 \\
61  &    2988   &    2988    &   0   & \hspace{1cm}   \cdots &\cdots &\cdots &\cdots 
\end{array} $$

\normalsize \smallskip
Les nombres $p$, $223 \leq p \leq 541$, ou absents du tableau, 
donnent $\lambda_p=\mu_p=0$ pour tous les 
$m \leq 10^4$ tels que $p$ soit d\'ecompos\'e dans $\Q(\sqrt m)$.

\smallskip
Dans \cite{Su} il y a deux exemples plus d\'elicats (pour $p=3$)~:

\smallskip
(i) $m = 2659$,  $h = 3$,   $\varepsilon = 63190881\, \sqrt m + 3258468890$,
$\pi_1 = -2\, \sqrt m + 103 $ qui dans notre table est indiqu\'e
avec $z_\pi= \frac{1}{3}$ et $z_\varepsilon=\frac{1}{3^2}$ ($\Cl_k$ est engendr\'e par $S_k$).
Il faut alors d'autres calculs explicites dans la tour pour d\'emontrer que $\lambda_3=\mu_3=0$
(cf. \cite{FT1}, \cite{FK}, \cite{IS1}, \cite{IS2} utilisant soit le ``Spiegelungssatz'' dans
$k(j)$ ($j^3=1,\ j \ne 1$), soit les unit\'es cyclotomiques).

\smallskip
(ii) \!$m =\!12007$, $h = 3$, $\varepsilon = 65199591367431507 \sqrt m + 7144340241111277688$,
$\pi_1 = 429331\,  \sqrt m + 47044570$,
avec $z_\pi=\frac{1}{3^6}$ et $z_\varepsilon=\frac{1}{3^2}$ ($\Cl_k$ est engendr\'e par $S_k$). 
Dans ce cas, la v\'erification utilise les fonctions $L$ $p$-adiques, de nombreux arguments 
et aussi \cite{IS1}, \cite{IS2}.

\smallskip
Voir d'autres raisonnements dans \cite{BJ}, \cite{CN}, \cite{F1}, \cite{FK}, \cite{FT1}, 
\cite{FT2}, \cite{I}, \cite{KS} (cas $p=3$ non d\'ecompos\'e), \cite{LMN}, \cite{Ni} 
(cas $p=2$), \cite{T1}, \cite{T2} (pour des corps cubiques totalement r\'eels et $p=3$), 
et bien d'autres. 

\medskip
Pour une table assez compl\`ete ($p= 3, 5, 7, 11, 13, 17, 19, 23, 29$), pri\`ere de se con\-necter \`a~:

\medskip\footnotesize
\centerline{\url{https://www.dropbox.com/s/tcqfp41plzl3u60/R}}

\normalsize \medskip
Il y a identit\'e des valeurs de $\delta_{{\mathfrak p}_1}(\pi_2)$ et
$\delta_p(\varepsilon)$, pour $p=3$, avec celles de la table de \cite{FT1}.

\medskip \normalsize
Le programme suivant calcule la structure des groupes des classes 
de $k$ et de $k_1$ pour $p=3$ et $b \leq m \leq B$, $m \equiv 1 \pmod 3$~:
\footnotesize 
\begin{verbatim}
{b=10^6; B=b+10^3; b=b-component(Mod(b,3),2)+1; m=b;
while(m<B, m=m+3; if(core(m)==m, K=bnfinit(x^2-m,1);
h=bnrinit(K,1); h=component(h,5);
R=component(polcompositum(x^3-3*x+1,x^2-m),1);
H=bnrinit(bnfinit(R,1),1); H=component(H,5);
print("m = ",m,"  h = ",h,"  structure = ", H)))}
\end{verbatim}

\normalsize
Ce qui donne les quelques exemples suivants avec $\Cl_{k_1} \ne 1$
(sous la forme $ [\order \Cl_{k_1}, [\hbox{structure}]]$)~:
\footnotesize
$$\begin{array}{rllrll}
 m \ \ \ \  \  & \  \Cl_k &\   \Cl_{k_1}    & \hspace{1cm}  m \ \ \ \ \ & \  \Cl_k &\   \Cl_{k_1}  \vspace{0.15cm} \\ 
1000003   &   [3, [3]]  &   [27, [3, 3, 3]]   &  \hspace{1cm} 1000126 &   [2, [2]]   &  [24, [6, 2, 2]]  \\
1000018  &    [4, [4]]   &  [12, [12]]    &  \hspace{1cm} 1000135  &    [4, [2, 2]]   &  [12, [6, 2]]  \\
1000042  &    [36, [18, 2]]  &   [36, [18, 2]]   & \hspace{1cm}  1000147  &    [30, [30]]  &   [90, [30, 3]]  \\
1000051 &     [4, [2, 2]]    &   [12, [6, 2]]     & \hspace{1cm}  1000159 &   [1  ]   &   [3, [3]]  \\
1000093  &    [12, [6, 2]]  &   [48, [6, 2, 2, 2]]     & \hspace{1cm}  1000177  &    [1  ]    &   [12, [6, 2]]  \\
1000099  &    [1  ]     &   [3, [3]]  \   &  \hspace{1cm} 1000189  &    [1 ]   &   [9, [3, 3]]  \\
1000102  &    [2, [2]]    &   [18, [6, 3]]    &  \hspace{1cm} 1000198  &  [12, [12]]  &   [36, [36]] \\ 
(......)   &  &  &  &  &  \vspace{0.05cm}  \\
100000006 &   [12, [12]] &   [36, [36]]   &  \hspace{1cm} 100000027  & [6, [6]] & [72, [18, 2, 2]]  
\end{array} $$
\normalsize

\section{Filtration des $\Cl_{k_n}$ -- 
Groupes de nombres $\Lambda_i^n$}\label{quant}\label{sect6}
On revient au cas g\'en\'eral d'un corps de nombres Galoisien r\'eel $k$ de degr\'e $d$ et
o\`u $p>2$ est un nombre premier totalement d\'ecompos\'e dans $k$.
On suppose que la conjecture de Leopoldt est v\'erifi\'ee pour $p$ dans $k$.

\smallskip
Dans cette section nous reprenons l'analyse de la conjecture de Greenberg sous la forme 
directe du ``calcul'' du $p$-groupe des classes de $k_n$ selon l'algo\-rithme d\'efini dans 
\cite{Gra4}, et en utilisant des id\'eaux \'etrangers \`a $p$ pour repr\'esenter les classes.

\subsection{Introduction de la filtration des $\Cl_{k_n}$}
Si l'on pose pour simplifier $M^n := \Cl_{k_n}$, 
$k_n \subset k_\infty$ de degr\'e $p^n$ sur $k$ et de groupe de Galois $G_n =: 
\langle \sigma_n \rangle$, il existe une filtration ainsi d\'efinie~:

\begin{definition} \label{def1} Pour $n \geq 1$ fix\'e, $(M^n_i)_{i \geq 0}$ est la $i$-suite croissante 
de sous-${G_n}$-modules de $M^n$ d\'efinie (avec $M^n_0:= 1$) par
$M^n_{i+1}/M^n_i := (M^n/M^n_i)^{G_n}$, pour $0\leq i \leq m_n-1$,
o\`u $m_n$ est le plus petit entier $i$ tel que $M^n_i = M^n$.
\end{definition}

\begin{remarks}\label{filtration}
(i) Pour $i=0$, on obtient $M^n_1 = (M^n)^{G_n}$. 

\smallskip
(ii) On a $M^n_{i+1} = \{c \in M^n, \ c^{1-\sigma_n} \in M^n_i \}$~; ainsi~:
$$M^n_i= \{c \in M^n, \, c^{(1-\sigma_n)^i} =1 \}. $$

(iii) Pour $n$ fix\'e, la $i$-suite des $\order  (M^n_{i+1}/M^n_i)$, $0 \leq i \leq m_n$, est 
{\it d\'ecroissante} vers $1$ et major\'ee par $\order  M_1^n$  en raison des injections~:
$$M^n_{i+1}/M^n_i \hookrightarrow M^n_i/M^n_{i-1}
\hookrightarrow \cdots \hookrightarrow M^n_2/M^n_1 \hookrightarrow M_1^n$$

d\'efinies par l'op\'eration de $1-\sigma_n$.
\end{remarks}

Ensuite, pour les sous-${G_n}$-modules $M^n_i =: \cl_{k_n}({\mathcal I}_i^n)$ de $M^n$
(o\`u l'on peut toujours supposer que ${\mathcal I}_i^n \subseteq {\mathcal I}_{i+1}^n$),
on a la formule g\'en\'erale du Th\'eor\`eme \ref{thm1} qui devient dans notre cas particulier~:
\begin{equation}\label{eq5}
\order  \big( M^n_{i+1} / M^n_i \big) = \frac{\order  \Cl_{k}}{\order {\rm N}_{k_n/k}(M^n_i) }  \cdot 
\frac{p^{n \cdot (d -1)}}{(\Lambda_i^n :  \Lambda_i^n \cap {\rm N}_{k_n/k}(k_n^\times))},
\end{equation}

o\`u $\Lambda_i^n := \{x \in k^\times ,\ \, (x) \in {\rm N}_{k_n/k}({\mathcal I}_i^n) \}$
contient $E_k$, et o\`u tout $x \in \Lambda_i^n$ est par nature
norme locale en dehors de $p$ dans $k_n/k$. On a alors~:
\begin{equation}\label{prod}
\order  M^n = \prd_{i=0}^{m_n-1} \order  \big( M^n_{i+1}/M^n_i\big). 
\end{equation}

Pour $n$ fix\'e, les $M^n_i$ et ${\rm N}_{k_n/k}(M^n_i)$ d\'efinissent des $i$-suites 
{\it croissantes} de sous-groupes de $\Cl_{k_n}$ et $\Cl_k$ respectivement.
On obtient que les $i$-suites d'entiers ($i \geq 0$)~:

\centerline{$\ds \frac{\order \Cl_{k}}{\order {\rm N}_{k_n/k}(M^n_i)}\ \ $ et 
$\ \ \ds \frac{p^{n \cdot (d -1)}}{(\Lambda_i^n : \Lambda_i^n \cap {\rm N}_{k_n/k}(k_n^\times))}$, }

\smallskip
sont respectivement des $i$-suites {\it d\'ecroissantes vers $1$} de diviseurs 
de $\order \Cl_{k}$ et de $R_k$ (Th\'eor\`eme \ref{constant} (iii)), en raison, 
pour la seconde, des injections~: 
\begin{equation*}
\begin{aligned}
E_k/E_k \cap {\rm N}_{k_n/k}(k_n^\times)\hookrightarrow \cdots \hookrightarrow &
\Lambda_i^n/\Lambda_i^n \cap {\rm N}_{k_n/k}(k_n^\times) \\ & \hookrightarrow 
\Lambda_{i+1}^n/\Lambda_{i+1}^n \cap {\rm N}_{k_n/k}(k_n^\times)  \hookrightarrow \cdots
\end{aligned} 
\end{equation*}

les $\order  \big( M^n_{i+1} / M^n_i \big)$ divisant $\order {\mathcal T}_k 
= \order \Cl_k \!\cdot\! R_k$ et d\'ecroissant vers $1$.
Donc on a au rang final $i=m_n$, en utilisant ce qui pr\'ec\`ede pour
$M^n_{m_n} = \Cl_{k_n}$, les relations ${\rm N}_{k_n/k} \big (M_{m_n}^n \big) = \Cl_k$ et
$(\Lambda_{m_n}^n  : \Lambda_{m_n}^n \cap {\rm N}_{k_n/k}(k_n^\times)) 
= p^{n \cdot (d -1)}$, ce qui explique que $\order  \Cl_{k_n}$ d\'epend essentiellement 
du nombre de pas $m_n$~:

\begin{theorem} \label{O} Soit $k$ Galoisien r\'eel dans lequel $p>2$ est 
totalement d\'ecompos\'e et v\'erifie la conjecture de Leopoldt.
Soit ${\mathcal T}_k$ le groupe de torsion du groupe de Galois de la pro-$p$-extension 
Ab\'elienne $p$-ramifi\'ee maximale de $k$ (cf. \S\,\ref{torsion}). 

Soit $N_{\!\rm Iw} \geq 0$ tel que la formule d'Iwasawa $\order \Cl_{k_n} = 
p^{ \lambda \cdot n + \mu \cdot p^n + \nu}$ soit valable pour tout $n \geq N_{\!\rm Iw}$.

\smallskip
(i) \!On a les in\'egalit\'es $m_n \leq \lambda \cdot n + \mu \cdot p^n + \nu \leq 
v_p(\order  {\mathcal T}_k) \cdot m_n$ pour tout $n \geq N_{\!\rm Iw}$, o\`u $v_p$ d\'esigne
la valuation $p$-adique. Si ${\mathcal T}_k=1$, alors 
$\lambda=\mu=\nu= 0$.

\smallskip
(ii) Si ${\mathcal T}_k \ne 1$, alors $m_n = c(n)\cdot (\lambda \cdot n + \mu \cdot p^n +  \nu)$,
$\frac{1}{v_p(\order  {\mathcal T}_k)} \leq c(n) \leq 1$, avec en particulier l'in\'egalit\'e 
principale sur le nombre de pas $m_n$ de l'algorithme~:
\begin{equation*}
m_n \geq \frac{1}{v_p(\order  {\mathcal T}_k)} \,\big (\lambda \cdot n + \mu \cdot p^n +  \nu \big),
\  \hbox{pour tout $n \geq N_{\!\rm Iw}$}.
\end{equation*}
\end{theorem}

\begin{proof} Comme $\order  \big( M^n_{i+1}/M^n_i\big) \geq p$ pour $0 \leq i \leq m_n-1$,
on obtient $\order \Cl_{k_n}  = \order M^n \geq p^{m_n}$ en utilisant l'expression \eqref{prod}~;
on en d\'eduit l'in\'egalit\'e triviale $m_n \leq \lambda \cdot n + \mu \cdot p^n +  \nu$.
Ensuite, d'apr\`es le fait que $\order \big( M^n_{i+1}/M^n_i \big) \mid \order {\mathcal T}_k$,
on a $\order  \big( M^n_{i+1}/M^n_i \big) \leq \order  {\mathcal T}_k$ pour $0 \leq i\leq m_n-1$~; 
d'o\`u $\order \Cl_{k_n} \leq (\order  {\mathcal T}_k)^{m_n}$ \`a nouveau par \eqref{prod}, ce qui 
ach\`eve la d\'emonstration.
\end{proof}

\begin{remark}\label{mn}
Une heuristique raisonnable est que $m_n$ reste, pour tout $n$, d'une valeur 
moyenne fonction de ${\rm Gal}(k/\Q)$ et de ${\mathcal T}_k$ et non de $n$.
Or, pour tout $n \gg 0$, les $\order \big( M^n_{i+1}/M^n_i\big)$ forment, 
\`a partir de $\order  M_1^n$, une $i$-suite {\it d\'ecroissante vers $1$} d'entiers divisant 
$\order {\mathcal T}_k$~; une telle suite ne peut avoir, pour $\lambda$ ou $\mu$ non nuls, 
au moins $\frac{1}{v_p(\order {\mathcal T}_k)} \, (\lambda \cdot n + \mu \cdot p^n +  \nu)$ 
termes non triviaux sauf \`a \'etablir un lien de type particulier avec les normes d'id\'eaux 
${\mathfrak A}_i \in {\mathcal I}_i^n$ utilis\'es dans la tour, ce qui philosophiquement 
semble exclu, \`a notre avis, car ceci implique l'existence, pour $n \to \infty$, 
d'au moins un diviseur $t_n \ne 1$ de $\order  {\mathcal T}_k$ tel que 
$\order \big( M^n_{i+1}/M^n_i\big) = t_n$ pour 
$O(1)\cdot (\lambda \cdot n + \mu \cdot p^n +  \nu)$ 
valeurs {\it cons\'ecutives} de $i \in [1, m_n]$.

\smallskip
En faisant l'hypoth\`ese tr\`es faible que la probabilit\'e de
$\order \big( M^n_{i+1}/M^n_i\big) \geq p$ pour $m_n = O(n)$ pas cons\'ecutifs
de l'algorithme est de la forme $\ds \frac{1}{p^{f(n)}}$, o\`u $f(n) \to \infty$ avec $n$,
on peut admettre, sur un plan heuristique assez \'evident, qu'\`a partir d'un certain
$n \gg N_{\!\rm Iw}$, ceci n'est plus vraisemblable, auquel cas c'est {\it l'existence m\^eme
de la formule d'Iwasawa} qui sugg\`ere la nullit\'e de $\lambda$ et $\mu$ et l'existence 
de $i_0$ assez grand tel que $\order \big( M^n_{i_0+1}/M^n_{i_0}\big)=1$ pour
tout $n \gg 0$ (i.e., $m_n \leq i_0$ pour tout $n$).
Nous approfondirons cet aspect, concernant l'existence de $i_0$, en examinant 
s\'epar\'ement les deux facteurs~:
$$\hbox{$\ds \frac{\order \Cl_{k}}{\order {\rm N}_{k_n/k}(M^n_i)}\ \ $ $\ \&\ $ 
$\ \ \ds \frac{p^{n \cdot (d -1)}}{(\Lambda_i^n : \Lambda_i^n \cap {\rm N}_{k_n/k}(k_n^\times))}$,} $$

qui ne sont pas de la m\^eme nature,
et en \'enon\c cant les Heuristiques \ref{heur1} et \ref{heur2}, puis
nous donnerons des justifications probabilistes tr\`es naturelles (\S\,\ref{resglob}) et  
des statistiques num\'eriques (Section \ref{BC}).
\end{remark}

A cet effet, rappelons l'algorithme num\'erique qui permet de passer de $\Lambda^n_i$ \`a 
$\Lambda^n_{i+1}$, et qui d\'etermine le nombre de pas $m_n$, pour analyser les 
ph\'enom\`enes en jeu, car d\`es qu'un $x \in \Lambda_{i+1}^n$ se rajoute 
aux \'el\'ements de $\Lambda_i^n$, ses $\delta_{\mathfrak p}(x)$, ${\mathfrak p} \mid p$,
ont une grande probabilit\'e d'\^etre nuls ou au moins inf\'erieurs aux pr\'ec\'edents, 
ce qui donne la $i$-suite des $\ds \frac{p^{n \cdot (d -1)}}
{(\Lambda_i^n\  :\  \Lambda_i^n \cap {\rm N}_{k_n/k}(k_n^\times))}$ 
rapidement d\'ecroissante vers~1.

\subsection{Algorithme de calcul des $\Lambda_i^n$}\label{algo}
On omet l'indice $n$ qui est fix\'e et on se place dans l'extension 
$K \subset k_\infty$ de degr\'e $p^n$ et de groupe de Galois 
$G = :\langle \sigma \rangle$ d'ordre $p^n$.
On d\'esigne pour simplifier par ${\rm N}$ la norme arithm\'etique ${\rm N}_{K/k}$
et on pose $M := \Cl_K$.

\smallskip
(i) Pour le calcul de $M_1=M^G$ \`a partir de $M_0=1$ et ${\mathcal I}_0=1$, 
on a  $\order M_1=\order  \Cl_k \cdot \ds \frac{p^{n \cdot (d -1)}}
{(\Lambda_0 : \Lambda_0 \cap {\rm N}( K^\times))}$
avec $\Lambda_0= \{x_0 \in k^{\times},\  (x_0) \in {\rm N} (1)\} = E_k$.

\smallskip
On consid\`ere les $x_0 \in \Lambda_0$ qui sont normes d'un \'el\'ement $y_1 \in K^{\times}$. 
Donc $(x_0) =  {\rm N}(y_1) = (1)$, ce qui conduit \`a l'existence de ${\mathfrak A}_1 \in I_K$ tel que
${\mathfrak A}_1^{1-\sigma} = (y_1)$, o\`u ${\mathfrak A}_1$ est d\'efini
\`a un id\'eal invariant pr\`es~; donc ici, puisque les id\'eaux premiers de $K$ au-dessus de $p$
sont invariants, on peut prendre ${\mathfrak A}_1$ \'etranger \`a $p$. On a 
$\cl_K({\mathfrak A}_1) \in M_1$.
R\'eciproquement, si $\cl_K({\mathfrak A}'_1) \in M_1$, ${\mathfrak A}'_1$ choisi \'etranger \`a $p$
dans sa classe, il existe $y'_1 \in K^{\times}$ tel que ${\mathfrak A}'_1{}^{1-\sigma} = (y'_1)$, 
donnant ${\rm N}(y'_1) = x'_0 \in \Lambda_0 \cap {\rm N}(K^\times)$. Ainsi $y'_1$ est \'etranger \`a $p$.

\smallskip
Les classes de ces id\'eaux ${\mathfrak A}_1$ engendrent $M_1$ et on pose~:
$${\mathcal I}_1=\langle {\mathfrak A}_1^1, \ldots, {\mathfrak A}_1^{r_1} \rangle. $$

Ceci suppose que l'on a r\'esolu suffisamment d'\'equations normes
pour engendrer $M_1$. Par exemple, le cas
$x_0=y_1=1$ conduit normalement \`a inclure dans ${\mathcal I}_1$ 
des id\'eaux ambiges repr\'esentant les classes de $k$ et les id\'eaux
${\mathfrak P}$ de $K$ au-dessus de $p$~; en effet, ${\mathfrak A}_1$ est d\'efini
\`a un id\'eal invariant pr\`es, ce qui peut changer la classe car ${\mathfrak A}_1$ peut \^etre
principal et non ${\mathfrak A}_1\cdot {\mathfrak P}$~; mais on peut toujours repr\'esenter la classe de 
${\mathfrak P}$ par un id\'eal ${\mathfrak A}'$ de $K$ \'etranger \`a $p$. 
On obtient ${\rm N}(M_1)$ comme sous-groupe de $\Cl_k$, \`a partir de
${\rm N}({\mathcal I}_1) = \langle {\rm N}( {\mathfrak A}_1^1), \ldots, {\rm N}({\mathfrak A}_1^{r_1}) \rangle$~;
alors il vient $\Lambda_1 = \{x_1 \in k^{\times}, \ (x_1) \in {\rm N}({\mathcal I}_1)\}$ et, 
toujours au niveau $n$ fix\'e~:
\begin{equation*}
\order ( M_2 /M_1)  = \frac{\order  \Cl_k}{\order  {\rm N}(M_1)} \times
 \frac{p^{n \cdot (d -1)}}{(\Lambda_1  : \Lambda_1 \cap {\rm N}(K^\times))} .
\end{equation*}

On v\'erifie qu'en d\'epit de la non unicit\'e de ${\mathcal I}_1$, le groupe
$\Lambda_1$ est unique modulo ${\rm N}(K^\times)$.

\smallskip
(ii) Pour le calcul de ${\mathcal I}_2$, on consid\`ere les \'el\'ements $x_1$ de
$\Lambda_1$ qui sont normes d'un $y_2 \in K^{\times}$ (de fait on se limite aux $x_1$
repr\'esentant $ \Lambda_1 \cap {\rm N}(K^\times)$ modulo $E_k \cap {\rm N}(K^\times)$)~; 
alors $(x_1) = {\rm N}(y_2) = {\rm N}({\mathfrak B}_1)$, ${\mathfrak B}_1 \in {\mathcal I}_1$,
donc il existe ${\mathfrak A}_2 \in I_K$ tel que 
$ {\mathfrak A}_2^{1-\sigma}\cdot {\mathfrak B}_1= (y_2)$, avec ${\mathfrak A}_2$
choisi \'etranger \`a $p$. On a $\cl_K({\mathfrak A}_2)^{1-\sigma} \in M_1$.
Inversement, si $\cl_K({\mathfrak A}'_2) \in M_2$,
${\mathfrak A}'_2$ \'etranger \`a $p$, il existe $y'_2 \in K^{\times}$ 
tel que ${\mathfrak A'}_2^{1-\sigma} \cdot {\mathfrak B}'_1= (y'_2)$, avec
${\mathfrak B}'_1 \in {\mathcal I}_1$, d'o\`u ${\rm N}({\mathfrak B}'_1) = 
{\rm N}(y'_2) =: (x'_1)$, $x'_1 \in  \Lambda_1\cap {\rm N}(K^\times)$.

\smallskip
Ces id\'eaux de la forme ${\mathfrak A}_2^1, \ldots, {\mathfrak A}_2^{r_2}$,
sont ajout\'es \`a ${\mathcal I}_1$ pour former~:
$${\mathcal I}_2=\langle {\mathfrak A}_1^1, \ldots, {\mathfrak A}_1^{r_1}\ ; \ 
{\mathfrak A}_2^1, \ldots, {\mathfrak A}_2^{r_2} \rangle, $$

d'o\`u ${\rm N}({\mathcal I}_2)$ et $\Lambda_2 = \{x_2 \in k^{\times}, \ (x_2) \in {\rm N}({\mathcal I}_2)\}$, etc.
On obtient donc des groupes $\Lambda_i$, \'etrangers \`a $p$, uniques modulo 
${\rm N}(K^\times)$, tels que~:
$$E_k \subseteq \Lambda_1 \subseteq \cdots \subseteq \Lambda_i \subseteq \ldots. $$

On suppose implicitement que chaque ${\mathcal I}_i$ est form\'e d'id\'eaux \'etrangers \`a $p$.
Dans le cas $\Cl_k=1$, les groupes $\Lambda_i$ sont engendr\'es par les unit\'es
$\varepsilon \in E_k$ et les $\alpha_i^j$ g\'en\'erateurs des ${\rm N}({\mathfrak A}_i^j)$,
ce qui simplifie l'expression de $\Lambda_i$.

\smallskip
Cette description n'a pas lieu d'\^etre effective dans la tour $k_\infty$, mais elle montre
comment se d\'eterminent les \'el\'ements $x_i$ des $\Lambda_i$ dont on rappelle que 
tout repose sur leurs ${\mathfrak p}$-quotients de Fermat, pour ${\mathfrak p} \mid p$, 
\`a partir de l'\'ecriture $x_i^{p-1} =1+ p\cdot \beta(x_i)$, conduisant \`a 
$\delta_{\mathfrak p}(x_i) = v_{\mathfrak p}(\beta(x_i)) \geq 0$.

\smallskip
Une heuristique classique est que, pour ${\mathfrak p}$ fix\'e, les 
$\delta_{\mathfrak p}(x)$, $x \in k^\times$, sont al\'eatoires et ind\'ependants.
Cependant, il est n\'ecessaire de pr\'eciser cet aspect pour le cas 
des id\'eaux ${\mathfrak A}_i$ de $K$ constituant les groupes ${\mathcal I}_i$
pour lesquels ${\rm N}_{K/k}({\mathcal I}_i)$ conduit aux nombres $x_i$.
Nous montrerons au \S\,\ref{resglob} et dans la Section \ref{BC} 
ce qu'il en est au plan probabiliste en utilisant des programmes 
\'etablissant des statistiques tr\`es convaincantes, tenant compte de la formule du produit. 

\smallskip
Auparavant, examinons de fa\c con plus num\'erique un cas non trivial, mais
le plus simple possible, en supposant la conjecture fausse.

\subsection{La $i$-suite des $\Lambda_i^n$ sous les hypoth\`eses
$\Cl_k=1\ \& \  \lambda \geq 1$}\label{idx}
Pour fixer les id\'ees, supposons que $k$ est un corps quadratique r\'eel, que $p>2$ est
d\'ecompos\'e, que $\Cl_k=1$, que $\delta_p(\varepsilon) \geq 1$, et que
(en posant pour simplifier ${\mathfrak p}_1=: {\mathfrak p}$ et $\pi_2 =: \pi$) on a
$\delta_{\mathfrak p}(\pi) \geq 1$ (cf. Section \ref{quad}). 

Par exemple, $m = 103$, $\varepsilon = 22419 \, \sqrt {103} + 227528$, $\pi = \sqrt {103} + 10$,
pour lesquels $k$ est principal, $\delta_p(\varepsilon)= 1$
et $\delta_{\mathfrak p}(\pi)= 1$~; ainsi la condition suffisante du 
Th\'eor\`eme \ref{thm2} n'est pas v\'erifi\'ee. 
On a donc $\order {\mathcal T}_k = R_k = p$ et la relation de divisibilit\'e du 
Th\'eor\`eme \ref{regulateur} (i) pour tout $n$.

\smallskip
Soit $n$ fix\'e pour lequel la formule d'Iwasawa est applicable, posons
\`a nouveau $K:=k_n$, $M := \Cl_K $, et supposons par exemple $\lambda =1$ 
($\mu=0$, $\nu \in \Z$)~; on a donc au niveau $n$, $\order  M = p^{n + \nu}$ et
puisque $\delta_p(\varepsilon)= 1$~:
$$\order  M_1 =  \frac{p^{n} }{(\Lambda_0 : \Lambda_0 \cap {\rm N}_{K/k}(K^\times))} = p, 
\ \ \hbox{avec $\Lambda_0 = E_k$.}$$

On a ${\mathcal I}_1=\langle {\mathfrak A}_1 \rangle$ obtenu de la fa\c con suivante~:
pour une puissance convenable $\varepsilon'$ de l'unit\'e fondamentale $\varepsilon$ (de fait, 
sous l'hypoth\`ese faite, n\'ecessairement $\varepsilon'=\varepsilon^{p^{n-1}}$), on a
$\varepsilon' =  {\rm N}(y_1)$, $y_1 \in K^\times$, d'o\`u $(y_1) =  {\mathfrak A}_1^{1-\sigma}$
avec ${\mathfrak A}_1$ \'etranger \`a $p$ tel que 
la classe de ${\mathfrak A}_1$ soit d'ordre $p$ puisque $M_1$ est d'ordre~$p$. 

\smallskip
Posons ${\mathfrak A}_1^p=(z_1)$, $z_1 \in K^\times$, et
${\rm N}_{K/k}({\mathfrak A}_1)=(\alpha_1)$, $\alpha_1 \in k^\times$~; on obtient 
$\alpha_1^p = {\rm N}_{K/k}(z_1) \cdot \eta_1$ o\`u $\eta_1$ est une unit\'e de $k$. 
Soit alors $\omega = \omega _n$ l'application qui \`a $x \in k^\times$ associe
dans $\Omega(K/k) \simeq \Z/p^n\Z$ la famille des 
symboles de restes normiques dans $K/k$. On obtient 
$\omega(\eta_1) = \omega(\alpha_1^p) = \omega(\alpha_1)^p \in \Omega(K/k)^p$~; 
or, en raison de l'hypoth\`ese 
$\delta_p(\varepsilon) = 1$, on a $\omega(\eta_1) \in \Omega(K/k)^p$, et
par cons\'equent $\omega(\alpha_1)$, donc $\delta_p(\alpha_1)$, peut a priori 
prendre toute valeur de fa\c con \'equiprobable dans $\Omega(K/k)$ (sans la 
condition $\omega(\eta_1) \in \Omega(K/k)^p$, l'algorithme se serait arr\^et\'e avant, 
la relation $\omega(\eta_1) = \omega(\alpha_1^p)$ \'etant alors absurde).
On a donc obtenu 
${\rm N}_{K/k}({\mathcal I}_1) = \langle  {\rm N}_{K/k}({\mathfrak A}_1) \rangle
=  \langle (\alpha_1) \rangle$, d'o\`u~:
$$\order  (M_2/M_1) = \frac{p^{n} }{(\Lambda_1 : \Lambda_1 \cap {\rm N}_{K/k}(K^\times))}
= 1 \hbox{ ou }  p,
\ \ \hbox{avec}\  \Lambda_1 =  \langle \varepsilon,\, \alpha_1 \rangle.$$

Puisque $\delta_p(\varepsilon) = 1$, l'algorithme ne peut se terminer par $M=M_1$
que si $\delta_p(\alpha_1)$ est nul~; sinon l'algorithme se poursuit, et en supposant 
$\delta_p(\alpha_1) \geq 1$, il faut prendre un \'el\'ement convenable de $\Lambda_1$, 
de la forme $\alpha_1^u \cdot \varepsilon^v$, $u, v \in \Z$, comme norme de $y_2$ 
dans $K/k$ et on obtient, puisque ${\rm N}_{K/k}({\mathfrak A}_1)=(\alpha_1)$, 
une relation de la forme ${\mathfrak A}_2^{1-\sigma} \, {\mathfrak A}_1^u = (y_2)$ 
dans $K$ pour d\'efinir ${\mathcal I}_2$.

\smallskip
Mais $M_2/M_1$ est annul\'e par $p$, et en supposant $M_2 \ne M_1$,
on aura ${\mathfrak A}_2^p = {\mathfrak A}_1^w \cdot (z_2)$, $w \in \Z$,
$z_2 \in K^\times$, d'o\`u ${\rm N}_{K/k}({\mathfrak A}_2^p) = (\alpha_1^w)\, {\rm N}_{K/k}(z_2)$ 
qui conduit, en posant $(\alpha_2) = {\rm N}_{K/k}({\mathfrak A}_2)$, 
\`a $\omega(\alpha_2)^p = \omega(\alpha_1^w \cdot \eta_2)$ 
pour une unit\'e $\eta_2$ de $k$~; or $\omega(\alpha_1)$ et $\omega(\eta_2)$ sont dans 
$\Omega(K/k)^p$.
Comme pr\'ec\'edemment, $\omega(\alpha_2)$ n'est soumis a priori \`a aucune obstruction
pour conduire \'eventuellement \`a $\delta_p(\alpha_2) = 0$.

\smallskip
L'algorithme pour $K$ se poursuit avec des calculs analogues et on obtient~:
\begin{equation}\label{Lambda}
\omega(\alpha_1^p) \in \omega(\Lambda_0),  \  \omega(\alpha_2^p) \in \omega(\Lambda_1), \ldots, 
\ \omega(\alpha_{i+1}^p) \in \omega(\Lambda_i),  \ldots  
\end{equation} 

o\`u $\Lambda_i = \langle  \varepsilon,  \alpha_1,  \ldots,  \alpha_i \rangle$
et o\`u les $\order  \omega(\Lambda_i) = 
(\Lambda_i : \Lambda_i \cap {\rm N}_{K/k}(K^\times))$ 
forment une $i$-suite croissante minor\'ee par $p^{n-1}$
puisque $\order  (M_{i+1}/M_i)=\ds \frac{p^n}{\order  \omega(\Lambda_i)} = p$ pour $0 \leq i \leq m_n-1$,
donc stationnaire \`a la valeur $p^{n-1}$ d\`es l'indice~$0$~; les relations
$\omega(\alpha_{i+1}^p) \in \omega(\Lambda_i)$ permettent, statistiquement, une ``d\'ecroissance'' des
$\delta_p(\alpha_{i})$, ce qui semble contradictoire avec un nombre de pas 
$m_n \geq n +  \nu$ d\'ecoulant de l'hypoth\`ese $\lambda = 1$, $v_p(\order  {\mathcal T}_k)=1$
(cf. Th\'eor\`eme \ref{O}).

\smallskip
Une autre fa\c con d'aborder ces questions heuristiques est la suivante 
(avec les m\^emes hypoth\`eses que ci-dessus)~:

\smallskip
Repr\'esentons la classe de ${\mathfrak A}$ dans $K$ par un id\'eal premier ${\mathfrak L}$ de $K$~; 
on a ${\rm N}_{K/k}({\mathfrak L}) = {\mathfrak l}^{f^n} =: (\beta)^{f^n} =: (\alpha)$ 
o\`u ${\mathfrak l} = (\beta)$ est l'id\'eal premier de $k$ en-dessous de ${\mathfrak L}$ et $f^n$ 
son degr\'e r\'esiduel dans $K/k$ (autrement dit, $\alpha = \beta^{f^n} \!\!\cdot \eta$, $\eta \in E_k$). 
Or la condition normique $\delta_{\mathfrak p}(\alpha)=0$ ne peut
avoir lieu que si $f^n = 1$ (i.e., ${\mathfrak l}$ totalement d\'ecompos\'e dans $K/k$),
puisque $\delta_{\mathfrak p}(\eta) \geq 1$ par hypoth\`ese sur $\delta_{\mathfrak p}(\varepsilon)$~;
ainsi $\delta_{\mathfrak p}(\alpha) = 0$ implique, pour $\ell$ en-dessous de 
${\mathfrak l}$, $l^{p-1} \equiv 1 \pmod {p^{n+1}}$.
Mais on montrera au \S\,\ref{L} au moyen du th\'eor\`eme de 
Tchebotarev qu'une classe de $k_n$ est toujours (donc {\it n\'ecessairement}) 
repr\'esentable par un ${\mathfrak L}$ totalement d\'ecompos\'e dans $K/k$.

\smallskip
On a donc, sous l'hypoth\`ese $\order {\mathcal T}_k =R_k = p\ \&\ \lambda =1$, 
$\order (M_{i+1}/M_i) = p$ pour $O(n)$ pas, alors que l'on peut conjecturer que
$M=M_{O(1)}$ pour $n\to\infty$.

\smallskip
Le troisi\`eme programme du \S\,\ref{prog}, donne pour $m=103$~:
$$\Cl_{k_1} \simeq \Z/3\,\Z \oplus \Z/3\,\Z \ \,\hbox{(i.e., $M_2 \ne M_1$ 
au premier \'etage)}. $$

Le $S_{k_1}$-groupe des classes de $k_1$ est d'ordre $3$,
ce qui donne $\order \cl_{k_1}(S_{k_1})=3$, et le point (iii) du Th\'eor\`eme
\ref{regulateur} pour $n_0=e=1$ conduit \`a
$\lambda=\mu=0$~; pour ce calcul, rajouter au programme~:
\footnotesize
\begin{verbatim}
L=bnfinit(R,1); Su=bnfsunit(L,idealprimedec(L,p)); 
print(component(component(Su,5),1))
\end{verbatim}
\normalsize
Par des m\'ethodes de type ``Spiegelungssatz'' on montre aussi que 
$\lambda =0$ (cf. \cite{BaN}, \cite{CN}, \cite{FT1}). On peut v\'erifier que
 $\Cl_{k_2} \simeq \Cl_{k_3} \simeq \Z/3\,\Z \oplus \Z/3\,\Z$.

\medskip
La section suivante est l'\'etude, \`a $i$ fix\'e, de la $n$-suite $\order M^n_i$.

\section{Heuristiques sur les filtrations dans 
$k_\infty/k$} \label{nsuite}\label{sect7}
Ici, contrairement aux \'etudes pr\'ec\'edentes, on fixe l'\'etape $i$ des 
algorithmes et on consi\-d\`ere les $n$-suites des groupes 
$M^n_i := \cl_{k_n} ({\mathcal I}^n_i)$ et $\Lambda_i^n$.
On \'etudie alors les entiers $\order  \big( M^n_{i+1} / M^n_i \big)$
pour $i \geq 0$ fix\'e et $n \to \infty$ \`a partir de leurs deux facteurs 
donn\'es par la Formule \eqref{eq5}.

\subsection{Approche g\'en\'erale}\label{appgene}
On a, pour tout $n \geq 0$, le diagramme suivant o\`u les normes 
${\rm N}_{k_{n+1}/k_n}$, d\'efinies sur $M^{n+1}$ et $(M^{n+1})^{(1-\sigma_{n+1})^i}$, sont 
surjectives, mais non celle d\'efinie sur $M_i^{n+1}$ (qui peut \^etre ni injective ni surjective)~:
\footnotesize
\vspace{-0.15cm}
$$\begin{array}{ccccccccc}  
1 \!\!\!\! & \too & M_i^{n+1} & \toooo & M^{n+1} &  \stackrel{(1-\sigma_{n+1})^i}{\toooo} 
&  (M^{n+1})^{(1-\sigma_{n+1})^i} & \!\!\!\! \too 1 \\   \vspace{-0.3cm}   \\
& &   \Big \downarrow& & \hspace{-1.4cm}  {\rm N}_{k_{n+1}/k_n} \Big \downarrow 
\hspace{-0.29cm} \downarrow& 
&\hspace{-1.4cm}  {\rm N}_{k_{n+1}/k_n} \Big \downarrow 
\hspace{-0.29cm} \downarrow &  \\  \vspace{-0.4cm}   \\
1 \!\!\!\! &  \too   & M_i^n  
& \toooo &  M^n & \stackrel{(1-\sigma_n)^i}{\toooo} & \!\!\!\!  (M^n)^{(1-\sigma_n)^i} & \too 1 \,.
\end{array} $$

\normalsize
On a ${\rm N}_{k_{n+1}/k_n} (M_i^{n+1})  \subseteq M_i^n$~; donc, pour tout
id\'eal ${\mathfrak A}^{n+1} \in {\mathcal I}^{n+1}_i$, on peut \'ecrire
${\rm N}_{k_{n+1}/k_n}({\mathfrak A}^{n+1}) = (\alpha^n) \, {\mathfrak A}^n$, o\`u
$\alpha^n \in k_n^\times$ et ${\mathfrak A}^n \in {\mathcal I}^n_i$, auquel cas,
{\it en modifiant} ${\mathcal I}^n_i$ modulo des id\'eaux principaux de $k_n$, 
on peut supposer ${\rm N}_{k_{n+1}/k_n} ({\mathcal I}^{n+1}_i ) \subseteq  
{\mathcal I}^n_i$ et par cons\'equent ${\rm N}_{k_{n+1}/k} ({\mathcal I}^{n+1}_i ) 
\subseteq {\rm N}_{k_{n}/k} ({\mathcal I}^n_i)$~; ceci revient \`a modifier 
$\Lambda^{n}_i = \{ x\in k^\times,  (x) \in {\rm N}_{k_n/k}({\mathcal I}^n_i ) \}$ 
modulo des normes d'\'el\'ements de $k_n^\times$ ce qui laisse invariant 
$(\Lambda^{n}_i : \Lambda^{n}_i \cap {\rm N}_{k_{n}/k}(k_n^\times))$.
On peut alors supposer que, pour tout $h\geq 1$, on a~:
\begin{equation}\label{incl}
E_k \subseteq  \Lambda^{n+h}_i \subseteq \cdots 
\subseteq \Lambda^{n+1}_i \subseteq \Lambda^{n}_i. 
\end{equation}

\begin{lemma} \label{lem1} Pour tout $i \geq 0$ fix\'e, les $\order  \big( M^n_{i+1} / M^n_i \big)$ 
forment une $n$-suite croissante stationnaire de diviseurs de $\order {\mathcal T}_k$, et
les $\order  M^n_i$ d\'efinissent une $n$-suite {\it croissante} stationnaire d'entiers.
\end{lemma}

\begin{proof} Consid\'erons pour $i \geq 0$ fix\'e la $n$-suite d\'efinie par~:
\begin{equation}\label{cr}
\order  \big( M^n_{i+1} / M^n_i \big)=
\frac{\order  \Cl_{k}}{\order {\rm N}_{k_n/k}( M^n_i) } \cdot  \frac{p^{n \cdot (d -1)} }{ 
(\Lambda_i^n :  \Lambda_i^n \cap {\rm N}_{k_n/k}(k_n^\times))}. 
\end{equation}

Comme ${\rm N}_{k_{n+1}/k} (M_i^{n+1})\! \subseteq\! {\rm N}_{k_{n}/k} (M_i^n)$,
la $n$-suite $\ds \frac{\order \Cl_{k}}{\order {\rm N}_{k_n/k}( M^n_i) } 
=: p^{c_i^n}$ est {\it croissante} stationnaire \`a une valeur maximale not\'ee
$p^{c_i^{\infty}} \mid \order  \Cl_k$.

\smallskip
Le second facteur $\ds \frac{p^{n \cdot (d -1)}}
{(\Lambda_i^n :  \Lambda_i^n \cap {\rm N}_{k_n/k}(k_n^\times))} 
= \frac{p^{n \cdot (d -1)}}{\order \omega_n (\Lambda_i^n)}  =: p^{\rho_i^n}$
d\'efinit une $n$-suite {\it croissante} d'entiers.
En effet, on a $p^{\rho_i^{n+1}-\rho_i^n} = p^{d-1}\cdot \ds
\frac{\order \omega_n(\Lambda_i^n)}{\order \omega_{n+1}(\Lambda_i^{n+1})}$, et
comme on peut supposer que l'on a $\Lambda_i^{n+1} \subseteq \Lambda_i^n$
(cf. Relation \eqref{incl}), alors
$\order \omega_{n+1}(\Lambda_i^{n+1}) \leq \order \omega_{n+1}(\Lambda_i^n)$
puis $p^{\rho_i^{n+1}-\rho_i^n} \geq p^{d-1}\cdot \ds
\frac{\order \omega_n(\Lambda_i^n)}{\order \omega_{n+1}(\Lambda_i^{n})}$~;
dans la restriction des symboles de Hasse
$\Omega(k_{n+1}/k) \too \hspace{-0.5cm} \too \Omega(k_{n}/k)$
(dont le noyau est d'ordre $p^{d-1}$), l'image de $\omega_{n+1}(\Lambda_i^{n})$ 
est $\omega_n(\Lambda_i^n)$, d'o\`u le r\'esultat pour la $n$-suite 
$p^{\rho_i^n}$, stationnaire \`a une 
valeur maximale $p^{\rho_i^{\infty}} \mid R_k$ (Th\'eor\`emes \ref{regulateur}, 
\ref{theta}), et au total le premier point du lemme
en r\'esulte pour la $n$-suite $\order \big( M^n_{i+1} / M^n_i \big)$~; on a 
$\ds \lim_{n \to \infty} \order  \big( M^n_{i+1} / M^n_i \big) = 
p^{c_i^{\infty}} \!\cdot p^{\rho_i^{\infty}} \mid \order  {\mathcal T}_k$.

\smallskip
Enfin, si l'on suppose, par r\'ecurrence, que la $n$-suite $\order  M^n_i$ est
croissante stationnaire, la propri\'et\'e en r\'esulte pour la $n$-suite $\order M^n_{i+1}$.
\end{proof}

\begin{lemma} \label{lem11} Les $i$-suites $p^{c_i^{\infty}}$, $ p^{\rho_i^{\infty}}$, et
$\ds \lim_{n \to \infty} \order (M^n_{i+1} / M^n_i) = p^{c_i^{\infty}} \cdot p^{\rho_i^{\infty}}$ 
sont d\'e\-croissantes stationnaires, respectivement vers un diviseur de $\order \Cl_{k}$,
$R_k$, et $\order  {\mathcal T}_k$.
\end{lemma}

\begin{proof} Provient facilement de l'expression des deux facteurs de \eqref{cr}
pour $n$ pris assez grand (voir aussi la Remarque \ref{filtration} (iii)).
\end{proof}

\begin{corollary}\label{imin}
Il existe $i_{\rm min} \geq 0$ et des constantes $c \geq 0$ et $\rho \geq 0$
telles que $c_i^{\infty} = c$ et $\rho_i^{\infty} = \rho$ pour tout $i \geq i_{\rm min}$.
\end{corollary}

Mais la $i$-suite $ p^{c_i^{\infty}} \!\! \cdot p^{\rho_i^{\infty}}$ n'est pas 
n\'ecessairement de limite $1$~; en effet, on a seulement que 
$\ds \lim_{i \to \infty}\big( p^{c_i^{\infty}} \cdot p^{\rho_i^{\infty}} \big) = p^{c+\rho}$ 
divise $\order {\mathcal T}_k$ et il convient d'examiner chacun des deux facteurs. 
Auparavant, donnons la d\'efinition suivante~:

\begin{definition}\label{process}
On dira que le processus limite sur $i$ est fini (ce qui \'equivaut \`a $\lambda = \mu =0$) 
s'il existe $i_0 \in \N$ tel que $p^{c_{i_0}^{\infty}} \cdot p^{\rho_{i_0}^{\infty}} = 1$.
\end{definition}

Dans ce cas on a donc $i_0=i_{\rm min}$, $c=\rho=0$, et
pour tout $n \gg 0$, $\order M^n$ est une constante not\'ee $p^\nu$,
ind\'ependante de $n$.

\subsection{Comportement heuristique de $\omega_n(\Lambda_i^n)
\subseteq \Omega(k_n/k), \  n\to\infty$}\label{comportement1}
Le processus limite sur $i$ peut \^etre infini 
(i.e., $p^{\rho_{i}^{\infty}} \ne 1 \ \forall  i \geq 0$, ou encore $\rho>0$) si pour tout $i$ il existe 
$n \gg 0$ tel que $\omega_n(\Lambda_i^n) \subset \Omega(k_n/k)$ ou encore 
$(\Lambda_i^n : \Lambda_i^n \cap {\rm N}_{k_n/k}(k_n^\times)) < p^{n \cdot (d-1)}$.
Il existerait $n_1 \gg 0$ tel que pour $n \geq n_1$,
$\ds \frac{p^{n\cdot(d-1)}}{\omega_n(\Lambda_i^n)} = p^\rho \ne 1$,
pour $i_{\rm min} \leq i \leq m_n-1$ (Corollaire \ref{imin}). Ceci signifie que
sous l'hypoth\`ese $\lambda \geq 1$ ou $\mu \geq 1$, l'algorithme au niveau $n$
devrait comporter $m_n = O(\lambda\cdot n + \mu \cdot p^n)$ ``$i$-\'etapes'' 
successives (cf. Th\'eor\`eme \ref{O}), avec  $\omega_n(\Lambda_i^n)
\subset \Omega(k_n/k)$, et ceci pour $n$ arbitrairement grand.
Par cons\'equent la finitude du processus limite dans ce cas ne peut provenir que 
de l'heuris\-tique suivante, en raison des propri\'et\'es des quotients de Fermat 
(nous y reviendrons au \S\,\ref{resglob} au moyen d'une justification globale)~:

\begin{conjecture} \label{heur1}
On consid\`ere les groupes $\Lambda_i^n$ associ\'es \`a l'ensemble des algorithmes 
de d\'etermination des sous-groupes $M^n_i$, $0 \leq i \leq m_n$, des filtrations des 
$M^n := \Cl_{k_n}$, $n\geq 0$.
 Alors il existe $i_1$ assez grand, ind\'ependant de $n$, tel que 
$\ds \frac{p^{n \cdot (d -1)}}{(\Lambda_{i_1}^n : \Lambda_{i_1}^n \cap 
{\rm N}_{k_n/k}(k_n^\times))} = 1$ pour $n \to \infty$.
\end{conjecture}

\subsection{Comportement heuristique de ${\rm N}_{k_n/k}( M^n_i) 
\subseteq \Cl_k, \  n\to\infty$} \label{comportement2}
Lorsque $\Cl_k \ne 1$, la non finitude du processus limite sur $i$
peut provenir du fait que pour tout $i \in \N$ il existe $n \gg 0$ tel que 
${\rm N}_{k_n/k}( M^n_i) \subset \Cl_{k}$ (i.e., $p^{c_i^\infty} \ne 1$, ou encore $c>0$).
Comme pour l'Heuristique \ref{heur1}, il existerait 
$n_2 \gg 0$ tel que pour $n \geq n_2$, 
$\ds\frac{\order  \Cl_{k}}{\order {\rm N}_{k_n/k}( M^n_{i}) } = 
p^c \ne 1$, pour $i_{\rm min} \leq i \leq m_n-1$,
et une conclusion analogue sur le nombre d'\'etapes $m_n$.

\smallskip
Or on peut se baser sur l'heuristique naturelle suivante stipulant 
que les classes des id\'eaux de $k$ de la forme ${\rm N}_{k_n/k} ({\mathfrak A})$
(ou ${\rm N}_{k_n/k} ({\mathcal L})$ avec ${\mathcal L}$ premier
totalement d\'ecompos\'e dans $k_n/k$) sont al\'eatoires, ind\'ependantes, 
et se r\'epartissent dans le groupe fini $\Cl_k$ selon les probabilit\'es standard 
dans la mesure o\`u ${\rm N}_{k_n/k}(\Cl_{k_n}) = \Cl_k$ (voir de m\^eme 
le \S\,\ref{resglob} pour des justifications pr\'ecises \`a ce sujet)~:

\begin{conjecture} \label{heur2} 
On consid\`ere les groupes ${\rm N}_{k_n/k}(M^n_i)$ associ\'es \`a l'ensemble des 
algorithmes de calcul des sous-groupes $M^n_i$, $0 \leq i \leq m_n$, des filtrations des 
$M^n := \Cl_{k_n}$, $n\geq 0$. Alors il existe $i_2$, assez grand, ind\'ependant de $n$, 
tel que ${\rm N}_{k_n/k}( M^n_{i_2}) = \Cl_k$ pour $n \to \infty$.
\end{conjecture}

\begin{remark} Dans le cas o\`u $p$ est non d\'ecompos\'e dans le corps
totalement r\'eel $k$ et en supposant $k_\infty/k$ totalement ramifi\'ee en l'unique 
${\mathfrak p} \mid p$, le crit\`ere de Greenberg est que $\lambda = \mu = 0$ 
si et seulement si $\Cl_k$ capitule dans un $k_{n_0}$ 
\cite[Theorem 1]{Gre} (crit\`ere ind\'ependant de la conjecture de Leopoldt).
On a dans ce cas, pour tout $n \geq 0$ et tout $i$, $0 \leq i \leq m_n-1$~: 
$$\order (M_{i+1}^n/M_i^n) =\ds \frac{\order \Cl_{k}}{\order {\rm N}_{k_n/k}( M^n_i)}, $$

le ``facteur normique'' \'etant trivial en raison de la formule du produit. 

\smallskip
La $i$-suite des $\Lambda_i^n$ ($n$ fix\'e) est alors telle que tout $x_i \in \Lambda_i^n$ 
est de la forme ${\rm N}_{k_n/k}(y_i)$, $y_i \in k_n^\times$, et 
$(x_i) = {\rm N}_{k_n/k}({\mathfrak B}_i)$ pour un id\'eal ${\mathfrak B}_i$ de ${\mathcal I}_i^n$
de sorte qu'il existe ${\mathfrak A}_{i+1}$ \'etranger \`a $p$ de $k_n$ tel que 
${\mathfrak A}_{i+1}^{1-\sigma_n} \cdot (y_i) = {\mathfrak B}_i$~; d'o\`u 
${\rm N}_{k_n/k}({\mathfrak A}_{i+1})$ dont la classe dans 
${\rm N}_{k_n/k}(M_{i+1}^n) \supseteq {\rm N}_{k_n/k}(M_{i}^n)$ assure la 
croissance de la $i$-suite $\order {\rm N}_{k_n/k}(M^n_i) \mid \order \Cl_k$. 

\smallskip
L'algorithme est donc identique et d\'efinit, pour chaque $n \gg 0$, la $i$-suite croissante 
des diviseurs $\order {\rm N}_{k_n/k}(M_i^n)$ de $\order \Cl_k$ pour laquelle, sous l'hypoth\`ese 
$\lambda \geq 1$ ou $\mu \geq 1$, l'algorithme devrait comporter 
un palier de $O(\lambda\cdot n + \mu \cdot p^n)$ ``$i$-\'etapes'' cons\'ecutives 
pour lesquelles ${\rm N}_{k_n/k}(M_i^n) \subset \Cl_k$, 
ce qui est totalement irr\'ealiste lorsque $n\to\infty$ comme expliqu\'e au point pr\'ec\'edent,
ce qui sera analys\'e au \S\,\ref{resglob}.

\smallskip
Ce premier cas de la conjecture de Greenberg est donc susceptible de l'Heuristique 
\ref{heur2} pr\'ec\'edente affirmant l'existence de $i_2$ assez grand tel que 
${\rm N}_{k_n/k}( M^n_{i_2}) = \Cl_k$ pour $n\to\infty$. Comme expliqu\'e dans
l'Introduction, le principe ``algorithmique'' adopt\'e couvre grosso modo la totalit\'e
de la conjecture de Greenberg pour les corps totalement r\'eels.
\end{remark}

\subsection{Principales propri\'et\'es $p$-adiques 
dans $k_\infty/k$} \label{autres}
On a les propri\'et\'es locales et globales suivantes qui confortent l'absence d'obstruction
dans les consid\'erations et heuristiques probabilistes pr\'ec\'edentes sur les propri\'et\'es 
normiques des nombres et id\'eaux dans $k_\infty/k$~; le \S\,\ref{resglob} montrera,
de fa\c con essentielle, que l'analyse probabiliste dans la tour
est de fait de ``type fini'' et ne concerne que l'invariant ${\mathcal T}_k$.

\subsubsection{Propri\'et\'es locales dans $k_\infty/k$} 
On consid\`ere les $d$ normes locales dans $k_n/k$, en les ${\mathfrak p} \mid p$,
$n \geq 0$ fix\'e~; on d\'e\-signe par
${\mathfrak p}_n$ l'id\'eal premier de $k_n$ au-dessus de ${\mathfrak p}$.
La proposition suivante est une formulation exclusivement locale des 
calculs de la Section \ref{hasse}~:

\begin{proposition}\label{ccloc} 
Soit $\pi_{n}$ (resp. $\pi_{n+h}$) une uniformisante du compl\'et\'e 
$\overline k_{n}$ de $k_{n}$ en ${\mathfrak p}_{n} \mid p$ (resp. de celui de $k_{n+h}$ 
en ${\mathfrak p}_{n+h} \mid p$). 

\smallskip
(i) Pour tout $h\geq 1$ et toute unit\'e locale $\alpha_{n+h}$ de $\overline k_{n+h}$, 
il existe une unit\'e locale $\alpha_n$ de $\overline k_n$ telle que l'on ait
${\rm N}_{\overline k_{n+h}/\overline k_n} (1+ \alpha_{n+h} \cdot \pi_{n+h}) 
= 1+ \alpha_n \cdot \pi_{n}$.

(ii) Pour tout $h\geq 1$, il existe une unit\'e locale $\beta_n$ de $\overline k_n$ 
telle que l'on ait $(1+ \alpha_{n+h} \cdot \pi_{n+h})^{p^h} = 1+ \beta_n \cdot \pi_{n}$.

\smallskip
(iii) Soient $U_{{\mathfrak p}_n}$ et $U_{\mathfrak p}$ les groupes 
d'unit\'es locales principales de $\overline k_n$ et $\overline k$. 
Alors l'indice de ${\rm N}_{\overline k_n/ \overline k} (U_{{\mathfrak p}_n})$
dans $U_{\mathfrak p}$ est \'egal \`a $p^n$.
\end{proposition}

\begin{proof} Utiliser  \cite[Ch.\,V, \S\,3, Lemme 5 \& Ch.\,V, \S\,3, Lemme~4]{Se1}
pour le premier point. Le second est imm\'ediat et le dernier est la th\'eorie du corps 
de classes local dans l'extension totalement ramifi\'ee $\overline k_n/\overline k$.
\end{proof}

\subsubsection{Propri\'et\'es globales en $p$-ramification 
Ab\'elienne dans $k_\infty/k$} \label{resglob}
On consid\`ere (sous la conjecture de Leopoldt pour $p$ dans $k_\infty$)
le sch\'e\-ma suivant o\`u l'on pose $K=k_n$ pour simplifier et o\`u 
$H_k^{\rm pr}$ et  $H_K^{\rm pr}$ sont les pro-$p$-extensions Ab\'eliennes 
$p$-ramifi\'ees maximales de $k$ et $K$~; on d\'esigne par $F$ 
une extension de $k$, contenant le $p$-corps de classes de Hilbert $H_k$ de $k$,
telle que $H_k^{\rm pr}$ soit le compos\'e direct de $F$ et $k_\infty$ sur $k$~:
\unitlength=0.94cm 
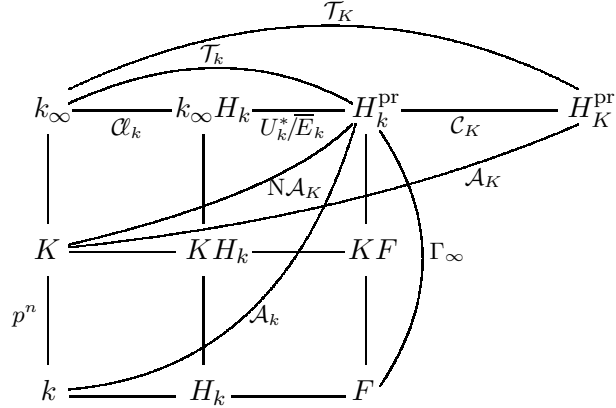
$$\vbox{\hbox{\hspace{-2.8cm}  \begin{picture}(11.5,5.9)
\put(6.4,4.50){\line(1,0){1.3}}
\put(6.4,2.50){\line(1,0){1.3}}
\put(6.1,0.45){\line(1,0){1.6}}
\put(8.5,4.50){\line(1,0){2.2}}
\put(3.85,4.50){\line(1,0){1.4}}
\put(3.8,2.50){\line(1,0){1.5}}
\put(3.8,0.45){\line(1,0){1.5}}
\put(9.2,4.18){\footnotesize${\mathcal C}_{K}$}
\bezier{350}(3.8,4.8)(7.6,6.6)(11.0,4.8)
\put(7.4,5.8){\footnotesize${\mathcal T}_{K}$}
\bezier{350}(3.8,4.6)(6.0,5.6)(7.8,4.6)
\put(5.65,5.2){\footnotesize${\mathcal T}_{k}$}
\put(3.50,2.9){\line(0,1){1.25}}
\put(3.50,0.9){\line(0,1){1.25}}
\put(5.7,2.9){\line(0,1){1.25}}
\put(5.7,0.9){\line(0,1){1.25}}
\put(8.0,0.9){\line(0,1){1.25}}
\put(8.0,2.9){\line(0,1){1.25}}
\put(4.4,4.15){\footnotesize$\Cl_k$}
\bezier{400}(3.8,2.56)(8.0,3.0)(11.0,4.3)
\put(9.4,3.45){\footnotesize${\mathcal A}_{K}$}
\bezier{400}(3.8,0.55)(6.8,0.8)(7.85,4.3)
\put(6.35,1.5){\footnotesize${\mathcal A}_k$}
\bezier{400}(3.8,2.6)(6.8,3.3)(7.8,4.3)
\put(6.6,3.3){\footnotesize${\rm N}{\mathcal A}_{K}$}
\put(10.85,4.4){$H_{K}^{\rm pr}$}
\put(5.3,4.4){$k_\infty H_k$}
\put(7.8,4.4){$H_k^{\rm pr}$}
\put(6.5,4.15){\footnotesize$U_k^*\!/\!\overline E_k$}
\put(3.3,4.4){$k_\infty$}
\put(5.45,2.4){$K H_k$}
\put(7.75,2.4){$K F$}
\put(5.5,0.4){$H_k$}
\put(7.8,0.4){$F$}
\put(3.3,2.4){$K$}
\put(3.4,0.40){$k$}
\put(3.0,1.5){\footnotesize $p^n$}
\bezier{400}(8.2,0.6)(9.4,2.4)(8.2,4.2)
\put(8.9,2.4){\footnotesize$\Gamma_\infty$}
\end{picture}   }} $$
\unitlength=1.0cm

On consid\`ere les symboles d'Artin $\big( \frac{H_k^{\rm pr}/k}{\cdot} \big)$ et 
$\big( \frac{H_{K}^{\rm pr}/K}{\cdot} \big)$
sur ${\mathcal I}_k :=  I_k \otimes \Z_p$ et ${\mathcal I}_K :=I_{K} \otimes \Z_p$,
o\`u $I_k$ et $I_{K}$ sont les groupes des id\'eaux \'etrangers \`a $p$ de $k$ et $K$.
Leurs images sont respectivement les groupes de Galois ${\mathcal A}_k$ et ${\mathcal A}_{K}$~;
les noyaux de ces symboles d'Artin sont les groupes d'id\'eaux principaux infinit\'esimaux 
${\mathcal P}_{k, \infty} \subset {\mathcal I}_k$ et 
${\mathcal P}_{K, \infty}  \subset {\mathcal I}_K$, o\`u ${\mathcal P}_{k, \infty}$ est 
l'ensemble des id\'eaux $(x_\infty)$ o\`u $x_\infty \in k^\times  \otimes \Z_p$ est d'image 
triviale dans $U_k$, et de m\^eme avec $K$ (cf. \cite[Theorem III.2.4, Proposition III.2.4.1]
{Gra1}, \cite[\S\,2]{J2}). 

\smallskip
Les groupes ${\mathcal T}_k$ et ${\mathcal T}_K$ sont les groupes 
de torsion respectifs de ${\mathcal A}_k$ et ${\mathcal A}_{K}$.

\smallskip
L'application ${\rm N}_{K/k}$ sur  ${\mathcal I}_K$ induit la restriction 
${\mathcal A}_{K} \too {\mathcal A}_k$ qui s'exprime par la suite exacte suivante,
o\`u $\Gamma_\infty = {\rm Gal}(H_k^{\rm pr}/F)$ et o\`u ${\rm N}_{K/k}({\mathcal A}_{K})$
est d'indice $p^n$ dans ${\mathcal A}_k$~:
\begin{equation*}
1 \too {\mathcal C}_{K} \tooo {\mathcal A}_{K} \tooo {\rm N}_{K/k}({\mathcal A}_{K}) =  
\Gamma_\infty^{p^n} \oplus {\mathcal T}_{k} \too 1.
\end{equation*}

Il en r\'esulte que si ${\mathfrak A} \in I_{K}$ (id\'eal ordinaire vu dans ${\mathcal I}_K$), 
${\rm N}_{K/k}({\mathfrak A})$ peut s'\'ecrire, 
modulo des id\'eaux principaux infinit\'esimaux, sous la forme~:

\smallskip
\centerline{${\rm N}_{K/k}({\mathfrak A}) = {\mathfrak a}^{p^n}\!\! \cdot {\mathfrak t}, \  
\hbox{ avec ${\mathfrak a},\, {\mathfrak t} \in {\mathcal I}_k$,  
$\ \ \Big( \frac{H_k^{\rm pr}/k}{{\mathfrak a}} \Big) \in \Gamma_\infty\,$
 \&  $\, \Big( \frac{H_k^{\rm pr}/k}{{\mathfrak t}} \Big) \in {\mathcal T}_{k}$.}$ }

Pour $n \gg 0$, ${\mathfrak a}^{p^n}$ est 
$p$-principal de la forme $(\alpha)$, avec $\alpha \equiv 1 \pmod {p^{n'}}$
o\`u $n' \to \infty$ avec $n$~; pour la d\'etermination des 
$\delta_{\mathfrak p}(x)$ pour $x \in k^\times$ de la forme
$(x) = {\rm N}_{K/k}({\mathfrak A})$, on aura 
$(x) = (\alpha(x)) \cdot {\mathfrak t}(x)$ o\`u le nombre $\alpha(x)$ 
est ``n\'egligeable'' quant \`a ses $\delta_{\mathfrak p}(\alpha(x))$, 
et les $\delta_{\mathfrak p}(x)$ ne d\'ependront que de ${\mathfrak t}(x)$
modulo $\delta_{\mathfrak p}(E_k)$~;
on peut ainsi dire que les groupes 
${\rm N}_{K/k}({\mathcal I}_i^n)$ sont engendr\'es (modulo des $(\alpha)$
``quasi-infinit\'esimaux'') par des ${\mathfrak t} \in  {\mathcal T}_{k}$,
et que les groupes $\Lambda_i^n$ sont obtenus via les id\'eaux principaux $(x)$
qui s'en d\'eduisent.

\begin{hypothesis}\label{hypo}
{\it On suppose que les id\'eaux ${\mathfrak A}$ de $K=k_n$, obtenus
par l'algorithme, d\'efinissent une variable al\'eatoire ainsi que la composante 
${\mathfrak t}$ de ${\rm N}_{K/k}({\mathfrak A})$ et que  
$\big(\frac{F/k}{{\mathfrak t}} \big)$ parcourt uniform\'ement 
${\rm Gal}(F/k) \simeq {\mathcal T}_{k}$.}
\end{hypothesis}

Ceci a les cons\'equences essentielles suivantes~:

\smallskip
(i) La classe de ${\mathfrak t}$ parcourt uniform\'ement $\Cl_k$.

\smallskip
(ii) Lorsque ${\mathfrak t} = (x)$, puisque ${\rm Gal}(F/H_k) 
\simeq {\rm Gal}(H_k^{\rm pr}/k_\infty \,H_k) \simeq U_k^*/\overline E_k$ 
est d'exposant $p^e$, on a $x^{p^e} = x_\infty\! \cdot \varepsilon$, avec
$\varepsilon \in E_k \otimes \Z_p$ et $x_\infty$ infini\-t\'esimal 
(donc ${\rm N}_{k/\Q}(x)=1$ dans $U_\Q$), et l'image 
de $x$ est d\'efinie dans $U_k^*/\overline E_k$. 

\smallskip
Ainsi les familles $\big (\delta_{\mathfrak p}(x) \big)_{{\mathfrak p} \mid p}$
modulo les $\big (\delta_{\mathfrak p}(\varepsilon) \big)_{{\mathfrak p} \mid p}$ 
parcourent un domaine repr\'esentatif effectif fini, ne 
d\'ependant que de $E_k$, et soumis aux probabilit\'es habituelles sur les quotients 
de Fermat~; dans le cas quadratique, si $\delta_p(\varepsilon) = r \geq 1$, ce domaine 
est $[0, r[$ car si $\delta_p(x) \geq r$, $\delta_p(x \cdot E_k) = [r, \infty[$.

\smallskip
Si l'on revient aux $i$-suites des ${\rm N}_{K/k}({\mathcal I}_i^n)$, 
repr\'esentant ${\rm N}_{K/k}(M_i^n)$, et aux groupes $\Lambda_i^n = 
\{x \in k^\times,\  (x) \in {\rm N}_{K/k}({\mathcal I}_i^n)\}$,
ce qui pr\'ec\`ede justifie \`a nouveau les Heuristiques \ref{heur1}, 
\ref{heur2} et l'existence de $i_0 \geq {\rm max}(i_1, i_2)$, ind\'ependant de $n$, 
assurant la finitude du processus limite (D\'efinition \ref{process}) 
et conduisant \`a l'Heuristique probabiliste finale \ref{probas}.

\subsection{Conjecture de Greenberg faible}
Elle s'\'enonce (pour $k$ r\'eel, $p$-d\'ecompos\'e) sous la forme~:

\smallskip
\centerline {``{\it $\cl_{k_n}({S_{k_n}}) = 1$ pour tout $n$ $\ \Longrightarrow\ $ $\lambda = \mu = 0\,$}'' }

\smallskip
(voir \cite{J6} pour diff\'erentes conditions \'equivalentes sous la conjecture de Gross-Kuzmin
et \cite{Ng1}, \cite{Ng2} pour une preuve dans le cadre Ab\'elien sous certaines hypoth\`eses sur les 
unit\'es cyclotomiques). Or le crit\`ere de Greenberg  (Th\'eor\`eme \ref{green}) s'\'enonce~:

\smallskip
\centerline{\it  $\lambda = \mu = 0$ si et seulement si  $\Cl_{k_n}^{G_n} = 
\cl_{k_n}({S_{k_n}})$ pour tout $n \gg 0$.}
  
\smallskip
d'o\`u il r\'esulte l'implication conjecturale qui serait \`a d\'emontrer~:

\smallskip
\centerline{\it  $\cl_{k_n}({S_{k_n}}) = 1$ pour tout $n$ $\ \Longrightarrow\ $
$\Cl_{k_n}^{G_n} = 1$ pour tout $n \gg 0$.}

\smallskip
Puisque $\order \Cl_{k_n}^{G_n} = \order {\mathcal T}_k$ pour
$n \geq e$ (Th\'eor\`eme \ref{regulateur} (ii)), ceci devient~:

\smallskip
\centerline{\it  $\cl_{k_n}({S_{k_n}}) = 1$ pour tout $n$ $\ \Longrightarrow\ $
$k$ est $p$-rationnel (i.e., $\lambda=\mu=\nu=0$).}

\medskip
Si l'on suppose $k$ non $p$-rationnel 
(e.g., $p=3$, $k=\Q(\sqrt {103})$ o\`u $\Cl_k = 1$ et $R_k=3$),
alors n\'ecessairement $\cl_{k_{n}}(S_{k_{n}}) \ne 1$ pour tout $n \gg 0$.%
\footnote{Noter que les $\cl_{k_{n+h}}({S_{k_{n+h}}})$ se surjectent (par la norme)
sur $\cl_{k_n}(S_{k_n})$, ce qui fait que si l'un des groupes est nul \`a un \'etage
$n_0 \gg 0$, ils sont tous nuls pour $n \leq n_0$.}

\smallskip
D'apr\`es la Remarque \ref{cohomo} (ii), on a
${\rm H}^1({G_n}, \cl_{k_n}({S_{k_n}}))=0$ si et seulement si
$\Cl_{k_n}(S_{k_n}) =1$, donc  si et seulement si
$\Cl_{k_n}^{G_n} = \Cl_{k_n}^{S_{k_n}{\!G_n}}$
(autrement dit, dans ce cas, la condition suffisante du 
Th\'eor\`eme \ref{thm2} est aussi n\'ecessaire).

\smallskip
Dans le cadre de l'algorithme d\'eterminant, au niveau $n$, les $i$-suites 
$M_i^n$ et $\Lambda_i^n$, la condition  $\cl_{k_n}({S_{k_n}}) = 1$
ne modifie en rien les aspects probabilistes du \S\,\ref{resglob} pr\'ec\'edent, puisqu'on 
ne travaille qu'avec des id\'eaux \'etrangers \`a $p$ pour repr\'esenter les classes
de $k_n$, calculer leurs normes dans $\Cl_k$ et les quotients de Fermat des 
$x \in \Lambda_i^n$~; la conjecture faible ne semble pas \^etre de nature 
diff\'erente de celle de la conjecture g\'en\'erale.

\subsection{Heuristique finale}
Les Heuristiques \ref{heur1}, \ref{heur2}, et les arguments du 
\S\,\ref{resglob} se r\'esumeraient par l'existence de $i_0$ assez grand,
ind\'ependant de $n$, tel que pour tout  $n \gg 0$, $\Cl_{k_n}$ soit 
atteint au $i_0$-i\`eme pas  au plus (i.e., $m_n \leq i_0$)~; ceci est alors \'equivalent
\`a $\order \Cl_{k_n} = p^\nu$ pour tout $n \gg 0$. 
Pour $0 \leq i \leq i_0$, on a donc une suite de diviseurs successifs 
de $\order  {\mathcal T}_k$, de la forme~:

\medskip
\centerline{$\{t_0 = \order {\mathcal T}_k, \ldots, t_i, \ldots, (t_{m_n}= \cdots = t_{i_0}) =1\}$, }

\smallskip
d\'efinis, pour chaque $i \geq 0$ fix\'e, comme 
$\ds\max_{n \to \infty} \big(\order (M_{i+1}^n/M_i^n) \big)$ (cf. \S\,\ref{appgene}). 

\smallskip
On peut donc proposer l'heuristique probabiliste suivante (pr\'ecisant les fondements 
$p$-adiques de la conjecture de Greenberg), reposant sur les heuristiques ci-dessus et 
en notant que, grosso modo, l'existence de chacune des $m_n = O(1) \cdot
\big(\lambda \cdot n + \mu \cdot p^n + \nu \big)$ \'etapes de l'algorithme 
a une probabilit\'e fonction de celles des $\delta_{\mathfrak p}$ \`a avoir une valeur ``non triviale''
(un cas embl\'ematique simple \'etant celui des $\order (M_{i+1}^n/M_i^n)=p$ pour tout $i$
pour un corps quadratique~; dans le cas g\'en\'eral, la formule propos\'ee en (iii) est 
seulement un ordre de grandeur largement suffisant pour conclure)~:

\begin{conjecture} \label{probas}{\it
Soit $k$ un corps de nombres totalement r\'eel et soit $p>2$ totalement d\'ecompos\'e 
dans $k/\Q$. Soit ${\mathcal T}_k$ le groupe de torsion du groupe de Galois 
de la pro-$p$-extension Ab\'elienne $p$-ramifi\'ee maximale de $k$~; on suppose
${\mathcal T}_k\ne 1$.
On consid\`ere l'algorithme associ\'e \`a la filtration de $\Cl_{k_n}$, $n\geq 1$
fix\'e, et la $i$-suite des groupes $\Lambda_i^n$, $1 \leq i \leq m_n$, pour laquelle 
on a $m_n \geq  \frac{1}{v_p \big(\order  {\mathcal T}_k)} \, 
\big(\lambda \cdot n + \mu \cdot p^n + \nu \big)$ (Th\'eor\`eme \ref{O}). Alors~:

(i) La probabilit\'e que, pour un $x \in \Lambda_i^n$, on ait 
$\delta_{\mathfrak p}(x) \geq r$ (resp. $\delta_{\mathfrak p}(x) = 0$),
pour tout ${\mathfrak p} \mid p$, est  $ \frac{1}{p^{\,r \, (d-1)}}$
\Big(resp. $1 -  \frac{1}{p^{\, d-1}}$ \Big).

(ii) Soit $c \in \Cl_k$. La probabilit\'e que, pour un id\'eal ${\mathfrak A}$ de $k_n$ 
\'etranger \`a~$p$, la $p$-classe de ${\rm N}_{k_n/k}({\mathfrak A})$ soit \'egale \`a $c$,
est $ \frac{1}{\order \Cl_k}$.

(iii) Pour $\lambda$ ou $\mu$ non nuls, la probabilit\'e d'avoir
$\order \Cl_{k_n} = p^{\lambda \cdot n + \mu \cdot p^n+\nu }$
est au plus en $\frac{1}{p^{\,O(1)\,\cdot\, (\lambda \cdot n + \mu \cdot p^n)}}$,
lorsque $n \to \infty$.}
\end{conjecture}

\subsection{Perspectives}
Cette heuristique (m\^eme imparfaite) montre la l\'egitimit\'e de la conjecture de 
Greenberg mais aussi que la $n$-suite des $\order  \Cl_{k_n}$ est 
probablement tr\`es rapidement stationnaire, ce qui est plut\^ot un \'el\'ement favorable 
pour une approche plus analytique. 
Il semble en effet difficile de trouver des arguments th\'eoriques qui ``obligeraient'' les 
classes des normes ${\rm N}_{k_n/k}({\mathfrak A})$ d'id\'eaux al\'eatoires ${\mathfrak A}$ 
de $k_n$ \`a ne pas se r\'epartir uniform\'ement dans $\Cl_k$, ou les 
$\delta_{\mathfrak p}(x)$ des quotients de Fermat 
des $x \in k^\times$ \`a ne pas suivre des lois binomiales reposant sur
la probabilit\'e de nullit\'e en $1-\frac{1}{p}$ (voire encore plus proches de~$1$ 
comme nous l'avons longuement analys\'e dans \cite{Gra5} et 
\cite{Gra6} pour les nombres alg\'ebriques en g\'en\'eral). Voir \`a ce sujet \cite{HB}
qui \'etablit d\'ej\`a la r\'epartition uniforme des quotients de Fermat des entiers 
rationnels et qui est probablement g\'en\'erale.

\smallskip
Cette question risque de rester sans r\'eponse, d'autant plus qu'un \'eminent 
math\'ematicien fran\c cais m'avait confort\'e en affirmant, dans un \'echange au sujet des 
quotients de Fermat des rationnels (Janvier 2012)~: 

\smallskip
{\it ``Or vous savez bien que c'est la pire des situations~: quand quelque chose 
est al\'eatoire (par exemple un quotient de Fermat) on est souvent compl\`etement 
d\'esarm\'e pour le d\'emontrer''}. 

\smallskip
Ceci explique la difficult\'e rencontr\'ee en ``th\'eorie d'Iwasawa alg\'ebrique''
concernant le ``calcul'' des invariants $\lambda, \mu, \nu$ de la limite projective des $\Cl_{k_n}$ 
pour le cas totalement r\'eel en l'absence d'une conjecture tr\`es
ambitieuse qui gouvernerait beaucoup de probl\`emes arithm\'etiques 
analogues o\`u interviennent de tels objets $p$-adiques  ``al\'eatoires'', d'autant plus que 
tout ce qui pr\'ec\`ede repose sur la conjecture de Leopoldt~;
or celle-ci revient \`a dire grosso modo que la probabilit\'e, pour le r\'egulateur $R_k$,
d'\^etre divisible par $p^n$ est en $\frac{1}{p^{f(n)}}$, $f(n) \to \infty$ si $n \to \infty$,
c'est-\`a-dire analogue \`a celle de l'Heuristique \ref{probas} (iii), et que $R_k=0$ 
est ``presque s\^urement'' impossible.

\smallskip
On peut donc se demander si une aproche de type transcendance $p$-adique
ne serait pas mieux adapt\'ee puisque c'est l'une des voies d\'emontrant quelques 
cas non triviaux de la conjecture de Leopoldt.

\smallskip
Tout ceci est aussi li\'e \`a un point de vue diff\'erent
qui est celui de fixer le corps $k$ et de faire tendre $p$ vers l'infini, auquel cas
un cadre conjectural $p$-adique analogue \cite[Section 7]{Gra6} conduirait \`a la 
$p$-rationalit\'e de $k$ pour tout premier $p \gg 0$, donc \`a la conjecture de Greenberg 
pour laquelle on aurait $\lambda_p = \mu_p = \nu_p =0$ 
ind\'ependamment de toute technique d'Iwasawa.

\section{Statistiques sur les $\delta_{\mathfrak p}(x)$ 
et les classes des ${\rm N}_{k_n/k}({\mathfrak A})$} \label{BC}\label{sect8}
Le point essentiel est le comportement, pour $n$ fix\'e, des 
``quotients de Fermat'' ${\mathfrak p}^{\delta_{\mathfrak p}(x_i)}$ (D\'efinition \ref{delta}), 
pour les $x_i \in \Lambda_i^n$ tels que $(x_i) \in {\rm N}_{k_n/k}({\mathcal I}^n_i)$
ainsi que les classes des id\'eaux ${\rm N}_{k_n/k}({\mathfrak A}_i)$
o\`u les ${\mathfrak A}_i \in {\mathcal I}^n_i$ repr\'esentent 
les classe de $M^n_i$.

\smallskip
L'\'etape de l'algorithme calculant le $(i+1)$-i\`eme sous-groupe $M^n_{i+1}$ 
de la filtration repose sur les $x_i = {\rm N}_{k_n/k}(y_i) \in \Lambda_i^n \cap 
{\rm N}_{k_n/k}(k_n^\times)$, tels que $(x_i) =  {\rm N}_{k_n/k}( {\mathfrak B}_i)$, 
${\mathfrak B}_i \in {\mathcal I}_i^n$, ce qui conduit \`a 
${\mathfrak A}_{i+1}^{1-\sigma_n} \cdot{\mathfrak B}_{i} = (y_i)$,
${\mathfrak A}_{i+1}$ \'etranger \`a $p$ dont on prend la norme, etc.

\smallskip
On a alors $\omega_n(\Lambda_{m_n}^n) = \Omega(k_n/k)$ et 
${\rm N}_{k_n/k}(M_{m_n}^n) = \Cl_k$, pour le nombre de pas ${m_n}$
de l'algorithme.

\smallskip 
D'apr\`es le Th\'eor\`eme \ref{O}, en supposant par exemple
${\mathcal T}_k\ne 1$ et $\Cl_k = 1$, 
l'hypoth\`ese $\lambda$ ou $\mu \geq 1$ doit conduire \`a au moins
$\frac{1}{v_p(\order  {\mathcal T}_k)} \, (\lambda \cdot n + \mu \cdot p^n +  \nu)$
valeurs cons\'ecutives de l'indice $i$ pour lesquelles tous les 
$x_i$ sont tels que $\delta_{\mathfrak p}(x_i) \geq 1$ pour certains ${\mathfrak p} \mid p$
de telle sorte que $\omega_n(\Lambda_i^n) \subset \Omega(k_n/k)$~; 
ceci rend indispensable la confrontation avec l'exp\'erimentation num\'erique.

\smallskip
Pour cela on commence par un cadre simple, mais non trivial, utilisant des 
id\'eaux ${\mathfrak A}_i$ premiers, dans le cas $\Cl_k=1$.
Ensuite, on aura \`a examiner le facteur ``classes'' et le facteur ``normique''
dans l'expression de $M_{i+1}^n/M_i^n$.

\subsection{Repr\'esentation des classes par des id\'eaux 
premiers ${\mathfrak L}_i$}\label{L}

D'apr\`es le th\'eor\`eme de Tchebotarev, pour
tout $c \in \Cl_{k_n}$ il existe une infinit\'e de
nombres premiers $\ell$ tels que, pour un id\'eal premier convenable 
${\mathfrak L}' \mid \ell$ dans $H_{k_n}$ (le $p$-corps de classes 
de Hilbert de $k_n$), $\Big(\frac{H_{k_n}/\Q}{{\mathfrak L}'} \Big)$ 
soit l'\'el\'ement de ${\rm Gal}(H_{k_n}/k_n)$ correspondant \`a $c$ 
par le corps de classes~; comme $H_{k_n}/k_n$ est Ab\'elienne, ce 
Frobenius ne d\'epend que de l'id\'eal premier ${\mathfrak L}$ de $k_n$ au-dessous de
${\mathfrak L}'$. L'image de $\Big(\frac{H_{k_n}/k_n}{{\mathfrak L}} \Big)$ 
dans ${\rm Gal}(H_{k_n}/k_n)$ est encore $c$ et est repr\'esent\'ee par 
${\mathfrak L}$ totalement d\'ecompos\'e dans $k_n/\Q$.

\smallskip
Autrement dit, lorsque par exemple $k$ est principal,
en se limitant \`a des id\'eaux premiers ${\mathfrak L}$ totalement 
d\'ecompos\'es dans $k_n/\Q$ et en consid\'erant ${\mathfrak l} = (\alpha)$
pour l'id\'eal premier ${\mathfrak l} = {\rm N}_{k_n/k}({\mathfrak L})$ de $k$ 
en-dessous de ${\mathfrak L}$, on peut effectuer des statistiques sur les valeurs prises 
par les $\delta_{\mathfrak p} (\alpha)$ {\it ind\'ependamment de tout contexte conjecture de 
Greenberg} puisque les ${\mathfrak L}_i$ de l'algorithme seront al\'eatoirement certains 
${\mathfrak L}$ particuliers.

\smallskip
On reprend ici les m\^eme hypoth\`ese simplificatrices faites au \S\,\ref{idx},
o\`u le corps $k = \Q(\sqrt m)$ est principal et les entiers $\delta_p(\varepsilon)$ 
et $\delta_{\mathfrak p}(\pi)$ non nuls, de sorte que la condition 
suffisante du Th\'eor\`eme \ref{thm1} ne s'applique pas (en r\'ealit\'e, les r\'esultats 
n'en d\'ependent pas.

\smallskip
Le nombre $\alpha$ tel que ${\rm N}_{k_n/k}({\mathfrak L}) = {\mathfrak l} = (\alpha)$
est un entier de $k$ (unique \`a une unit\'e pr\`es) de norme 
$\ell = {\mathfrak l} \cap \Z$ sur $\Q$, o\`u $\ell^{p-1} \equiv 1\!\! \pmod {p^{n+1}}$.
On a $\alpha^{p-1} = 1 + p\cdot p^{\delta_p(\alpha)} \cdot \beta$, 
$\beta$ \'etranger \`a $p$ si $\delta_p(\alpha) < n$.

\smallskip
Le choix par PARI de $\alpha$ modulo $E_k$ influe de fa\c con marginale sur
les statistiques (voir ci-apr\`es un cas sans \'equivoque avec $\delta_3(\varepsilon)=6$).
Ceci vaut pour tous les r\'esultats de cette section o\`u le ``choix'' de $\alpha$
intervient.

\smallskip
(i) Dans l'exp\'erimentation num\'erique, pour $p=3$, on prend $n$ assez grand
de sorte que l'on compte dans $C_r$ les nombres premiers
$\ell \equiv \pm1 \pmod {3^{n+1}}$ pour lesquels
$\delta_p(\alpha) = r$, o\`u $r$ varie de $0$ \`a $n-1$ au plus, et on 
compare le r\'esultat aux probabilit\'es naturelles
$\frac{2}{3}, \,  \frac{2}{3^2}, \, \cdots \, , \, \frac{2}{3^r}, \, \cdots$.
\footnotesize
\begin{verbatim}
{p=3; m=103; n=12; B=10^13; M=p^(n+1); Q=x^2-m;
K=bnfinit(Q,1); C0=0; C1=0; C2=0; C3=0; C4=0; C5=0; NL=0;
for(t=-1, 0, L=2*t+1; while(L<B, L=L+2*M;
if(isprime(L)==1 & kronecker(m,L)==1, NL=NL+1;
Su=bnfsunit(K, idealprimedec(K,L));
A=component(component(Su,1),1);
AA=(Mod(A,Q)^2-1)/3; v=valuation(AA,3);
if(v==0, C0=C0+1); if(v==1, C1=C1+1); if(v==2, C2=C2+1);
if(v==3, C3=C3+1); if(v==4, C4=C4+1); if(v>=5, C5=C5+1))));
print(p,"   ",m,"   ",n,"   ",B);
print(NL,"  ",C0," ",C1," ",C2," ",C3," ",C4," ",C5);
print(C0/NL+0.0,"  ",C1/NL+0.0,"  ",C2/NL+0.0,"  ",C3/NL+0.0,"  ",
C4/NL+0.0,"  ",C5/NL+0.0); print("  ");
S=0.0; for(j=1,8, S=S+(p-1.0)/p^(5+j));
print(2./3,"  ",2./9,"  ",2./27,"  ",2./81,"  ",2./243,"  ", S)}
\end{verbatim}

\normalsize
On obtient la remarquable confirmation du fait que les 
$\delta_p(\alpha)$ se r\'epartissent de fa\c con ind\'ependante de la condition 
$\ell \equiv \pm1 \pmod {3^{n+1}}$, quel que soit~$n$. 
On consid\`ere l'exemple ci-dessous pour lequel $\delta_p(\varepsilon)=
\delta_{\mathfrak p}(\pi)=1$~:

\smallskip
\centerline{$p=3$, $\ m=103$, $\ n=12$, $\ \ell \equiv \pm 1 \pmod {3^{13}}$,
$\ 1 <\ell <10^{13}$.}

\smallskip
Il y a alors $N_L = 325644 $ nombres premiers $\ell$ dans l'intervalle, dont respectivement 
$C_0= 217122 $, $C_1= 72353 $, $C_2= 24174 $, $C_3= 8043 $, $C_4= 2620 $, $C_5= 1332 $, 
sont tels que $\delta_p(\alpha)= 0, 1, 2, 3, 4, \geq 5$.
\footnotesize
$$\begin{array}{cccc}
& \hbox{proportions} &   & \hbox{probabilit\'es} \vspace{0.1cm}\\ 
& C_0= 0.6667465084 
 &\hspace{1.0cm}   \frac{2}{3} & \hspace{-0.3cm} = 0.6666666666 \\
& C_1= 0.2221843485
&\hspace{1.0cm} \frac{2}{3^2} & \hspace{-0.3cm}= 0.2222222222 \\
& C_2= 0.0742344400  
&\hspace{1.0cm} \frac{2}{3^3} & \hspace{-0.3cm}= 0.0740740740 \\
& C_3= 0.0246987507  
&\hspace{1.0cm} \frac{2}{3^4} & \hspace{-0.3cm}= 0.0246913580 \\
& C_4= 0.0080455958
&\hspace{1.0cm}  \frac{2}{3^5} & \hspace{-0.3cm}= 0.0082304526 \\
& C_5= 0.0040903563  
&  \sum_{j \geq 6}\frac{2}{3^j} &\hspace{-0.3cm} = 0.0041152264
\end{array} $$

\normalsize
Le m\^eme calcul pour $m = 2149$, o\`u $\delta_p(\varepsilon) =
\delta_{\mathfrak p}(\pi)= 3$, conduit \`a 
des r\'esultats similaires~: $N_L=325538$, $C_0=216955$, $C_1=72406$, $C_2=24145$, 
$C_3=8063$, $C_4=2667$, $C_5=1302$ (premi\`ere proportion $0.6664506140$).

\smallskip
L'exemple suivant pour lequel $\delta_3(\varepsilon) = 6$ est donc tel que
$\delta_3(\alpha)$ est ind\'ependant du repr\'esentant modulo $E_k$ donn\'e
par PARI, au moins pour $r \leq 5$ ($m=1213$, $n=12$, $B=10^{13}$)~:
$N_L= 325778$, $C_0= 217298$, $C_1= 72322$, $C_2= 24144$, $C_3= 8075$, 
$C_4= 2621$, $C_5= 1318$, avec les proportions 
$0.6670125054$, $0.2219978021$, $0.0741118184$, $0.0247868180$,
$0.0080453560$, $0.0040456998$. 

\smallskip
En prenant $m= 397$, $n=100$, $B=10^{54}$, et $\ell \equiv \pm 1 \pmod {3^{101}}$,
on obtient les valeurs tr\`es stables~:
$N_L= 7947 $, $C_0= 5305 $, $C_1= 1702 $, $C_2= 646 $, $C_3= 207 $, 
$C_4= 55 $, $C_5= 32$, et la proportion $0.6675475022$ 
de $\delta_p(\alpha)=0$.

\medskip
(ii) Dans le cas g\'en\'eral, on suppose toujours $p$ d\'ecompos\'e dans $k$ 
et $\Cl_k=1$. Le programme est un peu plus complexe car on souhaite obtenir tous les
$\ell$ v\'erifiant la condition $\ell^{p-1} \equiv 1 \pmod{p^{n+1}}$ (on utilise les puissances 
$\rho^k$, $1 \leq k \leq p-1$, d'une racine primitive ($p-1$)-i\`eme de l'unit\'e $\rho \pmod {p^{n+1}}$).
\footnotesize
\begin{verbatim}
{p=7; m= 44853; n=5; B=10^12; M=p^(n+1); Q=x^2-m; K= bnfinit(Q,1); 
ro=znprimroot(M)^(p^n); C0=0; C1=0; C2=0; C3=0; C4=0; C5=0; NL=0; 
for(k=1, p-1, R=Mod(ro,M)^k; L=component(R,2); 
while(L<B, L=L+M; if(isprime(L)==1 & kronecker(m,L)==1, NL=NL+1; 
Su=bnfsunit(K, idealprimedec(K,L)); A=component(component(Su,1),1);
AA =(Mod(A,Q)^(p-1)-1)/p; v=valuation(AA,p); 
if(v==0, C0=C0+1); if(v==1, C1=C1+1); if(v==2, C2=C2+1); 
if(v==3, C3=C3+1); if(v==4, C4=C4+1); if(v>=5, C5=C5+1)))); 
print(p,"   ",m,"   ",n,"   ",B); 
print(NL,"  " ,C0," ",C1," ",C2," ",C3," ",C4," ",C5); 
print(C0/NL+0.0,"  ",C1/NL+0.0,"  ",C2/NL+0.0,"  ",C3/NL+0.0,"  ", 
C4/NL+0.0,"  ",C5/NL+0.0); S=0.0; for(j=1,10, S=S+(p-1.0)/p^(5+j));
print((p-1.0)/p,"  ",(p-1.0)/p^2,"  ",(p-1.0)/p^3,"  ",
                             (p-1.0)/p^4,"  ",(p-1.0)/p^5,"  ", S)}
\end{verbatim}

\normalsize
On consid\`ere l'exemple pour lequel $\delta_p(\varepsilon)=
\delta_{\mathfrak p}(\pi) =1$~:
$p=7$, $m= 44853 $, $n=5$, $\ell^{p-1} \equiv 1 \pmod {7^{6}}$,
$1 <\ell <10^{12}$. On a alors
$N_L =  1118955$, $C_0= 959051$, $C_1= 137487$, $C_2= 19118$, 
$C_3= 2842$, $C_4= 388$, $C_5= 69$.
\footnotesize
$$\begin{array}{cccc}
& \hbox{proportions} & & \hbox{probabilit\'es} \vspace{0.1cm}\\ 
& C_0=   0.8570952361
 &\hspace{1cm} \frac{6}{7} &  \hspace{-0.3cm}= 0.8571428571  \\
& C_1=  0.1228708929
&\hspace{1cm} \frac{6}{7^2} & \hspace{-0.3cm} =  0.1224489795 \\
& C_2=    0.0170855843
&\hspace{1cm} \frac{6}{7^3} & \hspace{-0.3cm} = 0.0174927113  \\
& C_3=     0.0025398697
&\hspace{1cm}  \frac{6}{7^4} &\hspace{-0.3cm} = 0.0024989587  \\
& C_4=  0.0003467521
&\hspace{1cm}  \frac{6}{7^5} & \hspace{-0.3cm}=  0.0003569941 \\
& C_5=    0.0000616646
& \sum_{j \geq 6}\frac{6}{7^j} &\hspace{-0.3cm} =  0.0000059499
\end{array} $$

 \normalsize
On consid\`ere enfin le cas suivant pour lequel $\delta_p(\varepsilon)=
\delta_{\mathfrak p}(\pi)=1$~: 
$p=29$, $m= 683 $, $n=5$, $\ell^{p-1} 
\equiv 1 \pmod {29^{6}}$, $1 <\ell <10^{15}$.

\smallskip
$N_L =  728880$,
$C_0= 703535 $, $C_1= 24450 $, $C_2= 857 $, $C_3= 38 $, $C_4=C_5= 0$.
\footnotesize
$$\begin{array}{rlcc}
& \hbox{\ \ \ \ \  proportions} & & \hbox{probabilit\'es} \vspace{0.1cm}\\ 
& C_0=  0.9652274722
 &\hspace{1cm} \frac{28}{29} &\hspace{-0.3cm} = 0.9655172413 \\
& C_1= 0.0335446163
&\hspace{1cm}  \frac{28}{29^2} & \hspace{-0.3cm}= 0.0332936979 \\
& C_2=   0.0011757765
&\hspace{1cm}  \frac{28}{29^3} &\hspace{-0.3cm} =0.0011480585  \\
& C_3=   0.0000521347
&\hspace{1cm}  \frac{28}{29^4} &\hspace{-0.3cm} = 0.0000395882 \\
& C_4= 0
&\hspace{1cm} \frac{28}{29^5} &\hspace{-0.3cm} = 0.0000013651 \\
& C_5=   0
&  \sum_{j \geq 6}\frac{28}{29^j} &\hspace{-0.3cm} = 0.0000000488
\end{array} $$

\normalsize
(iii) Les groupes 
$\Lambda_i^n = \{x \in k^\times, \,  (x) \in {\rm N}_{k_n/k}({\mathcal I}^n_i) \}$
des exemples pr\'ec\'edents sont de la forme $\Lambda_i^n = 
\langle \varepsilon, \alpha_1, \ldots, \alpha_i\rangle$ pour $1 \leq i \leq m_n$ 
(\S\,\ref{Lambda}), et la probabilit\'e d'avoir $m_n$ g\'en\'erateurs
(outre $\varepsilon$) suit une loi binomiale {\it ind\'ependamment de la valeur 
de $n \to \infty$}, ce qui sugg\`ere une  probabilit\'e nulle d'avoir
 $\lambda$ ou $\mu \geq 1$, c'est-\`a-dire $m_n=O(\lambda  \cdot n + \mu \cdot p^n)$ 
 (Th\'eor\`eme \ref{O}).

\medskip
(iv) Exemple avec le sous-corps cubique $K$ de $\Q(\mu_7)$ d\'efini par le polyn\^ome
$Q=x^3 + x^2 - 2\,x - 1$, $p=13$,  $n=3$, $\ell^{p-1} \equiv 1 \pmod{p^4}$, $1 < \ell < 10^{12}$.
Pour avoir les $\ell$ d\'ecompos\'es dans $k$, le programme factorise
$Q$ modulo $\ell $ et teste le nombre de facteurs $d$ 
(principe qui permet de varier $Q$ facilement).
\footnotesize
\begin{verbatim}
{p=13; n=3; B=10^12; M=p^(n+1); Q=x^3+x^2-2*x-1; K= bnfinit(Q,1);
ro=znprimroot(M)^(p^n); C0=0; C1=0; C2=0; C3=0; NL=0;
for(k=1, p-1, R=Mod(ro,M)^k; L=component(R,2); while(L<B, L=L+M;
if(isprime(L)==1, QL=x^3+x^2-2*x-Mod(1,L); F=factor(QL);
d=component(matsize(F),1); if(d==3, NL=NL+1;
Su=bnfsunit(K, idealprimedec(K,L)); A=component(component(Su,1),1);
AA =(Mod(A,Q)^(p-1)-1)/p; v=valuation(AA,p);
if(v==0, C0=C0+1); if(v==1, C1=C1+1); if(v==2, C2=C2+1);
if(v==3, C3=C3+1))))); print(p," ",n," ",B);
print(NL," " ,C0," ",C1," ",C2," ",C3);
print(C0/NL+0.0," ",C1/NL+0.0," ",C2/NL+0.0," ",C3/NL+0.0);
print((p^2-1.0)/p^2," ",(p^2-1.0)/p^4," ",
                                 (p^2-1.0)/p^6," ",(p^2-1.0)/p^8)}
\end{verbatim}

\normalsize
Il y a alors $N_L = 5707184$ nombres $\ell$ dans l'intervalle, dont 
$C_0= 5673504 $, $C_1= 33487 $, $C_2= 192 $, $C_3=1$,
sont tels que $\delta_p(\alpha)= 0, 1, 2, 3$.
\footnotesize
$$\begin{array}{llcc}
& \hbox{proportions} &  &\hbox{probabilit\'es}  \vspace{0.1cm}\\ 
 C_0= &  \hspace{-0.2cm} 0.9940986658
 &  \frac{p^2-1}{p^2}  = & \hspace{-0.2cm}  0.9940828402  \\
 C_1=  & \hspace{-0.2cm} 0.0058675171
&  \frac{p^2-1}{p^4}  =& \hspace{-0.2cm}  0.0058821469  \\
 C_2=  & \hspace{-0.2cm} 0.0000336418
&  \frac{p^2-1}{p^6}  =& \hspace{-0.2cm} 0.0000348056  \\
 C_3=  &  \hspace{-0.2cm} 0.0000001752
&  \frac{p^2-1}{p^8}  = & \hspace{-0.2cm}  0.0000002059 \\
\end{array} $$
\normalsize

\subsection{Statistiques sur $m_n$}\label{stat}
Rappelons, pour l'algorithme, que si $x_i = {\rm N}_{k_n/k}(y_i) \in 
\Lambda_i^n \cap {\rm N}_{k_n/k}(k_n^\times)$, 
o\`u $(x_i) =  {\rm N}_{k_n/k}( {\mathfrak B}_i) \in  {\rm N}_{k_n/k}({\mathcal I}_i^n)$, 
la relation ${\mathfrak A}_{i+1}^{1-\sigma_n} \cdot{\mathfrak B}_{i} = (y_i)$ fournit le nouvel
${\mathfrak A}_{i+1}$ dont on calcule ${\rm N}_{k_n/k}({\mathfrak A}_{i+1})$
pour constituer $\Lambda_{i+1}^n$, puis les $\delta_p(x_{i+1})$, etc.
Consid\'erons alors un corps quadratique $k$ et le $n$-i\`eme \'etage de $k_\infty$ pour lequel 
on souhaite tester l'ind\'ependance des $\delta_p(x_i) \geq 0$, $x_i \in \Lambda_i^n$,
obtenus successivement par l'algorithme de calcul de $\Cl_{k_n}$ pour $1 \leq i \leq m_n$ 
et de m\^eme celle des classes des ${\rm N}_{k_n/k}({\mathfrak A}_{i})$ dans $\Cl_k$.

\smallskip
Comme la recherche num\'erique, par programme PARI, des id\'eaux ${\mathfrak A}_i$ 
et des $x_i \in \Lambda_i^n$ tels que $(x_i) =  {\rm N}_{k_n/k}({\mathfrak B}_i)$, 
est particuli\`erement difficile on proc\`ede de fa\c con indirecte par le biais du calcul de
$m_n$ ($n=1$ en pratique) \`a partir de celui de $\order \Cl_{k_n}$.

\smallskip
Pour \'etudier les deux facteurs de $\order (M^n_{i+1}/M^n_i)$ on consid\`ere s\'epar\'ement
les cas $\Cl_k=1$ et $\delta_p(\varepsilon) = 0$ qui permettent des statistiques
respectivement sur~:

\centerline{$\ds \frac{p^{n \cdot (d -1)}}{(\Lambda_i^n : \Lambda_i^n \cap {\rm N}_{k_n/k}(k_n^\times))}$
\ \ et \ \  $\ds \frac{\order \Cl_k}{\order {\rm N}_{k_n/k}(M^n_i)}$.}

\subsubsection{Calcul de $m_n$ dans le cas 
$\Cl_k=1\ \&\ \delta_p(\varepsilon) = 1$, $k$ quadratique} 
On se place dans les conditions pour lesquelles le nombre de classes ambiges 
dans $k_n/k$ est \'egal \`a $p$, donc avec la seule hypoth\`ese~:
$$\ds {(E_k : E_k \cap {\rm N}_{k_n/k} (k_n^\times))} = p^{n-1}, $$ 

qui \'equivaut au fait que $\delta_p(\varepsilon) = 1$.
Il en r\'esulte qu'il est \'equivalent de faire des statistiques sur l'ordre de $\Cl_{k_n}$ 
lorsque $k=\Q( \sqrt m)$ varie, ce que PARI effectue assez rapidement~; 
en effet, on a la filtration correspondante $(M^n_i)_{i \geq 0}$, de $M^n = \Cl_{k_n}$, 
telle que $\order  (M^n_{i+1}/M^n_i) = p$ pour $0 \leq i \leq m_n-1$, et telle que
$\order  \Cl_{k_n} = \prod_{i=0}^{m_n-1}\order  (M^n_{i+1}/M^n_i) = p^{m_n}$.

\smallskip
On a $M^n=M_1^n$ si et seulement si $M_2^n=M_1^n$, ce qui est \'equivalent \`a
$\delta_p(\alpha_1) = 0$, de probabilit\'e $1 - \frac{1}{p}$ et qui ne d\'epend pas du repr\'esentant 
$\alpha_1$ modulo $\langle \varepsilon \rangle$.

\smallskip
Ensuite, on a $M^n=M_2^n$ si et seulement si $M_2^n \ne M_1^n$ et $M_3^n=M_2^n$, 
\'equivalent \`a $\delta_p(\alpha_1) \geq 1$ et $\delta_p(\alpha_2) = 0$, 
de probabilit\'e $(1 - \frac{1}{p}) \cdot \frac{1}{p}$, etc.

\begin{remark}\label{note}
 Noter que puisque ${\rm N}_{k/\Q}(\varepsilon)=1$, la condition $R_k \equiv 0 \pmod p$ est 
\'equivalente \`a $\delta_p(\varepsilon) \geq 1$ et est donc de probabilit\'e $\frac{1}{p}$, 
alors que pour $x \in k^\times$ arbitraire, la condition $\delta_{\mathfrak p}(x) \geq 1$ 
pour tout ${\mathfrak p} \mid p$ est de probabilit\'e $\frac{1}{p^2}$ 
(cf. Lemme \ref{norme1} et Remarque \ref{remunit}).

\smallskip
Plus g\'en\'eralement, si $\alpha = 1+p \cdot p^r \cdot \beta$, $0 \leq r < n$, 
est tel que ${\rm N}_{k/\Q}(\alpha) \equiv 1 \pmod {p^{n+1}}$,
ce qui est le cas si $(\alpha)$ est norme d'un id\'eal de $k_n$,
on a la relation exceptionnelle $\beta + \beta' + p^{r+1} \cdot \beta\,\beta' = u \cdot p^{n-r}$
qui refl\`ete la formule du produit et modifie la probabilit\'e de $p$-divisibilit\'e 
de $\beta$ en $\frac{1}{p}$ au lieu de $\frac{1}{p^2}$ puisque $p \mid \beta$  
\'equivaut \`a $p \mid \beta'$, d'o\`u les probabilit\'es pour les $\delta_p(\alpha_i)$. 

\smallskip
Si $\alpha$ est une unit\'e $\varepsilon$, on a $\beta' +\beta\cdot \varepsilon' = 0$ 
qui rend la propri\'et\'e ci-dessus vraie pour tout $n$.

\smallskip
Les deux programmes suivants ($k=\Q(\sqrt 7)$, $p=3$) justifient \`a nouveau le
ph\'e\-nom\`ene, pour $y \in k^\times$, selon la proximit\'e $p$-adique
de ${\rm N}_{k/\Q}(y)$ avec $1$~:

\medskip
(i) On impose que la norme de l'\'el\'ement al\'eatoire $y$ soit assez proche de 1
(on peut remplacer $9$ par toute puissance de $3$ plus grande)~:
\footnotesize
\begin{verbatim}
{m=7; Q=x^2-m; X=Mod(x,Q); N=10^6; NY=0.0; NY0=0.0; B=10^7; 
for(k=1, B, a=random(N); b=random(N); Y=a*X+b; n=norm(Y); 
if(Mod(n^2, 9)==1, NY=NY+1; Z=(Y^2-1)/3; z=norm(Z); 
if(valuation(z,3)==0, NY0=NY0+1))); print(NY0/NY)}
\end{verbatim}

\normalsize
Densit\'e de $\delta_p(y)=0$ obtenue~: $0.6667709837 \sim \frac{2}{3}$.

\medskip
(ii) Aucune condition de norme locale en dehors de $p$ pour $y$~:
\footnotesize
\begin{verbatim}
{m=7; Q=x^2-m; X=Mod(x, Q); pi1=X-2; pi2=X+2; N=10^5; NY=0.0; NY0=0.0; 
B=10^7; for(k=1, B, a=random(N); b=random(N); Y=a*X+b; n=norm(Y); 
if(Mod(n, 3)!=0, NY=NY+1; Z=(Y^2-1)/3; Z1=Z*pi1; Z2=Z*pi2;  
v1=valuation(component(Z1,2), 3); v2=valuation(component(Z2,2), 3); 
v=min(v1, v2); if(v==0, NY0=NY0+1))); print(NY0/NY)}
\end{verbatim}

\normalsize
Densit\'e de ${\rm min}(\delta_{\mathfrak p}(y), \delta_{\mathfrak p'}(y)) = 0$ 
obtenue~: $0.8889293692 \sim \frac{8}{9}$.
\end{remark}

Dans le cas $n=1$, il suffit de supposer $R_k = p$, donc que $k$ 
est tel que $\delta_p(\varepsilon) \geq 1$.
Autrement dit, $\varepsilon$ est partout norme locale et n'intervient pas dans
les raisonnements.
L'ind\'ependance des $\alpha_i$  est alors mesur\'ee par les probabilit\'es
$\frac{p-1}{p^h}$ d'avoir $\order  \Cl_{k_1} = p^h$, $h \geq 1$ (m\^eme 
raisonnement que dans le cas pr\'ec\'edent). 

\smallskip
Le programme ci-dessous (\'ecrit pour $n=1$) est valable pour $p \geq 3$, 
mais pour $p>3$, le temps de calcul devient prohibitif~; pour $n>1$ remplacer 
$polsubcyclo(p^2, p)$ par $polsubcyclo(p^{n+1}, p^n)$.
\footnotesize
\begin{verbatim}
{p=3; Cyclo=polsubcyclo(p^2, p); C1=0; C2=0; C3=0; C4=0; CM=0; 
b=1; B=3*10^5;  m=b; while(m<b+B, m=m+1;  
if(core(m)==m & kronecker(m, p)==1, Q=x^2-m; M=m; t=Mod(m,4); 
if(t!=1, M=4*m); h=qfbclassno(M); if(valuation(h,p)==0, E=quadunit(M);  
e1=component(E,2); e2=component(E,3); if(t==1, e2=e2/2; e1=e1+e2); 
E=Mod(e1+e2*x, Q); EE=component(E^(p-1)-1,2); 
ve=valuation(EE, p)-1; if(ve>=1, CM=CM+1;  
P=polcompositum(Cyclo, Q); R=component(P,1); K=bnfinit(R, 1); 
H=bnrinit(K,1); G=component(component(H,5),1); w=valuation(G,p); 
if(w==1, C1=C1+1);  if(w==2, C2=C2+1); 
if(w==3, C3=C3+1); if(w>=4, C4=C4+1))))); 
print(CM,"  ",C1,"  ",C2,"  ",C3,"  ",C4); 
print(C1/CM+0.0, "   " ,C2/CM+0.0, "   " , 
C3/CM+0.0,"   ",C4/CM+0.0); S=0.0; for(j=0, 8, S=S+(p^2-1)/p^(8+2*j)); 
print((p^2-1.0)/p^2,"  ",(p^2-1.0)/p^4,"  ",(p^2-1.0)/p^6,"  ",S)}
\end{verbatim}

\normalsize\smallskip
Pour $p=3$ et $B= 3\cdot 10^5$ on obtient les valeurs num\'eriques
$C_M= 18928$, $C_1= 16857$, $C_2= 1814$, $C_3= 221$, $C_4= 36$ et le tableau suivant~:
\footnotesize
$$\begin{array}{cccc}
& \hbox{proportions} & & \hbox{probabilit\'es}   \vspace{0.1cm} \\ 
& \frac{C_1}{C_M} = 0.8905853761
 &\hspace{1cm} \frac{8}{9} &\hspace{-0.3cm} = 0.8888888888 \\
& \frac{C_2}{C_M} = 0.0958368554
&\hspace{1cm}  \frac{8}{9^2} &\hspace{-0.3cm}= 0.0987654320\\
& \frac{C_3}{C_M} =  0.0116758241
&\hspace{1cm} \frac{8}{9^3} & \hspace{-0.3cm}= 0.0109739368 \\
& \frac{C_4}{C_M} =  0.0019019442
&  \sum_{j\geq 0} \frac{8}{9^{4+j}} &\hspace{-0.3cm} = 0.0013717421
\end{array} $$
\normalsize

\begin{remark}\label{nature}
La diff\'erence de nature entre d'une part les r\'esultats num\'eriques obtenus ici sur le 
comportement ``fictif'' des $\delta_p(\alpha_i)$ d\'eduit du calcul effectif 
de $\order \Cl_{k_n}$, et d'autre part les exp\'erimentations du \S\,\ref{L} sur la 
repr\'esentation des classes par des id\'eaux premiers ${\mathfrak L}$ 
d\'ecompos\'es dans $k_n/\Q$, provient des faits suivants~:

\smallskip
Dans l'\'ecriture $\varepsilon =  {\rm N}_{{k_n}/k}(y_1)$
qui conduit \`a $(y_1) = {\mathfrak A}_1^{1-\sigma}$, comme ${\mathfrak A}_1$
est d\'efini modulo les id\'eaux invariants, l'algorithme (non unique) reste valable si l'on prend
${\mathfrak A}_1 = {\mathfrak A}'_1\cdot {\mathfrak P}_1$, o\`u ${\mathfrak A}'_1$
est \'etranger \`a $p$ et o\`u ${\mathfrak P}_1 \in \langle S_{k_n} \rangle$ est arbitraire~; 
on a alors ${\rm N}_{{k_n}/k}({\mathfrak A}_1) = {\rm N}_{{k_n}/k}({\mathfrak A'}_1) \cdot
{\rm N}_{{k_n}/k}({\mathfrak P}_1) = (\alpha_1)$ avec $\alpha_1 = \alpha'_1\cdot \eta_1$,
o\`u $\eta_1$ est une $S_k$-unit\'e arbitraire.

\smallskip
Or $\ds \order \big( M^n_2/M^n_1 \big) = \frac{p^n}{(\Lambda_1 : \Lambda_1 
\cap {\rm N}_{{k_n}/k}({k_n}^\times))}$, o\`u $\Lambda_1 = \langle \varepsilon, \alpha_1 \rangle$,
et la condition $M^n=M^n_1$ a lieu si et seulement si $M^n_2 = M^n_1$, soit
$\delta_p(\alpha_1)=0$~; si la $S_k$-unit\'e $\eta_1$
{\it est non norme dans $k_n/k$}, quel que soit $\alpha'_1$
on peut faire en sorte que $\alpha_1=\alpha'_1 \cdot \eta_1$ soit norme, 
auquel cas on a au contraire $M^n_2 \ne M^n_1$.

\smallskip
Autrement dit, cette statistique, d\'enombrant les $k$ tels que $M^n = M^n_{m_n}$,
$m_n \geq 1$, par l'algorithme utilisant des id\'eaux \'etrangers \`a $p$, \'elimine 
ceux dont les $S_k$-unit\'es sont normes~; ceci se propage pour 
chaque $M^n_i$, mais on \'evite ainsi le calcul des $S_k$-unit\'es
non unit\'es et de leurs symboles de Hasse.

\smallskip
On montre, pour $n=1$, que l'on passe d'une statistique \`a l'autre en multipliant 
par $\frac{p+1}{p^r}$, $r \geq 1$, celle relative aux id\'eaux premiers ${\mathfrak L}$
 totalement d\'ecompos\'es dans $k_1/\Q$~; 
les vraies densit\'es sont celles obtenues via la repr\'esentation ``fictive'' des
classes de $k_1$ par les id\'eaux~${\mathfrak L}$
(densit\'e des corps $k$ tels que $\order \Cl_{k_1}= \order \Cl_{k_1}^{G_1} = p$, 
\'egale \`a~$\frac{p-1}{p}$).
\end{remark}

\subsubsection{Calcul de $m_n$ dans le cas 
$\order \Cl_k=p \  \& \ \delta_p(\varepsilon)=0$, $k$ quadratique}
Lorsque $\delta_p(\varepsilon)=0$, on a
$\omega_n(\varepsilon)$ d'ordre $p^n$, d'o\`u $\omega_n(\Lambda_i^n)=p^n$ 
pour tout $i$ et tout $n$, et on obtient~:
$$\order \big(M^n_{i+1}/M^n_{i} \big) = \ds
\frac{\order \Cl_k}{\order {\rm N}_{k_n/k}( M^n_i)} = \frac{p}{\order {\rm N}_{k_n/k}( M^n_i)}. $$

On se limite au cas $p=3$, $n = 1$ et pour des corps $k$ de nombre de classes $3$.
On a  $\order \Cl_{k_1} = p^w$, $w \geq 1$, puisque $\order \Cl_k=p$, et
le groupe des classes ambiges, d'ordre $p$, est engendr\'e par les classes des id\'eaux
invariants (cf. Relation \eqref{amb}), donc par la classe de $({\mathfrak a})$ (\'etendu \`a $k_1$ 
de ${\mathfrak a}$ tel que $\cl_k({\mathfrak a}) \ne 1$), et celle de
${\mathfrak p}_1 \mid {\mathfrak p}$ dans $k_1$.

\smallskip
On a alors ${\rm N}_{k_1/k}(M^1_1) = \langle \cl_k({\mathfrak a}^p), \cl_k({\mathfrak p}) \rangle
= \langle \cl_k({\mathfrak p}) \rangle$ d'ordre $p$ (i.e., $w=1$) si et seulement si ${\mathfrak p}$ 
est non principal.

\smallskip
Le programme fournit $C_r := \order \{k = \Q(\sqrt m), \ w=r\}$, pour les $\Q(\sqrt m)$ 
v\'erifiant les hypoth\`eses pr\'ec\'edentes et pour un intervalle donn\'e pour $m$.
On pourrait aussi prendre des $m$ al\'eatoires.
\footnotesize
\begin{verbatim}
{p=3; n=1; b=10^3; B=10^6; Cyclo=polsubcyclo(p^(n+1),p^n); 
C1=0; C2=0; C3=0; C4=0; Cm=0.0; m=b; 
while(m<B, m=m+1; if(core(m)==m & kronecker(m,p)==1, Q=x^2-m; 
M=m; t=Mod(m,4); if(t!=1, M=4*m); h=qfbclassno(M); 
if(valuation(h,p)==1, E=quadunit(M); 
e1=component(E,2); e2=component(E,3); 
if(t==1, e2=e2/2; e1=e1+e2); E=Mod(e1+e2*x, Q); 
E=component(E^(p-1)-1,2); 
ve=valuation(E,p)-1; if(ve==0, Cm=Cm+1; 
P=polcompositum(Cyclo,Q); 
R=component(P,1); K=bnfinit(R,1); H=bnrinit(K,1); 
H=component(component(H,5),1); w=valuation(H,p); 
if(w==1, C1=C1+1); if(w==2, C2=C2+1); 
if(w==3, C3=C3+1); if(w>=4, C4=C4+1))))); 
print(C1/Cm,"  ",C2/Cm,"  ",C3/Cm,"  ",C4/Cm)}
\end{verbatim}

\normalsize
On obtient les r\'esultats num\'eriques suivants pour $p=3$, $10^3< m <10^6$:
$C_M= 8508$, $C_1= 6362$, $C_2= 1944$, $C_3= 178$, $C_4=24$, et le tableau~:
\footnotesize
$$\begin{array}{cccc}
& \hbox{proportions} & & \hbox{probabilit\'es} \vspace{0.1cm} \\ 
& \frac{C_1}{C_M} = 0.7477668077
 &\hspace{1cm} \frac{8}{11} & \hspace{-0.3cm} = 0.72727272 \\
& \frac{C_2}{C_M} = 0.2284908321
&\hspace{1cm} \frac{8}{3.11} &\hspace{-0.3cm} = 0.24242424\\
& \frac{C_3}{C_M} = 0.0209214856
&\hspace{1cm} \frac{8}{3^3. 11} & \hspace{-0.3cm}= 0.02693602 \\
& \frac{C_4}{C_M} = 0.0028208744
& \frac{8}{11}\sum_{j\geq 0} \frac{p-1}{p^{5+2 j}} & \hspace{-0.3cm}= 0.00336702
\end{array} $$
\normalsize

La probabilit\'e pour ${\mathfrak p} \mid p$ d'\^etre non principal (i.e., $r=1$) est 
difficile \`a \'etablir car contrairement aux heuristiques g\'en\'erales,
$\cl_k({\mathfrak p})$ n'est pas ``al\'eatoire''.
Un programme ind\'ependant testant la principalit\'e de ${\mathfrak p}$ 
(sans hypoth\`ese sur $\varepsilon$ mais pour des $m$ beaucoup 
plus grands que dans le programme ci-dessus) donne une proportion de l'ordre de 
$O(1)\cdot \frac{p^2-1}{p^2+p-1}$. Ensuite on aurait des probabilit\'es en
$O(1)\cdot \frac{p^2-1}{p^2+p-1}\cdot \frac{1}{p^{2 \cdot (r-1)-1}}$ pour tout 
$r\geq 2$.

\section{Conclusion} \label{conclusion}\label{sect9}
Cette \'etude (qui a, au moins pour la premi\`ere partie, des points g\'en\'eralisant
les approches de \cite{F1}, \cite{FK}, \cite{FT1}, \cite{FT2}, 
\cite{Gre}, \cite{I}, \cite{IS1}, \cite{IS2}, \cite{Su}, \cite{T1}, \cite{T2}) montre que 
la question de la conjecture de Greenberg (pour les corps totalement r\'eels) est 
essentiellement $p$-adique et sans doute moins ``th\'eorie d'Iwasawa alg\'ebrique'' 
qu'admis g\'en\'eralement. 

\smallskip
En outre, il s'agit probablement d'une question li\'ee au seul invariant $\lambda$ 
dans la mesure o\`u la nullit\'e de $\mu$ peut \^etre admise plus facilement 
pour les corps totalement r\'eels (voire d\'emontr\'ee, comme dans le cas Ab\'elien) 
car, d'apr\`es l'\'etude de \cite{HM} dans un cadre tr\`es g\'en\'eral~:

\smallskip
{\it Both in Iwasawa's work, and in the present one, the size of the 
$\mu$-invariant appears to be intimately related to the existence 
of primes that split completely in the tower. }

\smallskip
Le cas $\mu \ne 0$ serait donc li\'e \`a l'existence d'id\'eaux premiers totalement 
d\'ecompos\'es dans la $\Z_p$-extension de $k$ consid\'er\'ee
(premiers exemples de telles $\Z_p$-extensions d\^us, entre autres, \`a Iwasawa \cite{Iw}, 
Serre \cite[\S\,4.5]{Se3}, Cuoco \cite{Cu},  Hubbard--Washington \cite{HW}), ce qui est 
\'evidemment impossible pour la $\Z_p$-extension cyclotomique 
(cf. Section \ref{L} sur l'utilisation des id\'eaux premiers pour 
repr\'esenter les classes dans $k_\infty$).

\smallskip
Le cas des corps $k$ Ab\'eliens r\'eels peut \^etre de nature plus particuli\`ere 
en raison des ``formules analytiques $p$-adiques'' du nombre de classes introduisant 
les unit\'es cyclotomiques et conduisant, avec des hypoth\`eses ``ad hoc'', 
\`a de nombreux travaux sp\'ecifiques 
(voir \cite{F2}, \cite{BaN}, \cite{BeN}, \cite{KS}, \cite{Ng1}, \cite{Ng2} 
et leurs r\'ef\'erences), bien que nous pensions que la version fonctions $L$ 
$p$-adiques classique (au sens de \cite{Sin}, \cite{W}) ne soit 
qu'une traduction analytique de l'aspect ``modules sur l'alg\`ebre d'Iwasawa'', et 
n'apporte pas d'information suppl\'ementaire au niveau ``diophantien $p$-adique''. 

\smallskip
Si le passage \`a la limite (alg\'ebrique ou analytique 
$p$-adique) est au demeurant plus concis et structurant que les calculs 
aux niveaux finis, ceux-ci sont n\'ecessaires pour localiser les profonds 
ph\'enom\`enes arithm\'etiques sous-jacents.
En effet, du seul point de vue th\'eorie d'Iwasawa, le probl\`eme porte sur la d\'etermination
du quotient de Herbrand~:
$$q({\mathcal X}_k) := \order  ({\mathcal X}_k^G)/ \order ({}^{G}{\!\mathcal X}_k)$$

de ${\mathcal X}_k :=\ds \varprojlim_{n\to\infty} \Cl_{k_n}$ (pour la norme arithm\'etique),
o\`u ${\mathcal X}_k^G$ (resp. ${}^{G}{\!\mathcal X}_k$) est le noyau (resp. le conoyau)
de l'op\'eration de $1-\sigma$ sur ${\mathcal X}_k$, o\`u $\sigma$ est un g\'en\'erateur 
topologique de $G := {\rm Gal}(k_\infty/k)$ (voir \cite[\S\,1.1]{J6} et \cite[\S\,3]{Ng3}
pour quelques rappels sous des approches diff\'erentes). La pseudo-nullit\'e de
${\mathcal X}_k$ (i.e., $\lambda=\mu=0$) est \'equivalente \`a $q({\mathcal X}_k)=1$.

\smallskip
Si la d\'etermination de ${}^{G}{\!\mathcal X}_k$ \'equivaut grosso modo aux r\'esultats 
du \S\,\ref{torsion}, c'est-\`a-dire \`a la th\'eorie du corps de classes global,
la d\'etermination de ${\mathcal X}_k^G$ semble non triviale et 
proba\-blement li\'ee aux consid\'erations $p$-adiques pr\'ec\'edentes o\`u l'on rencontre 
manifestement des probl\`emes de type ``quotients de Fermat'' de 
nombres alg\'ebriques dont les heuristiques impliquent que les conjectures 
\'enonc\'ees sont tr\`es raisonnables. 

\medskip
Pour conclure, nous nous proposons de faire quelques remarques sur les groupes 
${\mathcal T}_{k_n}$ sous la conjecture de Leopoldt pour $p$ dans $k_\infty$, 
quelle que soit la d\'ecomposition de $p$ dans le corps totalement r\'eel $k$, 
ce qui nous para\^it plus canonique en raison de la sp\'ecificit\'e de ces 
groupes de torsion associ\'es \`a la $p$-ramification Ab\'elienne et plus 
susceptibles d'une approche essentiellement $p$-adique de la 
conjecture de Greenberg pour laquelle il serait utile de tester 
num\'eriquement l'Hypoth\`ese \ref{hypo}.

\medskip
Pour tout $n \gg 0$, on a encore $\order  {\mathcal T}_{k_n} = 
\order  \Cl_{k_n} \cdot \,\order  (U_{k_n}^* / \overline E_{k_n})$,
avec des notations analogues \`a celles du \S\,\ref{torsion}, o\`u
les $\Z_p$-modules $U_{k_n}^*$ et $\overline E_{k_n}$ 
sont $\Z_p$-libres de $\Z_p$-rangs $d \cdot p^n -1$.

\smallskip
On peut \'egalement noter $R_{k_n} := \order (U_{k_n}^*/\overline E_{k_n})$
le r\'egulateur convenablement normalis\'e de $k_n$.

\smallskip
Il resterait \`a \'etudier la formule d'Iwasawa $\order  {\mathcal T}_{k_{n}} =: 
p^{\wt \lambda \cdot n + \wt \nu}$ (en supposant pour simplifier
que $\wt \mu = 0$) telle que
$p^{\wt \lambda} =\ds \frac{\order  {\mathcal T}_{k_{n+1}}}{\order  {\mathcal T}_{k_{n}}}$
pour tout $n \gg 0$~; sous la conjecture de Greenberg pour $k$, on a 
$\order  \Cl_{k_{n+1}} = \order  \Cl_{k_{n}} = p^{\nu}$, pour tout $n \gg 0$, auquel cas 
$p^{\wt \lambda} = \ds \frac{R_{k_{n+1}}}{R_{k_{n}}}$.

On a $\wt \lambda = 0$ si et seulement si ${\mathcal T}_k = 1$ 
(i.e., $k$ est $p$-rationnel) pour les raisons suivantes~:
les ``transferts'' $j_{k_{n+h}/k_n}^{} : {\mathcal T}_{k_n} \too 
{\mathcal T}_{k_{n+h}}$ sont injectifs pour tous $n, h$ 
(en raison de la conjecture de Leopoldt pour $p$ dans la tour) \cite[Theorem IV.2.1]{Gra1}~;
on a en particulier la formule de points fixes ($k_{n+h}/k_n$ \'etant trivialement 
{\it $p$-primitivement ramifi\'ee}, \cite[Theorem IV.3.3]{Gra1}),
${\mathcal T}_{k_{n+h}}^{{\rm Gal}(k_{n+h}/k_n) } =
j_{k_{n+h}/k_n}^{}({\mathcal T}_{k_n}) \simeq {\mathcal T}_{k_n}$.

Consid\'erons alors $\nu_{k_{n+1}/k_n}^{} := 
j_{k_{n+1}/k_n}^{} \circ {\rm N}_{k_{n+1}/k_n}$~; si $\wt \lambda = 0$, on a
${\mathcal T}_{k_{n+1}} = j_{k_{n+1}/k_n}^{} ({\mathcal T}_{k_n})$ et
puisque l'on a ${\rm N}_{k_{n+1}/k_n} \circ   j_{k_{n+1}/k_n}^{} = p$, 
on obtient, \`a partir de l'\'egalit\'e pr\'ec\'edente
(la norme arithm\'etique est surjective, cf. Sch\'ema du \S\,\ref{resglob}), 
${\mathcal T}_{k_n} = {\mathcal T}_{k_n}^{p}$,
d'o\`u ${\mathcal T}_{k_n} = 1$, et comme ${\mathcal T}_{k_n}^{G_n} 
\simeq {\mathcal T}_{k}$, on a ${\mathcal T}_{k}=1$. R\'eciproque \'evidente.
On notera que $\wt \lambda = 0$ implique $\wt \nu = 0$.

\smallskip
Ainsi $\wt \lambda = 0$ est \'equivalent \`a $\Cl_{k_n} = R_{k_n}=1$ pour tout $n$
($p$-rationalit\'e dans la tour), ce qui peut sugg\'erer que $\wt \lambda$ 
est davantage accessible puisque $\order {\mathcal T}_{k_n}$ est essentiellement donn\'e
par le r\'esidu de la fonction z\^eta $p$-adique de $k_n$ (cf. \cite{Co}, \cite{Se2}, \cite{W}), 
ce qui permet, comme dans \cite{CN}, \cite{IS1}, \cite{IS2}, d'en d\'eduire des cas de nullit\'e de 
$\lambda$ lorsque $\wt \lambda \geq 1$.

On remarque que $\ds \frac{\order {\mathcal T}_{k_n}}{\order {\mathcal T}_{k}} =
\order ({\mathcal T}_{k_n} / {\mathcal T}_{k_n}^{G_n})$ et que le calcul
de $ ({\mathcal T}_{k_n} / {\mathcal T}_{k_n}^{G_n})^{G_n}$
serait le second pas de l'algorithme d\'efinissant la filtration habituelle 
dans le cadre diff\'erent de la $p$-ramification Ab\'elienne qui, \`a notre 
connaissance, n'a pas \'et\'e \'etudi\'e.

\subsection*{Remerciements}
{Je remercie J-F. Jaulent, T. Nguyen Quang Do et C. Maire pour plusieurs 
\'echanges, commentaires et indications (techniques et bibliographiques).}

\end{document}